\documentclass[a4paper]{amsart}

\usepackage{amsmath, amsfonts, amsthm , amssymb, amscd}
\usepackage{graphicx}
\usepackage{pst-grad} 

\usepackage{url}
\usepackage{pstricks}
\usepackage[utf8]{inputenc}
\usepackage{pstricks-add}
\usepackage{pgf}

\usepackage{tikz}
\usetikzlibrary{decorations.markings}
\usetikzlibrary{shapes.geometric,positioning}
\usepackage{stmaryrd}
\usepackage{mathalfa}
\usepackage{euscript}
\usepackage{eufrak}
\usepackage{extarrows}
\usepackage[normalem]{ulem}
\usepackage{mathtools}

\usepackage{proof}
\newcommand{\col}[1]{\mid\!\!\cdot\!\!\cdot\!\!\cdot\: {\it #1}}
\newcommand{\cols}[1]{{\it col}(#1)}
\newcommand{\colsp}[1]{{\it col}^+(#1)}
\newcommand{\colsm}[1]{{\it col}^-(#1)}

 \newtheorem{thm}{Theorem}[section]
 \newtheorem{prop}[thm]{Proposition}
 \newtheorem{lem}[thm]{Lemma}
 
 \newtheorem{cor}[thm]{Corollary}
\theoremstyle{definition}
 \newtheorem{rem}[thm]{Remark}
 \newtheorem{definition}{Definition}

\numberwithin{equation}{section}
\newtheorem{example}{\sc Example}

\newcommand\mycom[2]{\genfrac{}{}{0pt}{}{#1}{#2}}
\newcommand{\HDS}{\vrule width0pt height2.3ex depth1.05ex\displaystyle}
\newcommand{\f}[2]{{\frac{\HDS #1}{\HDS #2}}}
\newcommand{\afrac}[1]{\mycom{\phantom{\HDS #1}}{\HDS #1}}

\def\d#1{{#1\kern-0.4em\char"16\kern-0.1em}}
\def\D#1{{\raise0.2ex\hbox{-}\kern-0.4em#1}}
\def \Dj{\mbox{\raise0.3ex\hbox{-}\kern-0.4em D}}
\definecolor{britishracinggreen}{rgb}{0.0, 0.26, 0.15}

\def\diskon{\begin{picture}(8,8)
\put(4,3){\makebox(0,0){$\vee$}}
\put(4,3){\makebox(0,0){$\wedge$}}
\end{picture}}

\title{Proofs and surfaces}
\author[Barali\' c]{Djordje Barali\' c}
\address{\scriptsize{Mathematical Institute SANU\\ Knez Mihailova 36, p.f.\ 367\\ 11001 Belgrade, Serbia}}
\email{djbaralic@mi.sanu.ac.rs}
\author[Curien]{Pierre-Louis Curien}
\address{\scriptsize{IRIF $\pi r^2$, Univ.\ Paris Diderot, INRIA and CNRS}}
\email{curien@irif.fr}
\author[Mili\' cevi\' c]{Marina Mili\' cevi\' c}
\address{\scriptsize{Production and Management Faculty, Trebinje, Bosnia and Herzegovina}}
\email{marina.milicevic@fpm.ues.rs.ba}
\author[Obradovi\' c]{Jovana Obradovi\' c}
\address{\scriptsize{Institute of Mathematics AS CR, \v Zitn\' a 25
CZ - 115 67 Praha 1, Czech Republic}}
\email{obradovic@math.cas.cz}
\author[Petri\' c]{Zoran Petri\' c}
\address{\scriptsize{Mathematical Institute SANU\\ Knez Mihailova 36, p.f.\ 367\\ 11001 Belgrade, Serbia}}
\email{zpetric@mi.sanu.ac.rs}
\author[Zeki\' c]{Mladen Zeki\' c}
\address{\scriptsize{Mathematical Institute SANU\\ Knez Mihailova 36, p.f.\ 367\\ 11001 Belgrade, Serbia}}
\email{mzekic@mi.sanu.ac.rs}
\author[\v Zivaljevi\' c]{Rade T.\ \v Zivaljevi\' c}
\address{\scriptsize{Mathematical Institute SANU\\ Knez Mihailova 36, p.f.\ 367\\ 11001 Belgrade, Serbia}}
\email{rade@mi.sanu.ac.rs}

\date{}

\begin{document}

\begin{abstract}
A formal sequent system dealing with Menelaus' configurations is introduced in this paper. The axiomatic sequents of the system stem from 2-cycles of $\Delta$-complexes. The Euclidean and projective interpretations of the sequents are defined and a soundness result is proved. This system is decidable and its provable sequents deliver incidence results. A cyclic operad structure tied to this system is presented by generators and relations.

\vspace{.3cm}

\noindent {\small {\it Mathematics Subject Classification} ({\it
        2010}): 03F03, 18D50, 18G30, 51A05, 51A20}

\vspace{.5ex}

\noindent {\small {\it Keywords$\,$}: Ceva-Menelaus proof, connected sum, cyclic operad, decidability, $\Delta$-complex, incidence theorem, sequent system, simplicial homology}
\end{abstract}

\maketitle

\section{Introduction}
As a part of his program of mechanical theorem proving in projective geometry, J\" urgen Richter-Gebert, partly in collaboration with Susanne Jasmin Apel, investigated so-called Ceva-Menelaus proofs of incidence theorems (see \cite{RG06}, \cite{RG11}, \cite{ARG11} and \cite{A13}). In the paper \cite{RG06}, he gave a proof-theoretical analysis of the method, which served as the starting point for our investigation. For the sake of clarity, we restrict ourselves only to the Menelaus proofs in this paper.

Roughly speaking, a Ceva-Menelaus proof transforms a triangulation of a surface into an incidence result in projective geometry. At the first glance, this is something that connects geometry with geometry. However, the triangulation in question could be envisaged purely combinatorially through a type of abstract cell complexes, like we do in this paper. Even one step further is possible: triangulations of surfaces may be considered as a special kind of syntax built out of two symbols (a dot and a dash) written on various writing pads in the way explained below.

It is convenient to think about syntax as something built out of
primitive symbols combined together in words and written on a
piece of paper or a blackboard. These writing pads, for the sake
of uniformity, could be taken as parts of the two-dimensional
sphere. The same writing pads could be used in proof theory for
trees, which are not linear syntactical forms. The syntax
appropriate for Ceva-Menelaus proofs requires some other surfaces,
not just the sphere. For example, one can start with a piece of
paper (or some other material) in the shape of a torus and produce
``words'' consisting of dots and line segments. A word is
considered to be correct if it triangulates the torus. For
example, the following word consisting of three dots and nine line
segments is correct.

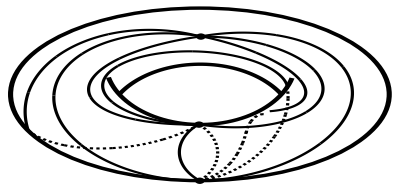
\psset{xunit=1.0cm,yunit=1.0cm,algebraic=true,dimen=middle,dotstyle=o,dotsize=3pt 0,linewidth=0.8pt,arrowsize=3pt 2,arrowinset=0.25}
\begin{center}
\psscalebox{.6 .4} 
{
\begin{pspicture*}(-4.2,-3.2)(4.69,3.2)
\rput{0}(0,0){\psellipse[linewidth=3.2pt](0,0)(4.15,2.86)}
\parametricplot[linewidth=3.2pt]{3.3936126414032994}{6.021215152704372}{1*2.08*cos(t)+0*2.08*sin(t)+0.01|0*2.08*cos(t)+1*2.08*sin(t)+1.08}
\parametricplot[linewidth=3.2pt]{0.545463398049105}{2.5961292555406885}{1*2.08*cos(t)+0*2.08*sin(t)+0.01|0*2.08*cos(t)+1*2.08*sin(t)+-1.08}
\parametricplot[linewidth=2pt]{2.279976774261442}{4.024668457434242}{1*1.22*cos(t)+0*1.22*sin(t)+0.77|0*1.22*cos(t)+1*1.22*sin(t)+-1.92}
\parametricplot[linewidth=2pt,linestyle=dashed,dash=2pt 2pt]{-0.7960284050386095}{0.8174883295547072}{1*1.29*cos(t)+0*1.29*sin(t)+-0.9|0*1.29*cos(t)+1*1.29*sin(t)+-1.94}
\parametricplot[linewidth=2pt]{4.876359128662692}{5.921249764843227}{-0.93*2.85*cos(t)+-0.37*2.37*sin(t)+-0.41|0.37*2.85*cos(t)+-0.93*2.37*sin(t)+-0.42}
\parametricplot[linewidth=2pt]{-0.36193554233635883}{0.28801900984834206}{-0.93*2.85*cos(t)+-0.37*2.37*sin(t)+-0.41|0.37*2.85*cos(t)+-0.93*2.37*sin(t)+-0.42}
\parametricplot[linewidth=2pt]{0.28801900984834206}{0.9867712530003544}{-0.93*2.85*cos(t)+-0.37*2.37*sin(t)+-0.41|0.37*2.85*cos(t)+-0.93*2.37*sin(t)+-0.42}
\parametricplot[linewidth=2pt]{0.9867712530003544}{2.035888472841568}{-0.93*2.85*cos(t)+-0.37*2.37*sin(t)+-0.41|0.37*2.85*cos(t)+-0.93*2.37*sin(t)+-0.42}
\parametricplot[linewidth=2pt]{4.796287866793456}{5.513370746979559}{-0.99*2.96*cos(t)+0.16*1.46*sin(t)+0.5|-0.16*2.96*cos(t)+-0.99*1.46*sin(t)+0.52}
\parametricplot[linewidth=2pt]{-0.769814560200027}{0.09751614892685402}{-0.99*2.96*cos(t)+0.16*1.46*sin(t)+0.5|-0.16*2.96*cos(t)+-0.99*1.46*sin(t)+0.52}
\parametricplot[linewidth=2pt]{0.09751614892685402}{0.7823645134637963}{-0.99*2.96*cos(t)+0.16*1.46*sin(t)+0.5|-0.16*2.96*cos(t)+-0.99*1.46*sin(t)+0.52}
\parametricplot[linewidth=2pt]{0.7823645134637963}{1.3099205773499074}{-0.99*2.96*cos(t)+0.16*1.46*sin(t)+0.5|-0.16*2.96*cos(t)+-0.99*1.46*sin(t)+0.52}
\parametricplot[linewidth=2pt]{5.101666532076473}{5.812100631499277}{-0.73*2.99*cos(t)+-0.68*2.81*sin(t)+-0.93|0.68*2.99*cos(t)+-0.73*2.81*sin(t)+-0.75}
\parametricplot[linewidth=2pt]{-0.4710846756803093}{0.01597084132437098}{-0.73*2.99*cos(t)+-0.68*2.81*sin(t)+-0.93|0.68*2.99*cos(t)+-0.73*2.81*sin(t)+-0.75}
\parametricplot[linewidth=2pt]{0.01597084132437098}{0.4930772245259334}{-0.73*2.99*cos(t)+-0.68*2.81*sin(t)+-0.93|0.68*2.99*cos(t)+-0.73*2.81*sin(t)+-0.75}
\parametricplot[linewidth=2pt]{0.4930772245259334}{0.9378310323419501}{-0.73*2.99*cos(t)+-0.68*2.81*sin(t)+-0.93|0.68*2.99*cos(t)+-0.73*2.81*sin(t)+-0.75}
\parametricplot[linewidth=2pt,linestyle=dashed,dash=2pt 2pt]{0.29980171239023573}{0.8026092272871177}{-0.96*2.4*cos(t)+0.28*1.62*sin(t)+-1.7|-0.28*2.4*cos(t)+-0.96*1.62*sin(t)+-0.1}
\parametricplot[linewidth=2pt,linestyle=dashed,dash=2pt 2pt]{0.8026092272871177}{1.2856576751432676}{-0.96*2.4*cos(t)+0.28*1.62*sin(t)+-1.7|-0.28*2.4*cos(t)+-0.96*1.62*sin(t)+-0.1}
\parametricplot[linewidth=2pt,linestyle=dashed,dash=2pt 2pt]{1.2856576751432676}{1.764156086743084}{-0.96*2.4*cos(t)+0.28*1.62*sin(t)+-1.7|-0.28*2.4*cos(t)+-0.96*1.62*sin(t)+-0.1}
\parametricplot[linewidth=2pt,linestyle=dashed,dash=2pt 2pt]{1.764156086743084}{2.178177947197833}{-0.96*2.4*cos(t)+0.28*1.62*sin(t)+-1.7|-0.28*2.4*cos(t)+-0.96*1.62*sin(t)+-0.1}
\parametricplot[linewidth=2pt]{4.3610629972392605}{5.012668775187195}{-0.98*2.29*cos(t)+-0.17*1.48*sin(t)+0.44|0.17*2.29*cos(t)+-0.98*1.48*sin(t)+0.41}
\parametricplot[linewidth=2pt]{3.6701867790263365}{4.3610629972392605}{-0.98*2.29*cos(t)+-0.17*1.48*sin(t)+0.44|0.17*2.29*cos(t)+-0.98*1.48*sin(t)+0.41}
\parametricplot[linewidth=2pt]{2.6745386261246695}{3.6701867790263365}{-0.98*2.29*cos(t)+-0.17*1.48*sin(t)+0.44|0.17*2.29*cos(t)+-0.98*1.48*sin(t)+0.41}
\parametricplot[linewidth=2pt]{1.4802438664873026}{2.6745386261246695}{-0.98*2.29*cos(t)+-0.17*1.48*sin(t)+0.44|0.17*2.29*cos(t)+-0.98*1.48*sin(t)+0.41}
\parametricplot[linewidth=2pt]{3.952050680584713}{4.532538608784681}{-0.92*3.03*cos(t)+0.39*2.34*sin(t)+0.41|-0.39*3.03*cos(t)+-0.92*2.34*sin(t)+-0.41}
\parametricplot[linewidth=2pt]{3.2064712294409152}{3.952050680584713}{-0.92*3.03*cos(t)+0.39*2.34*sin(t)+0.41|-0.39*3.03*cos(t)+-0.92*2.34*sin(t)+-0.41}
\parametricplot[linewidth=2pt]{2.1186643232226268}{3.2064712294409152}{-0.92*3.03*cos(t)+0.39*2.34*sin(t)+0.41|-0.39*3.03*cos(t)+-0.92*2.34*sin(t)+-0.41}
\parametricplot[linewidth=2pt]{1.1163143309152501}{2.1186643232226268}{-0.92*3.03*cos(t)+0.39*2.34*sin(t)+0.41|-0.39*3.03*cos(t)+-0.92*2.34*sin(t)+-0.41}
\parametricplot[linewidth=2pt]{2.458508275316489}{3.66283004762946}{-0.93*2.38*cos(t)+-0.38*1.07*sin(t)+0.07|0.38*2.38*cos(t)+-0.93*1.07*sin(t)+0.77}
\parametricplot[linewidth=2pt]{3.66283004762946}{4.353411313885404}{-0.93*2.38*cos(t)+-0.38*1.07*sin(t)+0.07|0.38*2.38*cos(t)+-0.93*1.07*sin(t)+0.77}
\parametricplot[linewidth=2pt]{4.353411313885404}{4.91451518639337}{-0.93*2.38*cos(t)+-0.38*1.07*sin(t)+0.07|0.38*2.38*cos(t)+-0.93*1.07*sin(t)+0.77}

\parametricplot[linewidth=2pt,linestyle=dashed,dash=2pt 2pt]{3.1877435668202536}{4.224426726437027}{-0.2*0.55*cos(t)+0.98*0.27*sin(t)+1.3|-0.98*0.55*cos(t)+-0.2*0.27*sin(t)+-1.2}

\parametricplot[linewidth=2pt,linestyle=dashed,dash=2pt 2pt]{4.224426726437027}{4.962155189093001}{-0.2*0.55*cos(t)+0.98*0.27*sin(t)+1.3|-0.98*0.55*cos(t)+-0.2*0.27*sin(t)+-1.2}

\parametricplot[linewidth=2pt,linestyle=dashed,dash=2pt 2pt]{3.365534736074497}{3.4091646125566917}{-0.43415047669753043*28.79714620222122*cos(t)+-0.9008403652053492*5.674955749185209*sin(t)+-12.39|
0.9008403652053492*28.79714620222122*cos(t)+-0.43415047669753043*5.674955749185209*sin(t)+23.13}

\parametricplot[linewidth=2pt,linestyle=dashed,dash=2pt 2pt]{3.317991240386494}{3.365534736074497}{-0.43415047669753043*28.79714620222122*cos(t)+-0.9008403652053492*5.674955749185209*sin(t)+-12.39|
0.9008403652053492*28.79714620222122*cos(t)+-0.43415047669753043*5.674955749185209*sin(t)+23.13}

\parametricplot[linewidth=2pt,linestyle=dashed,dash=2pt 2pt]{3.2675312276066593}{3.317991240386494}{-0.43415047669753043*28.79714620222122*cos(t)+-0.9008403652053492*5.674955749185209*sin(t)+-12.39|
0.9008403652053492*28.79714620222122*cos(t)+-0.43415047669753043*5.674955749185209*sin(t)+23.13}
\parametricplot[linewidth=2pt,linestyle=dashed,dash=2pt 2pt]{3.1193320943659657}{3.2675312276066593}{-0.43415047669753043*28.79714620222122*cos(t)+-0.9008403652053492*5.674955749185209*sin(t)+-12.39|
0.9008403652053492*28.79714620222122*cos(t)+-0.43415047669753043*5.674955749185209*sin(t)+23.13}

\parametricplot[linewidth=2pt]{3.2714122191065176}{4.775270846461568}{-1*2.03*cos(t)+-0.06*1.21*sin(t)+-0.1|0.06*2.03*cos(t)+-1*1.21*sin(t)+0.21}
\parametricplot[linewidth=2pt]{4.775270846461568}{5.736341849300105}{-1*2.03*cos(t)+-0.06*1.21*sin(t)+-0.1|0.06*2.03*cos(t)+-1*1.21*sin(t)+0.21}
\parametricplot[linewidth=2pt]{-0.546843457879481}{0.49921065979275303}{-1*2.03*cos(t)+-0.06*1.21*sin(t)+-0.1|0.06*2.03*cos(t)+-1*1.21*sin(t)+0.21}
\parametricplot[linewidth=2pt]{0.49921065979275303}{1.6476451344447252}{-1*2.03*cos(t)+-0.06*1.21*sin(t)+-0.1|0.06*2.03*cos(t)+-1*1.21*sin(t)+0.21}
\parametricplot[linewidth=2pt,linestyle=dashed,dash=2pt 2pt]{0.47503988185286944}{1.5219882611117619}{-0.39*2.37*cos(t)+0.92*1.72*sin(t)+0.1|-0.92*2.37*cos(t)+-0.39*1.72*sin(t)+-0.62}
\parametricplot[linewidth=2pt,linestyle=dashed,dash=2pt 2pt]{1.5219882611117619}{1.8448740078507204}{-0.39*2.37*cos(t)+0.92*1.72*sin(t)+0.1|-0.92*2.37*cos(t)+-0.39*1.72*sin(t)+-0.62}
\parametricplot[linewidth=2pt,linestyle=dashed,dash=2pt 2pt]{1.8448740078507204}{2.0832626535555896}{-0.39*2.37*cos(t)+0.92*1.72*sin(t)+0.1|-0.92*2.37*cos(t)+-0.39*1.72*sin(t)+-0.62}
\parametricplot[linewidth=2pt,linestyle=dashed,dash=2pt 2pt]{2.0832626535555896}{2.2608210007427445}{-0.39*2.37*cos(t)+0.92*1.72*sin(t)+0.1|-0.92*2.37*cos(t)+-0.39*1.72*sin(t)+-0.62}
\psdots[dotsize=6pt 0,dotstyle=*](0,-2.86)
\psdots[dotsize=6pt 0,dotstyle=*](-0.02,-1)
\psdots[dotsize=6pt 0,dotstyle=*](0.02,1.92)
\end{pspicture*}}
\end{center}

However, it is difficult to ``read'' directly such a syntax (we
will see, in Section~\ref{reading}, that the above word could be
read as a proof of Pappus theorem). That is the main reason for us
to present our results within a sequent system, which is, at least
from the proof-theoretical standpoint, a more convenient syntax.
The axiomatic sequents of this system are obtained by translating,
via simplicial homology, the triangulations of surfaces. The
required interpretation forces, quite naturally, one-sided
sequents (the formulae of a sequent are placed at one side of the
symbol $\vdash$). Besides the structural rule of cut, which is
implicit in \cite{RG06}, we consider two propositional connectives
and their rules of introduction. Also, an action of the octahedral
group on the set of atomic formulae, which stems from the
possibility to organise a set of six points into twenty four
Menelaus' configurations, is included in our system.

It is proved that the system is sound with respect to both Euclidean and projective interpretation. This system is also decidable. An analysis of possible generalizations of the method, which relies on a constructive solution of Steenrod's problem in dimension 2, shows that there is no way of extending the system by means of homological arguments. Also, a few examples of incidence results delivered from provable sequents of our system are given.

As it is natural to relate the single-conclusion sequent systems with multicategories and operads, this one-sided sequent system gives rise to a cyclic operad, which may serve as an initial framework for a general proof-theoretical study of the matters. This cyclic operad is based on the operation of connected sum on the abstract cell complexes that we use to generate the axiomatic sequents. A presentation of this operad by generators and relations is given in the last section of the paper.

\section{The Menelaus theorem}
By the Menelaus theorem we mean here a lemma used for the Sector
theorem of spherical trigonometry, which appears in
Al-Haraw\={\i}'s version of Menelaus' \emph{Spherics}. The Greek
original of the Menelaus text written at the end of the 1st
century A.D.\ is lost and the Arabic version mentioned above is
from the 9th or 10th century A.D. For historical remarks see
\cite{RP14} and \cite{SK14}. The same lemma appears in Ptolemy's
\emph{Almagest} (see \cite[Book I.13]{T84}).

For three mutually distinct collinear points $X$, $Y$ and $Z$ in the Euclidean plane $\mathbf{R}^2$ let
\[
(X,Y;Z)=_{df} \left\{
\begin{array}{cl} \frac{XZ}{YZ}, & \mbox{\rm if $Z$ is between $X$ and $Y$},\\[1ex]
-\frac{XZ}{YZ}, & \mbox{\rm otherwise}.
\end{array} \right .
\]

\begin{thm}[Menelaus]
For a triangle $ABC$ and points $P$, $Q$ and $R$ (different from the vertices) on the lines $BC$, $CA$ and $AB$ respectively, it holds that
\[
P,Q,R\; \mbox{\it are collinear} \quad  \mbox{\it iff}\quad (B,C;P)\cdot(C,A;Q)\cdot(A,B;R)=-1.
\]
\end{thm}
\noindent (The cardinality of the set of proofs of this theorem is not known. This is mostly because the question: \emph{What is a proof?}, i.e.\ when two proofs are equal, is still open.)

Note that for $X$, $Y$ and $Z$ not mutually distinct collinear points, we consider $(X,Y;Z)$ \emph{undefined}. A
sextuple $(A,B,C,P,Q,R)$ of points in $\mathbf{R}^2$ \emph{makes a
Menelaus configuration} when $(B,C;P)$, $(C,A;Q)$ and $(A,B;R)$
are defined and their product is -1.

We note that if $(A, B, C, P, Q, R)$ makes a Menelaus
configuration, then $P$, $Q$ and $R$ are collinear. This follows from the Menelaus theorem if $A$, $B$ and $C$ are
not collinear, while if $A$, $B$ and $C$ are collinear, then all $A$, $B$, $C$, $P$, $Q$ and $R$ are collinear.
On the other hand, if $A$, $B$ and $C$ are not collinear and $P$, $Q$ and $R$, different from the points $A$, $B$ and $C$,
are collinear and lie on the lines $BC$, $CA$ and $AB$, respectively,
then by the Menelaus theorem, the sextuple $(A,B,C,P,Q,R)$ makes a
Menelaus configuration.

\section{2-Cycles and Menelaus configurations}\label{configurations}
Our intention is to formalise and extend, within the proof theory,
the idea of Richter-Gebert (see \cite[Section~2.2]{RG06}), which
could be paraphrased as follows.
\begin{quote}
We consider compact, orientable 2-manifolds without boundary and
subdivisions by CW-complexes whose faces are triangles. Consider such a cycle as being interpreted by flat triangles (it does not matter if
these triangles intersect, coincide or are coplanar as long as they represent the combinatorial structure of the cycle). The presence of Menelaus configurations on all but one of the faces will imply automatically the existence of a Menelaus configuration on the final face.
\end{quote}

For example, consider the sphere $S^2$ triangulated in four
triangles arranged as the sides of a tetrahedron. Suppose that the
vertices of this triangulation as well as its six 1-faces are
interpreted as points $A$, $B$, $C$, $D$, $P$, $Q$, $R$, $U$, $V$
and $W$ in the Euclidean plane. Assume that the triangles $BCD$,
$CAD$ and $ABD$ together with the lines $WVP$, $WUQ$ and $VUR$
make Menelaus configurations. Hence, by the Menelaus theorem, the
following holds
\[
\begin{array}{lll}
(C,D;W)\cdot(D,B;V)\cdot(B,C;P)\!\! & = &\!\! -1
\\[1ex]
(D,C;W)\cdot(A,D;U)\cdot(C,A;Q)\!\! & = &\!\! -1
\\[1ex]
(B,D;V)\cdot(D,A;U)\cdot(A,B;R)\!\! & = &\!\! -1,
\end{array}
\]
which, after multiplication and cancellation, delivers
\[
(B,C;P)\cdot(C,A;Q)\cdot(A,B;R)=-1.
\]
Again, by the Menelaus theorem, this means that we have the Menelaus configuration on the final triangle $ABC$.

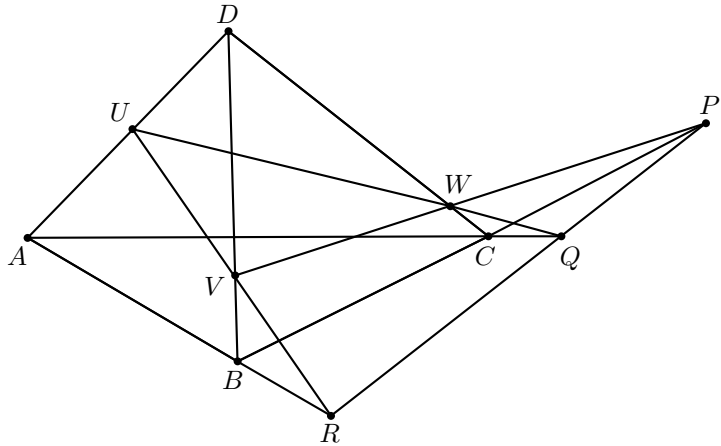
\begin{center}
\psset{xunit=1.0cm,yunit=1.0cm,algebraic=true,dimen=middle,dotstyle=o,dotsize=3pt 0,linewidth=0.8pt,arrowsize=3pt 2,arrowinset=0.25}
\begin{pspicture*}(-3.3,-2.5)(16.64,3.5)
\psline(0.02,0.32)(6.08,0.34)
\psline(0.02,0.32)(2.78,-1.32)
\psline(6.08,0.34)(2.78,-1.32)
\psline(2.66,3.06)(0.02,0.32)
\psline(2.66,3.06)(6.08,0.34)
\psline(2.66,3.06)(2.78,-1.32)
\psline(2.78,-1.32)(0.02,0.32)
\psline(2.78,-1.32)(6.08,0.34)
\psline(6.08,0.34)(2.66,3.06)
\psline(2.78,-1.32)(4.01,-2.04)
\psline(4.01,-2.04)(1.4,1.76)
\psline(1.4,1.76)(5.58,0.74)
\psline(5.58,0.74)(7.04,0.34)
\psline(4.01,-2.04)(8.94,1.84)
\psline(2.75,-0.18)(8.94,1.84)
\psline(6.08,0.34)(7.04,0.34)
\psline(6.08,0.34)(8.94,1.84)
\rput[tl](-0.25,0.2){$A$}
\rput[tl](2.58,-1.45){$B$}
\rput[tl](5.9,0.2){$C$}
\rput[tl](8.85,2.2){$P$}
\rput[tl](7.02,0.2){$Q$}
\rput[tl](3.86,-2.15){$R$}
\rput[tl](1.1,2.1){$U$}
\rput[tl](2.34,-0.2){$V$}
\rput[tl](5.5,1.15){$W$}
\rput[tl](2.5,3.4){$D$}
\begin{scriptsize}
\psdots[dotstyle=*](0.02,0.32)
\psdots[dotstyle=*](6.08,0.34)
\psdots[dotstyle=*](2.78,-1.32)
\psdots[dotstyle=*](2.66,3.06)
\psdots[dotstyle=*](4.01,-2.04)
\psdots[dotstyle=*](1.4,1.76)
\psdots[dotstyle=*](7.04,0.34)
\psdots[dotstyle=*](2.75,-0.18)
\psdots[dotstyle=*](5.58,0.74)
\psdots[dotstyle=*](8.94,1.84)
\end{scriptsize}
\end{pspicture*}
\end{center}

In the example above, the vertices and edges of the triangulation
are interpreted as points in $\mathbf{R}^2$ and it is assumed that
the three sextuples of points obtained by interpreting the
vertices and the edges of three sides of the tetrahedron make
Menelaus configurations. This suffices to conclude that the
sextuple obtained by interpreting the vertices and the edges of
the final side makes a Menelaus configuration. Our formalisation
of these matters is given in terms of $\Delta$-complexes (cf.\
\cite[Section~2.1]{H00}). (These combinatorial objects were
introduced in \cite{EZ50} under the name \emph{semi-simplicial}
complexes, and are also called \emph{Delta sets}, see \cite{F12}.)

An (abstract) $\Delta$-\emph{complex} $K$ consists of mutually disjoint sets $K_0,K_1,\ldots$ and functions $d^n_i\colon K_n\to K_{n-1}$, $n\geq 1$, $0\leq i\leq n$, which for $l-1\geq j$ satisfy
\[
d^{n-1}_j\circ d^n_l=d^{n-1}_{l-1}\circ d^n_j.
\]
The elements of $K_n$ are called $n$-\emph{cells} of $K$, and the functions $d^n_i$ are called \emph{faces}. Intuitively, $K_n$ could be conceived as a set of (ordered)
$n$-dimensional simplices. (An ordered $n$-dimensional simplex is
represented by the $(n+1)$-tuple of its vertices.) A face $d^n_i$
maps an ordered simplex
$(a_0,\hdots,a_{i-1},a_i,a_{i+1},\hdots,a_n)$ to its face
$(a_0,\hdots,a_{i-1},a_{i+1},\hdots,a_n)$ opposite to its vertex
$a_i$ as illustrated on the picture below.
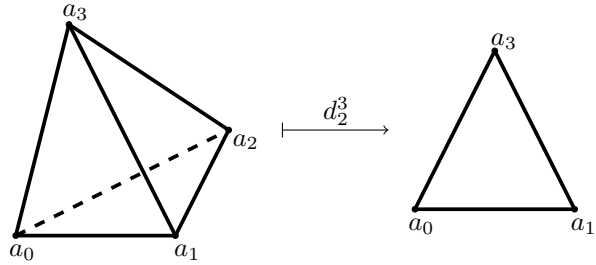
\begin{center}
\begin{tikzpicture}[scale=0.7][line cap=round,line join=round,x=1.0cm,y=1.0cm][scale=0.7]
\draw [line width=1.4pt](-3.0,-2.0)-- (0.0,-2.0);
\draw [line width=1.4pt](1.0,0.0)-- (0.0,-2.0);
\draw [line width=1.4pt][dash pattern=on 4pt off 4pt] (-3.0,-2.0)-- (1.0,0.0);
\draw [line width=1.4pt](-2.0,2.0)-- (1.0,0.0);
\draw [line width=1.4pt](-2.0,2.0)-- (0.0,-2.0);
\draw [line width=1.4pt](-2.0,2.0)-- (-3.0,-2.0);
\draw [|->] (2.0,0.0) -- (4.0,0.0);
\draw [line width=1.4pt](6.0,1.5)-- (4.5,-1.5);
\draw [line width=1.4pt](6.0,1.5)-- (7.5,-1.5);
\draw [line width=1.4pt](4.5,-1.5)-- (7.5,-1.5);
\draw (2.6,0.8) node[anchor=north west] {$d_2^3$};
\draw (-3.3,-2) node[anchor=north west] {$a_0$};
\draw (-0.2,-2) node[anchor=north west] {$a_1$};
\draw (0.9,0.1) node[anchor=north west] {$a_2$};
\draw (-2.3,2.55) node[anchor=north west] {$a_3$};
\draw (4.2,-1.5) node[anchor=north west] {$a_0$};
\draw (7.25,-1.5) node[anchor=north west] {$a_1$};
\draw (5.75,2) node[anchor=north west] {$a_3$};
\draw [fill=black] (-3.0,-2.0) circle (1.5pt);
\draw [fill=black] (0.0,-2.0) circle (1.5pt);
\draw [fill=black] (1.0,0.0) circle (1.5pt);
\draw [fill=black] (-2.0,2.0) circle (1.5pt);
\draw [fill=black] (6.0,1.5) circle (1.5pt);
\draw [fill=black] (4.5,-1.5) circle (1.5pt);
\draw [fill=black] (7.5,-1.5) circle (1.5pt);
\end{tikzpicture}
\end{center}

\begin{example}\label{Kpicture}
Let $K$ be the $\Delta$-complex illustrated by the following
picture in which $K_i$, for $i\geq 3$, is empty, and the members
of $K_i$, for $2\geq i\geq 0$, are drawn according to their
intuitive meaning (we assume that the order of vertices follows
the order of integers).
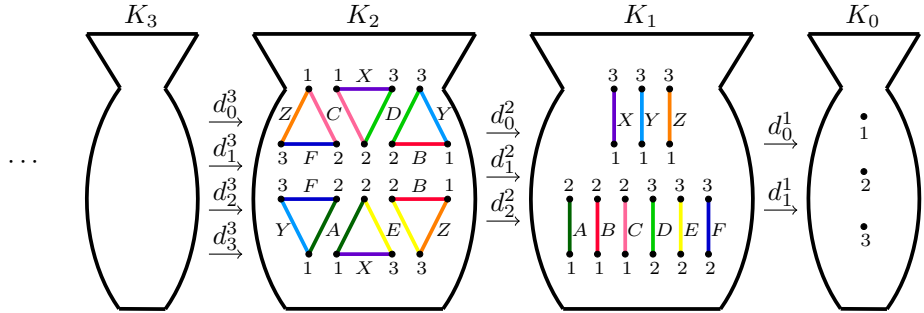
\definecolor{qqccqq}{rgb}{0.0,0.8,0.0}
\definecolor{ffwwzz}{rgb}{1.0,0.4,0.6}
\definecolor{ffxfqq}{rgb}{1.0,0.4980392156862745,0.0}
\definecolor{ffqqtt}{rgb}{1.0,0.0,0.2}
\definecolor{wwqqcc}{rgb}{0.4,0.0,0.8}
\definecolor{ffffqq}{rgb}{1.0,1.0,0.0}
\definecolor{qqwuqq}{rgb}{0.0,0.39215686274509803,0.0}
\definecolor{qqzzff}{rgb}{0.0,0.6,1.0}
\definecolor{qqqqcc}{rgb}{0.0,0.0,0.8}
\begin{center}
\begin{tikzpicture}[scale=0.73][line cap=round,line join=round]
\draw [shift={(-2.0,0.0)}][line width=1.4pt] plot[domain=0.7867592563343833:2.352763304040887,variable=\t]({-0.0*2.8284271247461907*cos(\t r)+-1.0*2.0000000000000004*sin(\t r)},{1.0*2.8284271247461907*cos(\t r)+-0.0*2.0000000000000004*sin(\t r)});
\draw [shift={(-2.0,0.0)}][line width=1.4pt] plot[domain=3.92885737586601:5.49559720598992,variable=\t]({-0.0*2.8284271247461907*cos(\t r)+-1.0*2.0000000000000004*sin(\t r)},{1.0*2.8284271247461907*cos(\t r)+-0.0*2.0000000000000004*sin(\t r)});
\draw [line width=1.4pt] (-3.4190576580550873,-1.9931258681353832)-- (-0.5831491898674361,-1.996263400369149);
\draw [line width=1.4pt] (-3.416137127902406,1.997275962392942)-- (-4.0,3.0);
\draw [line width=1.4pt] (-0.5826927915366895,1.9956153320908017)-- (0.0,3.0);
\draw [line width=1.4pt] (-4.0,3.0)-- (0.0,3.0);
\draw [line width=1.4pt] [shift={(3.0000000000000004,0.0)}] plot[domain=0.7871214706877923:2.3536736843519437,variable=\t]({-0.0*2.828427124746191*cos(\t r)+-1.0*2.0000000000000004*sin(\t r)},{1.0*2.828427124746191*cos(\t r)+-0.0*2.0000000000000004*sin(\t r)});
\draw [line width=1.4pt] [shift={(3.0000000000000004,0.0)}] plot[domain=3.9286971619280417:5.507127094800736,variable=\t]({-0.0*2.828427124746191*cos(\t r)+-1.0*2.0000000000000004*sin(\t r)},{1.0*2.828427124746191*cos(\t r)+-0.0*2.0000000000000004*sin(\t r)});
\draw [line width=1.4pt] [shift={(8.0,0.0)}] plot[domain=0.7753229437430315:2.352797277828895,variable=\t]({-0.0*2.82842712474619*cos(\t r)+-1.0*1.9999999999999998*sin(\t r)},{1.0*2.82842712474619*cos(\t r)+-0.0*1.9999999999999998*sin(\t r)});
\draw [line width=1.4pt] [shift={(6.0,0.0)}] plot[domain=3.927816279458381:5.502172784473228,variable=\t]({-0.0*2.82842712474619*cos(\t r)+-1.0*1.9999999999999998*sin(\t r)},{1.0*2.82842712474619*cos(\t r)+-0.0*1.9999999999999998*sin(\t r)});
\draw [line width=1.4pt] (1.5833514142477805,1.996550417337983)-- (1.0,3.0);
\draw [line width=1.4pt] (5.0,3.0)-- (4.400943385223519,2.0185923963982786);
\draw [line width=1.4pt] (6.0,3.0)-- (6.600106486684388,2.0200485892061466);
\begin{scope}[shift={(-2,0)}]
\draw [line width=1.4pt] (10.0,3.0)-- (9.407997749350256,2.008752019451188);
\end{scope}

\draw [line width=1.4pt] (6.580990223817692,-1.993194047301484)-- (7.415380460648183,-1.998348393856955);
\draw [line width=1.4pt] (1.5822259768757159,-1.9949520391999316)-- (4.416624638538606,-1.9965844001621214);
\draw [line width=1.4pt] (1.0,3.0)-- (5.0,3.0);
\draw [line width=1.4pt] (6.0,3.0)-- (8.0,3.0);
\draw [line width=1.4pt,color=qqqqcc] (-3.5,0.0)-- (-2.5,0.0);
\draw [line width=1.4pt,color=qqzzff] (-3.0,-1.0)-- (-3.5,0.0);
\draw [line width=1.4pt,color=qqwuqq] (-2.5,0.0)-- (-3.0,-1.0);
\draw [line width=1.4pt,color=qqwuqq] (-2.0,0.0)-- (-2.5,-1.0);
\draw [line width=1.4pt,color=ffffqq] (-1.5,-1.0)-- (-2.0,0.0);
\draw [line width=1.4pt,color=wwqqcc] (-2.5,-1.0)-- (-1.5,-1.0);
\draw [line width=1.4pt,color=ffqqtt] (-1.5,0.0)-- (-0.5,0.0);
\draw [line width=1.4pt,color=ffffqq] (-1.0,-1.0)-- (-1.5,0.0);
\draw [line width=1.4pt,color=ffxfqq] (-0.5,0.0)-- (-1.0,-1.0);
\draw [line width=1.4pt,color=ffxfqq] (-3.0,2.0)-- (-3.5,1.0);
\draw [line width=1.4pt,color=ffwwzz] (-2.5,1.0)-- (-3.0,2.0);
\draw [line width=1.4pt,color=qqqqcc] (-3.5,1.0)-- (-2.5,1.0);
\draw [line width=1.4pt,color=wwqqcc] (-2.5,2.0)-- (-1.5,2.0);
\draw [line width=1.4pt,color=qqccqq] (-2.0,1.0)-- (-1.5,2.0);
\draw [line width=1.4pt,color=ffwwzz] (-2.5,2.0)-- (-2.0,1.0);
\draw [line width=1.4pt,color=qqccqq] (-1.0,2.0)-- (-1.5,1.0);
\draw [line width=1.4pt,color=qqzzff] (-0.5,1.0)-- (-1.0,2.0);
\draw [line width=1.4pt,color=ffqqtt] (-1.5,1.0)-- (-0.5,1.0);

\draw [line width=1.4pt,color=wwqqcc] (2.5,2.0)-- (2.5,1.0); 
\draw [line width=1.4pt,color=qqzzff] (3.0,2.0)-- (3.0,1.0); 
\draw [line width=1.4pt,color=ffxfqq] (3.5,2.0)-- (3.5,1.0); 

\draw [line width=1.4pt,color=qqwuqq] (1.7,0.0)-- (1.7,-1.0);
\draw [line width=1.4pt,color=ffqqtt] (2.2,0.0)-- (2.2,-1.0);
\draw [line width=1.4pt,color=ffwwzz] (2.7,0.0)-- (2.7,-1.0);
\draw [line width=1.4pt,color=qqccqq] (3.2,0.0)-- (3.2,-1.0);
\draw [line width=1.4pt,color=ffffqq] (3.7,0.0)-- (3.7,-1.0);
\draw [line width=1.4pt,color=qqqqcc] (4.2,0.0)-- (4.2,-1.0);

\begin{scriptsize}
\draw (-3.7,1) node[anchor=north west] {$3$};
\draw (-3.25,2.5) node[anchor=north west] {$1$};
\draw (-2.7,2.5) node[anchor=north west] {$1$};
\draw (-0.7,1) node[anchor=north west] {$1$};
\draw (-0.7,0.5) node[anchor=north west] {$1$};
\draw (-2.7,-1) node[anchor=north west] {$1$};
\draw (-3.25,-1) node[anchor=north west] {$1$};
\draw (-3.7,0.5) node[anchor=north west] {$3$};
\draw (-1.7,-1) node[anchor=north west] {$3$};
\draw (-1.2,-1) node[anchor=north west] {$3$};
\draw (-1.2,2.5) node[anchor=north west] {$3$};
\draw (-1.7,2.5) node[anchor=north west] {$3$};
\draw (-2.7,1) node[anchor=north west] {$2$};
\draw (-2.2,1) node[anchor=north west] {$2$};
\draw (-1.7,1) node[anchor=north west] {$2$};
\draw (-1.7,0.5) node[anchor=north west] {$2$};
\draw (-2.2,0.5) node[anchor=north west] {$2$};
\draw (-2.7,0.5) node[anchor=north west] {$2$};
\draw (-3.7,1.8) node[anchor=north west] {$Z$};
\draw (-0.85,-0.3) node[anchor=north west] {$Z$};
\draw (3.6,-0.3) node[anchor=north west] {$E$};
\draw (1.5,-1) node[anchor=north west] {$1$};
\draw (2.3,1) node[anchor=north west] {$1$};
\draw (2.8,1) node[anchor=north west] {$1$};
\draw (2.5,-1) node[anchor=north west] {$1$};
\draw (3,-1) node[anchor=north west] {$2$};
\draw (3.5,-1) node[anchor=north west] {$2$};
\draw (1.45,0.5) node[anchor=north west] {$2$};
\draw (2.25,2.5) node[anchor=north west] {$3$};
\draw (2.75,2.5) node[anchor=north west] {$3$};
\draw (3.25,2.5) node[anchor=north west] {$3$};
\draw (3.95,0.5) node[anchor=north west] {$3$};
\draw (2.45,0.5) node[anchor=north west] {$2$};
\draw (2.95,0.5) node[anchor=north west] {$3$};
\draw (3.45,0.5) node[anchor=north west] {$3$};
\draw (3.3,1) node[anchor=north west] {$1$};
\draw (4,-1) node[anchor=north west] {$2$};
\draw (2,-1) node[anchor=north west] {$1$};
\draw (1.95,0.5) node[anchor=north west] {$2$};
\draw (-2.3,2.5) node[anchor=north west] {$X$};
\draw (-2.3,-1) node[anchor=north west] {$X$};
\draw (2.6,-0.3) node[anchor=north west] {$C$};
\draw (-0.85,1.8) node[anchor=north west] {$Y$};
\draw (-3.75,-0.3) node[anchor=north west] {$Y$};
\draw (3.1,-0.3) node[anchor=north west] {$D$};
\draw (-2.83,1.8) node[anchor=north west] {$C$};
\draw (-1.78,1.8) node[anchor=north west] {$D$};
\draw (-1.3,1) node[anchor=north west] {$B$};
\draw (-1.3,0.5) node[anchor=north west] {$B$};
\draw (3.4,1.7) node[anchor=north west] {$Z$};
\draw (2.4,1.7) node[anchor=north west] {$X$};
\draw (-3.25,1) node[anchor=north west] {$F$};
\draw (-3.25,0.5) node[anchor=north west] {$F$};
\draw (-2.85,-0.3) node[anchor=north west] {$A$};
\draw (-1.77,-0.3) node[anchor=north west] {$E$};
\draw (1.6,-0.3) node[anchor=north west] {$A$};
\draw (2.9,1.7) node[anchor=north west] {$Y$};
\draw (4.1,-0.3) node[anchor=north west] {$F$};
\draw (2.1,-0.3) node[anchor=north west] {$B$};
\draw (6.8,0.5) node[anchor=north west] {$2$};
\draw (6.8,-0.5) node[anchor=north west] {$3$};
\draw (6.77,1.45) node[anchor=north west] {$1$};
\end{scriptsize}
\draw (-6.5,3.7) node[anchor=north west] {$K_3$};
\draw (-2.5,3.7) node[anchor=north west] {$K_2$};
\draw (2.5,3.7) node[anchor=north west] {$K_1$};
\draw (6.5,3.7) node[anchor=north west] {$K_0$};
\draw [->] (0.2,0.4) -- (0.8,0.4);
\draw [->] (0.2,1.2) -- (0.8,1.2);
\draw [->] (0.2,-0.4) -- (0.8,-0.4);
\draw [->] (5.2,-0.2) -- (5.8,-0.2);
\draw [->] (5.2,1.0) -- (5.8,1.0);
\draw [->] (-4.8,1.4) -- (-4.2,1.4);
\draw [->] (-4.8,0.6) -- (-4.2,0.6);
\draw [->] (-4.8,-0.2) -- (-4.2,-0.2);
\draw [->] (-4.8,-1.0) -- (-4.2,-1.0);
\draw (-8.6,0.95) node[anchor=north west] {$\Large{\cdots}$};
\draw (0.1,1.95) node[anchor=north west] {$d_0^2$};
\draw (0.1,1.15) node[anchor=north west] {$d_1^2$};
\draw (0.1,0.35) node[anchor=north west] {$d_2^2$};
\draw (5.1,1.75) node[anchor=north west] {$d_0^1$};
\draw (5.1,0.55) node[anchor=north west] {$d_1^1$};
\draw (-4.9,2.15) node[anchor=north west] {$d_0^3$};
\draw (-4.9,1.35) node[anchor=north west] {$d_1^3$};
\draw (-4.9,0.55) node[anchor=north west] {$d_2^3$};
\draw (-4.9,-0.25) node[anchor=north west] {$d_3^3$};
\begin{scriptsize}
\draw [fill=black] (-3.5,0.0) circle (1.5pt);
\draw [fill=black] (-2.5,0.0) circle (1.5pt);
\draw [fill=black] (-3.0,-1.0) circle (1.5pt);
\draw [fill=black] (-2.0,0.0) circle (1.5pt);
\draw [fill=black] (-2.5,-1.0) circle (1.5pt);
\draw [fill=black] (-1.5,-1.0) circle (1.5pt);
\draw [fill=black] (-1.5,0.0) circle (1.5pt);
\draw [fill=black] (-0.5,0.0) circle (1.5pt);
\draw [fill=black] (-1.0,-1.0) circle (1.5pt);
\draw [fill=black] (2.2,0.0) circle (1.5pt);
\draw [fill=black] (2.2,-1.0) circle (1.5pt);
\draw [fill=black] (2.7,0.0) circle (1.5pt);
\draw [fill=black] (2.7,-1.0) circle (1.5pt);
\draw [fill=black] (3.2,0.0) circle (1.5pt);
\draw [fill=black] (3.2,-1.0) circle (1.5pt);
\draw [fill=black] (3.7,0.0) circle (1.5pt);
\draw [fill=black] (3.7,-1.0) circle (1.5pt);
\draw [fill=black] (7.0,-0.5) circle (1.5pt);
\draw [fill=black] (7.0,0.5) circle (1.5pt);
\draw [fill=black] (-3.0,2.0) circle (1.5pt);
\draw [fill=black] (-3.5,1.0) circle (1.5pt);
\draw [fill=black] (-2.5,1.0) circle (1.5pt);
\draw [fill=black] (-2.5,2.0) circle (1.5pt);
\draw [fill=black] (-1.5,2.0) circle (1.5pt);
\draw [fill=black] (-2.0,1.0) circle (1.5pt);
\draw [fill=black] (-1.0,2.0) circle (1.5pt);
\draw [fill=black] (-1.5,1.0) circle (1.5pt);
\draw [fill=black] (-0.5,1.0) circle (1.5pt);
\draw [fill=black] (1.7,-1) circle (1.5pt);
\draw [fill=black] (1.7,0) circle (1.5pt);
\draw [fill=black] (2.5,2.0) circle (1.5pt);
\draw [fill=black] (2.5,1.0) circle (1.5pt);
\draw [fill=black] (3.0,2.0) circle (1.5pt);
\draw [fill=black] (3.0,1.0) circle (1.5pt);
\draw [fill=black] (3.5,2.0) circle (1.5pt);
\draw [fill=black] (3.5,1.0) circle (1.5pt);
\draw [fill=black] (4.2,-1) circle (1.5pt);
\draw [fill=black] (4.2,0) circle (1.5pt);
\draw [fill=black] (7.0,1.5) circle (1.5pt);
\end{scriptsize}

\begin{scope}[shift={(-13,0)}]
\draw [line width=1.4pt] [shift={(6.0,0.0)}] plot[domain=3.927816279458381:5.502172784473228,variable=\t]({-0.0*2.82842712474619*cos(\t r)+-1.0*1.9999999999999998*sin(\t r)},{1.0*2.82842712474619*cos(\t r)+-0.0*1.9999999999999998*sin(\t r)});
\draw [line width=1.4pt] [shift={(8.0,0.0)}] plot[domain=0.7753229437430315:2.352797277828895,variable=\t]({-0.0*2.82842712474619*cos(\t r)+-1.0*1.9999999999999998*sin(\t r)},{1.0*2.82842712474619*cos(\t r)+-0.0*1.9999999999999998*sin(\t r)});
\begin{scope}[shift={(-2,0)}]
\draw [line width=1.4pt] (10.0,3.0)-- (9.407997749350256,2.008752019451188);
\end{scope}
\draw [line width=1.4pt] (6.0,3.0)-- (6.600106486684388,2.0200485892061466);
\draw [line width=1.4pt] (6.580990223817692,-1.993194047301484)-- (7.415380460648183,-1.998348393856955);
\draw [line width=1.4pt] (6.0,3.0)-- (8.0,3.0);
\end{scope}
\end{tikzpicture}
\end{center}
The faces are indicated by the labels (and colours). The face
$d_0^2$ maps the $2$-cell represented by the triangle with the
sides $BDY$ to the $1$-cell $D\in K_1$, which is opposite to the
initial vertex $1$ of this triangle. Also, $d_2^2$ maps the same
triangle to $B\in K_1$.
\end{example}

The above intuitive interpretation of a $\Delta$-complex brings us closer to the notion of its geometric realisation. Formally, it goes as follows.

For the standard ordered simplices
\[
\Delta^n=\{(t_0,\ldots,t_n)\mid t_0,\ldots, t_n\geq 0, \sum_{i=0}^n
t_i=1\},
\]
in the Euclidean space, and the maps $\delta^n_i\colon \Delta^{n-1}\to \Delta^n$ defined by
\[
\delta^n_i(t_0,\ldots,t_i,\ldots,t_{n-1})=(t_0,\ldots,0,t_i,\ldots,t_{n-1}), \]
the \emph{geometric realisation} of a $\Delta$-complex $K$ is the following quotient space (a tensor product)
\[
|K|=\left(\coprod_n K_n\times\Delta^n\right){\Big\slash}\sim,
\]
where the equivalence relation $\sim$ is generated by
\[
(d^n_ix,t)\sim(x,\delta^n_it).
\]
Intuitively, as in our illustration of $K$ in Example~\ref{Kpicture}, we take copies of $\Delta^n$ as much as $K_n$ has elements, and then glue them along their sides according to the functions $d^n_i$. As the result of the procedure applied to this $K$ we obtain a torus triangulated as in Example~\ref{example-4} of Section~\ref{reading}.

In order to relax the notation, we omit the superscripts from $d^n_i$ when they are clear from the context. For every $n\geq 0$, let $C_n$ be the free abelian group generated by $K_n$ and let the \emph{boundary} homomorphism $\partial_n\colon C_n\to C_{n-1}$ be defined on generators by
\[
\partial_n x=\sum_{i=0}^n (-1)^i d_i x.
\]
Since we are interested just in the boundary homomorphism $\partial_2$, we omit the subscript 2. A 2-\emph{cycle} $c$ is an element of $C_2$ such that $\partial c=0$. For $n\geq 1$, let
\[
c=\varepsilon_1 x_1+\varepsilon_2 x_2+\ldots +\varepsilon_{n-1} x_{n-1} - x_n
\]
be a 2-cycle, where $\varepsilon_i\in\{-1,1\}$, and $x_i$, $x_j$ could be equal when $i\neq j$. (If the occurrence of $x_n$ in $c$ is positive, then one can replace $c$ by the 2-cycle $-c$.)

\begin{rem}\label{neven}
In the expression of $c$ as $\varepsilon_1 x_1+\varepsilon_2
x_2+\ldots +\varepsilon_{n-1} x_{n-1} - x_n$, $n$ is even.
\end{rem}

\begin{proof}
Since $c$ is a 2-cycle, and every boundary $\partial x_i$ is of the form $y_{3i-2}-y_{3i-1}+y_{3i}$, for $y_j\in K_1$, we have that
\[
0=\partial c=\sum_{j=1}^{3n} \tau_j y_j,
\]
where $\tau_j\in\{-1,1\}$. Here $y$'s must repeat, since $C_1$ is free, and the cancellation happens only when two identical $y$'s with the opposite signs occur in the sum. Therefore the number $3n$, and hence $n$, must be even.
\end{proof}

For an arbitrary function $v\colon K_0\cup K_1\to \mathbf{R}^2$, consider the operator $\mu\colon K_2\to (\mathbf{R}^2)^6$ defined by
\[
\mu x=(vd_1d_2 x, vd_0d_2 x, vd_0d_0 x, vd_0 x, vd_1 x, vd_2 x).
\]

\begin{example}
Assume that $v$ maps the elements of
\[
K_0\cup K_1=\{A,B,C,D,E,F,X,Y,Z,1,2,3\}
\]
from Example~\ref{Kpicture} to the points in $\mathbf{R}^2$ denoted by the same symbols. Let $x$ be the element of $K_2$ illustrated by the triangle with the sides $BDY$. Then $\mu x=(1,2,3,D,Y,B)$.
\end{example}

Intuitively, the operator $\mu$ maps an ordered triangle $ABC$, whose sides are $a$, $b$ and $c$ respectively, into the sextuple $(vA,vB,vC,va,vb,vc)$. We claim the following.

\begin{prop}\label{cycle}
If $\mu x_1,\ldots,\mu x_{n-1}$ make Menelaus configurations, then $\mu x_n$ makes a Menelaus configuration, too.
\end{prop}

For the proof of this proposition we introduce the following partial function (``partial homomorphism'') $h\colon (C_1,+,0)\to (\mathbf{R}-\{0\},\cdot,1)$. Note that every element $a$ of $C_1-\{0\}$ could be written uniquely (up to associativity and commutativity) as $\alpha_1 y_1+\ldots+\alpha_m y_m$, where $\alpha_i\in \mathbf{Z}-\{0\}$ and the $y_i$'s are mutually distinct elements of $K_1$. If for every $i\in\{1,\ldots,m\}$ we have that
\[
hy_i=(vd_0 y_i, vd_1 y_i; vy_i)
\]
is defined, then $ha=_{df} (hy_1)^{\alpha_1}\cdot\ldots\cdot (hy_m)^{\alpha_m}$, otherwise, $ha$ is undefined. To complete the definition, let $h0=1$.

\begin{rem}\label{homo}
If $ha_1$ and $ha_2$ are defined, then $h(a_1+a_2)$ is defined and equal to $ha_1 \cdot ha_2$.
\end{rem}

\begin{rem}\label{hmene}
A sextuple $\mu x$ makes a Menelaus configuration iff $h\partial x$ is defined and equal to -1.
\end{rem}

\begin{proof}[Proof of Proposition~\ref{cycle}]
Since $c$ is a 2-cycle, we have that
\[
\partial x_n=\sum_{i=1}^{n-1} \varepsilon_i\partial x_i.
\]
By Remark~\ref{hmene}, for every $i\in\{1,\ldots,n-1\}$, we have that
$h(\varepsilon_i \partial x_i)$, which is either $h(\partial x_i)$ or its reciprocal value, is equal to $-1$. By Remark~\ref{neven}, the number $n-1$ is odd, and by Remark~\ref{homo}, we have $h\partial x_n=-1$, which means, again by Remark~\ref{hmene}, that $\mu x_n$ makes a Menelaus configuration.
\end{proof}

\begin{rem}\label{mcomplex}
If for some $i\in\{1,\ldots,n\}$ the $x_i$ involved in the 2-cycle
$c$ is such that two different faces map $x_i$ to the same element
of $K_1$, then $\mu x_i$ does not make a Menelaus configuration.
In such a situation, the implication of Proposition~\ref{cycle}
holds vacuously since its antecedent is false. The same holds if
for some $j\in\{1,\ldots,3n\}$, the face $y_j$ of some $x_i$,
involved in $c$, is such that two different faces map $y_j$ to the
same element of $K_0$. Hence, the only interesting case is to
consider 2-cycles in $\Delta$-complexes $K$ in which two different
faces map each element of $K_2$ ($K_1$) to two different elements
of $K_1$ ($K_0$).
\end{rem}

Proposition~\ref{cycle} shows that Richter-Gebert's idea could be generalized from closed orientable triangulated surfaces to arbitrary 2-cycles of $\Delta$-complexes. However, we will show that this generalisation is not essential. What follows may serve as a light introduction to Steenrod's problem (see \cite[Section~7, Problem~25]{E47} for the formulation, and \cite{T54} for the solution).

Let $K$ be a $\Delta$-complex and let $c=\sum_{i=1}^n \varepsilon_i x_i$, where $\varepsilon_i\in\{-1,1\}$ and $x_i\in K_2$, be a 2-cycle as above. Moreover, assume that the 2-dimensional $\Delta$-complex structure involved in the 2-cycle $c$ satisfies the condition formulated at the end of Remark~\ref{mcomplex}.
For $\partial x_i=y_{3i-2}-y_{3i-1}+y_{3i}$, we have that
\[
0=\partial c=\sum_{j=1}^{3n} \tau_j y_j,
\]
where $\tau_j\in\{-1,1\}$, and $3n=2m$ for some $m\geq 1$. Choose a partition of the set $\{1,\ldots,2m\}$ with all classes containing exactly two elements, such that $i$ and $j$ belong to the same class when $y_i=y_j$ and $\tau_i=-\tau_j$. Denote these classes by $s_1,\ldots,s_m$.

Let $L$ be the $\Delta$-complex constructed as follows. For $m\geq 3$, we set $L_m=\emptyset$, $L_2=\{u_1,\ldots,u_n\}$ (for genuine $u_i$'s), and $L_1=\{s_1,\ldots,s_m\}$. For $k\in\{0,1,2\}$, let $d_k\colon L_2\to L_1$ be the functions defined so that
\[
d_k(u_i)=s_j\quad \mbox{\rm when}\quad 3i-2+k\in s_j.
\]
Let $L_0$ be the quotient set $\{(s_1,0),(s_1,1),\ldots,(s_m,0),(s_m,1)\}/\approx$, where $\approx$ is the smallest equivalence relation satisfying for every $i\in\{1,\ldots,n\}$:
\[
(d_1 u_i,0)\approx(d_0 u_i,0),\quad (d_2 u_i,0)\approx(d_0 u_i,1),\quad (d_2 u_i,1)\approx(d_1 u_i,1).
\]
Finally, let $d_0 s_j=(s_j,0)_\approx$ and $d_1 s_j=(s_j,1)_\approx$.

\begin{rem}
Note that $c'=\sum_{i=1}^n \varepsilon_i u_i$ is a 2-cycle of the $\Delta$-complex $L$. This complex depends on the 2-cycle $c$ of the $\Delta$-complex $K$, and the choice of an appropriate pairing of occurrences of 1-cells in the expression of $\partial c$ given above.
\end{rem}

A \emph{morphism} $f\colon L\to K$ between $\Delta$-complexes is a family of functions $\{f_i\colon L_i\to K_i\mid i\in \mathbf{N}\}$ that commute with the faces. For $K$ and $L$ as above, if $f_2$, $f_1$ and $f_0$ are defined so that $f_2(u_i)=x_i$, $f_1(s_j)=y_k$, for $k\in s_j$, and for $l\in\{0,1\}$, $f_0((s_j,l)_\approx)=d_l f_1(s_j)$, then $f=(f_0,f_1,f_2,\emptyset,\ldots)$ is a morphism from $L$ to $K$.

Every function $v\colon K_0\cup K_1\to \mathbf{R}^2$, as above, is lifted by $f$ to a function $v'\colon L_0\cup L_1\to \mathbf{R}^2$ (roughly $v'=v\circ (f_0\cup f_1)$). If the operator $\mu'\colon L_2\to \mathbf{R}^2$ is defined as the operator $\mu$ from above, save that $v$ is now replaced by $v'$, then we have that $\mu' u_i$ makes a Menelaus configuration iff $\mu x_i$ makes a Menelaus configuration. Hence, if an incidence result follows from an interpretation of the 2-cycle $c=\sum_{i=1}^n \varepsilon_i x_i$ of $K$, then it follows from an interpretation of the 2-cycle $c'=\sum_{i=1}^n \varepsilon_i u_i$ of $L$.

\vspace{1ex}

The complex $L$ that we just constructed satisfies special properties, which we now list. We first introduce some terminology.
For a 2-cell $u$, we call a face of $u$ an \emph{edge} of $u$ and a face of a face of $u$ a \emph{vertex} of $u$. Also, a face of a 1-cell is called its \emph{vertex}. A $\Delta$-complex is \emph{connected} when for each pair of mutually distinct 0-cells $w$ and $w'$, there is a sequence $w=w_0,\ldots,w_n=w'$ of 0-cells, such that every pair of consecutive elements in it is the pair of vertices of a 1-cell. A \emph{connected component} of a $\Delta$-complex is defined as expected.
We say that two 2-cells are $w$-\emph{neighbours} when they share an edge having $w$ as a vertex.

The $\Delta$-complex $L$ satisfies:
\begin{enumerate}
\item[(0)] $L$ is \emph{finite}, i.e.\ it has a finite number of cells;
\item[(1)] $L$ is \emph{homogeneous} 2-dimensional, i.e.\ for every $m\geq 3$ the set $L_m$ is empty and every element of $L_0\cup L_1$ is a face of some element of $L_1\cup L_2$;
\item[(2)] $L$ is \emph{regular}, i.e.\ two different faces map an element of $L_2$ ($L_1$) to two different elements of $L_1$ ($L_0$);
\item[(3)] for every 1-cell of $L$ there are exactly two 2-cells having this 1-cell as an edge;
\item[(4)] for every $w\in L_0$, the set $L_w=\{u\in L_2\mid w\; \mbox{\rm is a vertex of}\; u\}$ is \emph{linked} in the sense that if $u,u'\in L_w$, then there is a sequence of 2-cells starting at $u$ and ending at $u'$, such that every two consecutive 2-cells are $w$-neighbours;
\item[(5)] $L$ is \emph{orientable}, i.e.\ the second homology
group \[H_2(L;\mathbf{Z})=\mbox{\rm Ker}\, \partial_2/ \mbox{\rm
Im}\, \partial_3=\mbox{\rm Ker}\, \partial_2\cong\mathbf{Z},\]
which consists of all 2-cycles of $L$, is isomorphic to the direct
sum of $k$ copies of $\mathbf{Z}$, where $k$ is the number of
connected components of $L$ (cf.\ Remark~\ref{orientability}
below). If $L$ is connected, then one may take a generator
$\sum_{i=1}^n \varepsilon_i u_i$ of $H_2(L;\mathbf{Z})$ as an
\emph{orientation} of $L$.
\end{enumerate}
The regularity condition follows from our assumption that $c$ satisfies the condition formulated at the end of Remark~\ref{mcomplex}.
The condition~(4) follows from the definition of $\approx$. We call a connected $\Delta$-complex that satisfies (0)-(5), an $\mathcal{M}$-\emph{complex} ($\mathcal{M}$ comes from Menelaus).

\begin{rem}\label{triang}
By the regularity property, every 2-cell $u$ has three mutually distinct edges $d_0 u$, $d_1 u$ and $d_2 u$, and three mutually distinct \emph{vertices}, the 0th, $d_1d_2 u=d_1d_1 u$, the 1st, $d_0d_2 u=d_1d_0 u$ and the 2nd, $d_0d_0 u=d_0d_1 u$. Every vertex of $u$ is the common vertex of exactly two edges of $u$.
\end{rem}

\begin{rem}\label{orientability}
For a definition of orientability, one may consult textbooks in
algebraic topology. According to \cite[Theorem~3.25]{H00}, a
connected $n$-manifold $M$ is orientable iff
$H_n(M;\mathbf{Z})=\mathbf{Z}$. A generator of the infinite cyclic
group is called the \emph{fundamental class}. Since there are two
possible choices of a generator, there are exactly two
orientations of such a manifold. The fundamental class of a
connected orientable $n$-manifold having a structure of CW complex
is the sum of its $n$-cells oriented in accordance with the local
orientations of the manifold, see \cite[Problem~5,
Section~16.4]{TD08}, and this is the classical interpretation of
the fundamental class of a triangulated manifold.
\end{rem}

\begin{example}\label{Mcomplex} The ``dunce hat'' is a geometric realisation of a $\Delta$-complex obtained by identifying all the three edges of a single triangle, preserving the orientations of these edges. By construction, this is not a regular $\Delta$-complex. The ``dunce hat'' is not a manifold, and its second homology group contains just the trivial 2-cycle--hence, it cannot be triangulated in a manner interesting for the Menelaus reasoning.

\smallskip

Consider the $\Delta$-complexes (a), (b), (c), (d) and (e).
The $\Delta$-complex (a) is not homogeneous, since it contains an ``antenna'' which is not an edge of any triangle. The  $\Delta$-complex (b), obtained from the $\Delta$-complex (a) by removing the antenna, is homogeneous, but it does not satisfy the property (3) from the definition of $\mathcal{M}$-complexes. The $\Delta$-complex (d), obtained by identifying (without twisting) the opposite sides of the square, triangulated by a diagonal, is not regular. The   $\Delta$-complexes (c) and (e) are  $\mathcal{M}$-complexes.
\end{example}

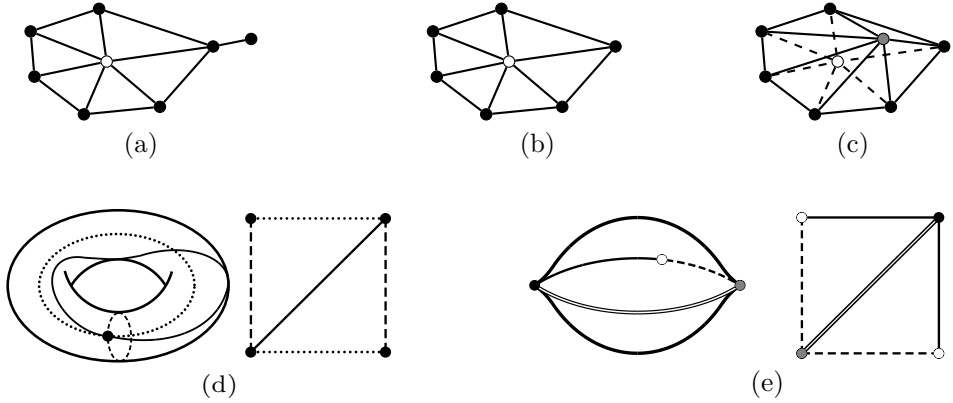
\begin{center}
\begin{tabular}{ccc}
\begin{tikzpicture}
\node (1) [circle,draw=black,fill=black,inner sep=0,minimum size=1.5mm] at (0,0) {};
\node (2) [circle,draw=black,fill=black,inner sep=0,minimum size=1.5mm] at (1.5,-0.5) {};
\node (3) [circle,draw=black,fill=black,inner sep=0,minimum size=1.5mm] at (0.8,-1.3) {};
\node (4) [circle,draw=black,fill=black,inner sep=0,minimum size=1.5mm] at (-0.2,-1.4) {};
\node (5) [circle,draw=black,fill=black,inner sep=0,minimum size=1.5mm] at (-0.85,-0.9) {};
\node (6) [circle,draw=black,fill=black,inner sep=0,minimum size=1.5mm] at (-0.9,-0.3) {};
\node (7) [circle,draw=black,fill=black,inner sep=0,minimum size=1.5mm] at (2,-0.4) {};
\node (c) [circle,draw=black,inner sep=0,minimum size=1.5mm] at (0.1,-0.7) {};
\draw[thick] (1)--(2)--(3)--(4)--(5)--(6)--(1);
\draw[thick,fill=cyan] (2)--(c)--(1)--cycle;
\draw[thick] (3)--(c)--(4);
\draw[thick] (5)--(c)--(6);
\draw[thick] (2)--(7);
\end{tikzpicture} \quad\quad\quad & \quad\quad\quad \begin{tikzpicture}
\node (1) [circle,draw=black,fill=black,inner sep=0,minimum size=1.5mm] at (0,0) {};
\node (2) [circle,draw=black,fill=black,inner sep=0,minimum size=1.5mm] at (1.5,-0.5) {};
\node (3) [circle,draw=black,fill=black,inner sep=0,minimum size=1.5mm] at (0.8,-1.3) {};
\node (4) [circle,draw=black,fill=black,inner sep=0,minimum size=1.5mm] at (-0.2,-1.4) {};
\node (5) [circle,draw=black,fill=black,inner sep=0,minimum size=1.5mm] at (-0.85,-0.9) {};
\node (6) [circle,draw=black,fill=black,inner sep=0,minimum size=1.5mm] at (-0.9,-0.3) {};
\node (c) [circle,draw=black,inner sep=0,minimum size=1.5mm] at (0.1,-0.7) {};
\draw[thick] (1)--(2)--(3)--(4)--(5)--(6)--(1);
\draw[thick,fill=cyan] (2)--(c)--(1)--cycle;
\draw[thick] (3)--(c)--(4);
\draw[thick] (5)--(c)--(6);
\end{tikzpicture}\quad\quad\quad &  \quad\quad \begin{tikzpicture}
\node (1) [circle,draw=black,fill=black,inner sep=0,minimum size=1.5mm] at (0,0) {};
\node (2) [circle,draw=black,fill=black,inner sep=0,minimum size=1.5mm] at (1.5,-0.5) {};
\node (3) [circle,draw=black,fill=black,inner sep=0,minimum size=1.5mm] at (0.8,-1.3) {};
\node (4) [circle,draw=black,fill=black,inner sep=0,minimum size=1.5mm] at (-0.2,-1.4) {};
\node (5) [circle,draw=black,fill=black,inner sep=0,minimum size=1.5mm] at (-0.85,-0.9) {};
\node (6) [circle,draw=black,fill=black,inner sep=0,minimum size=1.5mm] at (-0.9,-0.3) {};
\node (7) [circle,draw=black,fill=gray,inner sep=0,minimum size=1.5mm] at (0.7,-0.4) {};
\node (c) [circle,draw=black,inner sep=0,minimum size=1.5mm] at (0.1,-0.7) {};
\draw[thick] (1)--(2)--(3)--(4)--(5)--(6)--(1);
\draw[thick,dashed] (2)--(c)--(1);
\draw[thick,dashed] (3)--(c)--(4);
\draw[thick,dashed] (5)--(c)--(6);
\draw[thick] (2)--(7)--(1);
\draw[thick] (3)--(7)--(4);
\draw[thick] (5)--(7)--(6);
\end{tikzpicture}  \\
(a) \quad\quad\quad   & \quad\quad  (b)  & \quad\quad  (c)
\end{tabular}
\end{center}
\newrgbcolor{qqwuqq}{0 0.39 0}
\psset{xunit=1.0cm,yunit=1.0cm,algebraic=true,dimen=middle,dotstyle=o,dotsize=3pt 0,linewidth=0.8pt,arrowsize=3pt 2,arrowinset=0.25}
\psscalebox{.35 .35}{\begin{pspicture*}(-4.2,-5.0)(14.95,3.67)
\rput{0}(0,0){\psellipse[linewidth=2.8pt](0,0)(4.15,2.86)}
\parametricplot[linewidth=2.8pt]{3.3936126414032994}{6.021215152704372}{1*2.08*cos(t)+0*2.08*sin(t)+0.01|0*2.08*cos(t)+1*2.08*sin(t)+1.08}
\parametricplot[linewidth=2.8pt]{0.545463398049105}{2.5961292555406885}{1*2.08*cos(t)+0*2.08*sin(t)+0.01|0*2.08*cos(t)+1*2.08*sin(t)+-1.08}
\parametricplot[linewidth=2pt,linestyle=dashed]{3.1725306936091595}{6.27699651909602}{0.03*0.91*cos(t)+1*0.44*sin(t)+0.05|-1*0.91*cos(t)+0.03*0.44*sin(t)+-1.94}
\parametricplot[linewidth=2pt,linestyle=dashed]{-0.006188788083566266}{3.1725306936091595}{0.03*0.91*cos(t)+1*0.44*sin(t)+0.05|-1*0.91*cos(t)+0.03*0.44*sin(t)+-1.94}
\rput{0.26}(-0.05,0.02){\psellipse[linewidth=3pt,linestyle=dotted, dotsep=2.2pt](0,0)(2.92,1.94)}

\parametricplot[linewidth=2pt]{0.566705036957486}{0.7054235896031844}{-0.9998760816978752*7.096177226644504*cos(t)+-0.015742339359308265*0.18510870365745755*sin(t)+5.301019427474969|0.015742339359308265*7.096177226644504*cos(t)+-0.9998760816978752*0.18510870365745755*sin(t)+1.031783664384778}
\parametricplot[linewidth=2pt]{0.380967606642863}{0.566705036957486}{-0.9998760816978752*7.096177226644504*cos(t)+-0.015742339359308265*0.18510870365745755*sin(t)+5.301019427474969|0.015742339359308265*7.096177226644504*cos(t)+-0.9998760816978752*0.18510870365745755*sin(t)+1.031783664384778}
\parametricplot[linewidth=2pt]{2.3425091746090967}{2.8387547334015952}{0.7426377437737344*1.8076922130348425*cos(t)+
0.6696933488714497*1.1946599732112166*sin(t)+-0.9238970303772363|-0.6696933488714497*1.8076922130348425*cos(t)+0.7426377437737344*1.1946599732112166*sin(t)+-0.41343291695143974}
\parametricplot[linewidth=2pt]{2.8387547334015952}{3.683110804165384}{0.74*1.81*cos(t)+0.67*1.19*sin(t)+-0.92|-0.67*1.81*cos(t)+0.74*1.19*sin(t)+-0.41}
\parametricplot[linewidth=2pt]{3.683110804165384}{4.391376177691744}{0.74*1.81*cos(t)+0.67*1.19*sin(t)+-0.92|-0.67*1.81*cos(t)+0.74*1.19*sin(t)+-0.41}
\parametricplot[linewidth=2pt]{4.391376177691744}{5.597148995225213}{0.74*1.81*cos(t)+0.67*1.19*sin(t)+-0.92|-0.67*1.81*cos(t)+0.74*1.19*sin(t)+-0.41}
\parametricplot[linewidth=2pt]{5.1095219063415405}{5.491304767012243}{0.88*5.19*cos(t)+0.47*3.43*sin(t)+-0.66|-0.47*5.19*cos(t)+0.88*3.43*sin(t)+1.82}
\parametricplot[linewidth=2pt]{5.491304767012243}{5.761969349776725}{0.88*5.19*cos(t)+0.47*3.43*sin(t)+-0.66|-0.47*5.19*cos(t)+0.88*3.43*sin(t)+1.82}
\parametricplot[linewidth=2pt]{5.761969349776725}{6.187839890934538}{0.88*5.19*cos(t)+0.47*3.43*sin(t)+-0.66|-0.47*5.19*cos(t)+0.88*3.43*sin(t)+1.82}
\parametricplot[linewidth=2pt]{-0.09534541624504822}{0.21017604257805264}{0.88*5.19*cos(t)+0.47*3.43*sin(t)+-0.66|-0.47*5.19*cos(t)+0.88*3.43*sin(t)+1.82}
\parametricplot[linewidth=2pt]{1.6068552714148963}{1.9688311330728567}{-0.99*2.28*cos(t)+0.11*0.48*sin(t)+-0.14|-0.11*2.28*cos(t)+-0.99*0.48*sin(t)+1.47}
\parametricplot[linewidth=2pt]{2.9814526399487637}{3.585684166480425}{-0.98*2.55*cos(t)+0.19*1.49*sin(t)+1.63|-0.19*2.55*cos(t)+-0.98*1.49*sin(t)+-0.19}
\parametricplot[linewidth=2pt]{4.236273924276731}{4.561990599550362}{-0.98*2.55*cos(t)+0.19*1.49*sin(t)+1.63|-0.19*2.55*cos(t)+-0.98*1.49*sin(t)+-0.19}
\parametricplot[linewidth=2pt]{4.561990599550362}{4.941608992984511}{-0.98*2.55*cos(t)+0.19*1.49*sin(t)+1.63|-0.19*2.55*cos(t)+-0.98*1.49*sin(t)+-0.19}
\parametricplot[linewidth=2pt]{3.585684166480425}{4.236273924276731}{-0.98*2.55*cos(t)+0.19*1.49*sin(t)+1.63|-0.19*2.55*cos(t)+-0.98*1.49*sin(t)+-0.19}

\psline[linewidth=2.5pt,linestyle=dashed, dash=10pt 5pt](4.99,2.55)(4.99,-2.51)
\psline[linewidth=3pt,linestyle=dotted, dotsep=3pt](4.99,-2.51)(10.05,-2.51)
\psline[linewidth=2.5pt,linestyle=dashed, dash=10pt 5pt](10.05,2.55)(10.05,-2.51)
\psline[linewidth=3pt,linestyle=dotted, dotsep=3pt](4.99,2.55)(10.05,2.55)
\psline[linewidth=2.5pt](10.05,2.55)(4.99,-2.51)
\begin{scriptsize}
\psdots[dotsize=12pt 0,dotstyle=*](-0.39,-1.91)
\psdots[dotsize=12pt 0,dotstyle=*](4.99,2.55)
\psdots[dotsize=12pt 0,dotstyle=*](4.99,-2.51)
\psdots[dotsize=12pt 0,dotstyle=*](10.05,-2.51)
\psdots[dotsize=12pt 0,dotstyle=*](10.05,2.55)
\rput(3.8,-3.8){\psscalebox{3.7 3.7}{(d)}}
\end{scriptsize}
\end{pspicture*}}
\newrgbcolor{qqwuqq}{0 0.39 0}
\psset{xunit=1.0cm,yunit=1.0cm,algebraic=true,dimen=middle,dotstyle=o,dotsize=3pt 0,linewidth=0.8pt,arrowsize=3pt 2,arrowinset=0.25}
\psscalebox{.45 .45}{\begin{pspicture*}(-3.2,-3.9)(14.95,3.4)
\parametricplot[linewidth=2.8pt]{1.0243541906660873}{1.433210851110643}{-0.45*1.76*cos(t)+0.89*1.6*sin(t)+-3.8|-0.89*1.76*cos(t)+-0.45*1.6*sin(t)+1.44}
\parametricplot[linewidth=2.8pt]{4.849974456068943}{5.258831116513499}{-0.45*1.76*cos(t)+-0.89*1.6*sin(t)+-3.8|0.89*1.76*cos(t)+-0.45*1.6*sin(t)+-1.44}
\parametricplot[linewidth=2.8pt]{1.7083818024790887}{2.1172384629236785}{-0.45*1.76*cos(t)+-0.89*1.6*sin(t)+3.8|0.89*1.76*cos(t)+-0.45*1.6*sin(t)+1.44}
\parametricplot[linewidth=2.8pt]{4.16594684425588}{4.574803504700436}{-0.45*1.76*cos(t)+0.89*1.6*sin(t)+3.8|-0.89*1.76*cos(t)+-0.45*1.6*sin(t)+-1.44}
\parametricplot[linewidth=2.8pt]{5.40522191192191}{5.627886470648845}{0*4.15*cos(t)+-1*3.25*sin(t)+0|1*4.15*cos(t)+0*3.25*sin(t)+-2.15}
\parametricplot[linewidth=2.8pt]{-0.6552988365307408}{0.0}{0*4.15*cos(t)+-1*3.25*sin(t)+0|1*4.15*cos(t)+0*3.25*sin(t)+-2.15}
\parametricplot[linewidth=2.8pt]{0.0}{0.6552988365307406}{0*4.15*cos(t)+-1*3.25*sin(t)+0|1*4.15*cos(t)+0*3.25*sin(t)+-2.15}
\parametricplot[linewidth=2.8pt]{0.6552988365307406}{0.8779633952576767}{0*4.15*cos(t)+-1*3.25*sin(t)+0|1*4.15*cos(t)+0*3.25*sin(t)+-2.15}
\parametricplot[linewidth=2.8pt]{2.2636292583321147}{2.4862938170590523}{0*4.15*cos(t)+-1*3.25*sin(t)+0|1*4.15*cos(t)+0*3.25*sin(t)+2.15}
\parametricplot[linewidth=2.8pt]{2.4862938170590523}{3.141592653589793}{0*4.15*cos(t)+-1*3.25*sin(t)+0|1*4.15*cos(t)+0*3.25*sin(t)+2.15}
\parametricplot[linewidth=2.8pt]{3.141592653589793}{3.796891490120534}{0*4.15*cos(t)+-1*3.25*sin(t)+0|1*4.15*cos(t)+0*3.25*sin(t)+2.15}
\parametricplot[linewidth=2.8pt]{3.796891490120534}{4.01955604884747}{0*4.15*cos(t)+-1*3.25*sin(t)+0|1*4.15*cos(t)+0*3.25*sin(t)+2.15}
\parametricplot[linewidth=2pt,linestyle=dashed, dash=3pt 3pt 6pt 3pt 6pt 3pt 6pt 3pt 6pt 3pt]
{0.7609531576085978}{1.174256684761854}{1*4.14*cos(t)+0*2.58*sin(t)+0|0*4.14*cos(t)+1*2.58*sin(t)+-1.78}
\parametricplot[linewidth=2pt]
{1.5707963267948966}{1.9673359688279375}{1*4.14*cos(t)+0*2.58*sin(t)+0|0*4.14*cos(t)+1*2.58*sin(t)+-1.78}
\parametricplot[linewidth=2pt]{1.9673359688279375}{2.3806394959811943}{1*4.14*cos(t)+0*2.58*sin(t)+0|0*4.14*cos(t)+1*2.58*sin(t)+-1.78}
\parametricplot[linewidth=1pt, doubleline=true]{3.902545811198392}{4.315849338351649}{1*4.14*cos(t)+0*2.58*sin(t)+0|0*4.14*cos(t)+1*2.58*sin(t)+1.78}
\parametricplot[linewidth=1pt, doubleline=true]{4.315849338351649}{4.71238898038469}{1*4.14*cos(t)+0*2.58*sin(t)+0|0*4.14*cos(t)+1*2.58*sin(t)+1.78}
\parametricplot[linewidth=1pt, doubleline=true]{4.71238898038469}{5.108928622417732}{1*4.14*cos(t)+0*2.58*sin(t)+0|0*4.14*cos(t)+1*2.58*sin(t)+1.78}
\parametricplot[linewidth=1pt, doubleline=true]{5.108928622417732}{5.522232149570988}{1*4.14*cos(t)+0*2.58*sin(t)+0|0*4.14*cos(t)+1*2.58*sin(t)+1.78}
\parametricplot[linewidth=2pt,linestyle=dashed, dash=3pt 3pt 6pt 3pt 6pt]
{1.174256684761854}{1.397143184659889}{1*4.14*cos(t)+0*2.58*sin(t)+0|0*4.14*cos(t)+1*2.58*sin(t)+-1.78}
\parametricplot[linewidth=2pt]{1.397143184659889}{1.5707963267948966}{1*4.14*cos(t)+0*2.58*sin(t)+0|0*4.14*cos(t)+1*2.58*sin(t)+-1.78}
\psdots[dotsize=9pt 0,dotstyle=*](-3,0)
\psdots[dotsize=9pt 0,dotstyle=o,linecolor=black,fillcolor=gray](3,0)
\psdots[dotsize=9pt 0,dotstyle=o,linecolor=black,fillcolor=white](0.72,0.76)
\psscalebox{1.30 1.30}{
\psline[linewidth=1.5pt,linestyle=dashed](3.7,1.54)(3.7,-1.54)
\psline[linewidth=1.5pt](6.78,-1.54)(6.78,1.54)
\psline[linewidth=1.5pt](3.7,1.54)(6.78,1.54)
\psline[linewidth=1.5pt,linestyle=dashed](3.7,-1.54)(6.78,-1.54)
\psline[linewidth=1pt,doubleline=true, doublesep=0.8pt](6.78,1.54)(3.7,-1.54)
\begin{scriptsize}
\psdots[dotsize=7pt 0,dotstyle=o,linecolor=black,fillcolor=white](3.7,1.54)
\psdots[dotsize=7pt 0,dotstyle=o,linecolor=black,fillcolor=gray](3.7,-1.54)
\psdots[dotsize=7pt 0,dotstyle=o,linecolor=black,fillcolor=white](6.78,-1.54)
\psdots[dotsize=7pt 0,dotstyle=*](6.78,1.54)
\end{scriptsize}}
\rput(3.8,-2.9){\psscalebox{2.3 2.3}{(e)}}
\end{pspicture*}}

\begin{prop}\label{surface}
The geometric realisation of any $\mathcal{M}$-complex is a closed, connected, orientable surface.
\end{prop}

\begin{proof}
Let $L$ be an $\mathcal{M}$-complex. By (1) and (3) it is evident
that $|L|$ is a manifold, locally homeomorphic to $\mathbf{R}^2$,
except at the realisations of 0-cells. Let $w\in L_0$, and let
$L_w$ be defined as above. If $u\in L_w$ and $w$ is its $i$th
vertex, then let $\Delta^2_u$ be the intersection of $\Delta^2$
and the open halfspace $t_i>1/2$. For $\sim$ obtained by
restricting the equivalence relation that defines geometric
realisation, we claim that
\[
U_w=\left(\coprod_{u\in L_w} \{u\}\times\Delta^2_u\right){\Big\slash}\sim,
\]
is an open neighbourhood of the realisation of $w$, which is homeomorphic to $\mathbf{R}^2$.

Let $u\in L_w$ and let $y$ and $y'$ be the edges of $u$ having $w$ as a vertex. Let $u'$ be the 2-cell sharing $y'$ with $u$, and let $y''$ be the second edge of $u'$ having $w$ as a vertex. If $y=y''$, then $L_w=\{u,u'\}$, otherwise, $L_w$ cannot be linked. If $y\neq y''$, then let $u''$ be the 2-cell sharing $y''$ with $u'$, and let $y'''$ be the second edge of $u''$ having $w$ as a vertex. We have that $y'''\neq y',y''$ and if $y=y'''$, then $L_w=\{u,u',u''\}$ by the same reasons as above. If $y\neq y'''$, then it is evident how to proceed with this listing of $L_w$ until we reach $L_w=\{u,u',u'',\ldots,u^{(k)}\}$ such that every two consecutive members, as well as $u$ and $u^{(k)}$, are $w$-neighbours. Then $U_w$ is an open disc triangulated in $k+1$ triangles. By (5), we have that $|L|$ is orientable. Since $L$ is an $\mathcal{M}$-complex, it is connected. Hence, $|L|$ is connected and it has just two possible orientations.
\end{proof}

\section{Permutations and switching of triangles}\label{permutations}
All the entailments that we have, up to now, are of the form: conclude a Menelaus configuration from several such configurations. Is the above ``surface'' reasoning the only one of such a form? We show that there are some other, quite elementary, reasonings with Menelaus configurations, which keep this form.

\begin{rem}\label{permutation}
If $(A_1,A_2,A_3,B_1,B_2,B_3)$ makes a Menelaus configuration and $\pi$ is a permutation of the set $\{1,2,3\}$, then it is easy to check that \[(A_{\pi(1)},A_{\pi(2)},A_{\pi(3)},B_{\pi(1)},B_{\pi(2)},B_{\pi(3)})\] makes a Menelaus configuration, too.
\end{rem}

\begin{rem}\label{switching}
If $(A,B,C,P,Q,R)$ makes a Menelaus configuration, then the sextuples $(B,P,R,Q,A,C)$, $(A,R,Q,P,C,B)$ and $(C,P,Q,R,A,B)$ make Menelaus configurations, too.
\end{rem}

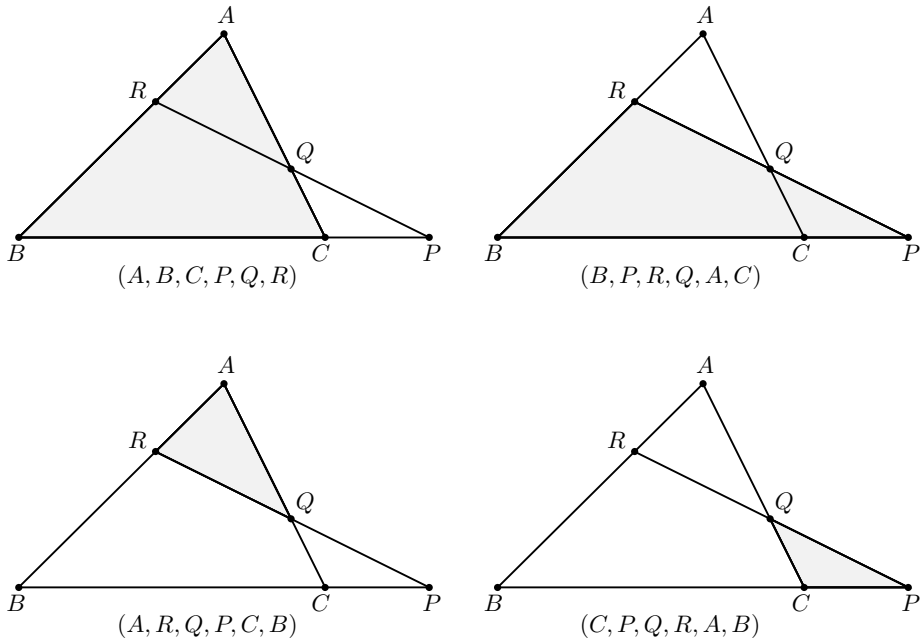
\begin{center}
\psset{xunit=1.0cm,yunit=1.0cm,algebraic=true,dimen=middle,dotstyle=o,dotsize=3pt 0,linewidth=0.8pt,arrowsize=3pt 2,arrowinset=0.25}
\psscalebox{.9 .9}{\begin{pspicture*}(-3.5,-1.0)(16.84,3.5)
\pspolygon[fillcolor=gray,fillstyle=solid,opacity=0.1](-3,0)(0,3)(1.48,0)
\pspolygon[fillcolor=gray,fillstyle=solid,opacity=0.1](4,0)(6,2)(10,0)
\psline(-3,0)(0,3)
\psline(1.48,0)(0,3)
\psline(-3,0)(1.48,0)
\psline(-1,2)(3,0)
\psline(4,0)(7,3)
\psline(7,3)(8.48,0)
\psline(4,0)(8.48,0)
\psline(8.48,0)(10,0)
\psline(6,2)(10,0)
\psline(1.48,0)(3,0)
\psline(-3,0)(0,3)
\psline(0,3)(1.48,0)
\psline(1.48,0)(-3,0)
\psline(4,0)(6,2)
\psline(6,2)(10,0)
\psline(10,0)(4,0)
\rput[tl](-0.1,3.4){$A$}
\rput[tl](-3.18,-0.1){$B$}
\rput[tl](1.28,-0.1){$C$}
\rput[tl](1.05,1.4){$Q$}
\rput[tl](2.9,-0.1){$P$}
\rput[tl](-1.4,2.3){$R$}
\rput[tl](6.9,3.4){$A$}
\rput[tl](3.8,-0.1){$B$}
\rput[tl](8.3,-0.1){$C$}
\rput[tl](8.05,1.4){$Q$}
\rput[tl](5.6,2.3){$R$}
\rput[tl](9.9,-0.1){$P$}
\rput[tl](-1.56,-0.4){$(A,B,C,P,Q,R)$}
\rput[tl](5.2,-0.4){$(B,P,R,Q,A,C)$}
\begin{scriptsize}
\psdots[dotstyle=*](-3,0)
\psdots[dotstyle=*](0,3)
\psdots[dotstyle=*](1.48,0)
\psdots[dotstyle=*](-1,2)
\psdots[dotstyle=*](3,0)
\psdots[dotstyle=*](0.98,1.01)
\psdots[dotstyle=*](4,0)
\psdots[dotstyle=*](7,3)
\psdots[dotstyle=*](8.48,0)
\psdots[dotstyle=*](10,0)
\psdots[dotstyle=*](6,2)
\psdots[dotstyle=*](7.98,1.01)
\end{scriptsize}
\end{pspicture*}}
\end{center}

\begin{center}
\psset{xunit=1.0cm,yunit=1.0cm,algebraic=true,dimen=middle,dotstyle=o,dotsize=3pt 0,linewidth=0.8pt,arrowsize=3pt 2,arrowinset=0.25}
\psscalebox{.9 .9}{\begin{pspicture*}(-3.5,-1.0)(16.84,3.5)
\pspolygon[fillcolor=gray,fillstyle=solid,opacity=0.1](-1,2)(0,3)(0.98,1.01)
\pspolygon[fillcolor=gray,fillstyle=solid,opacity=0.1](7.98,1.01)(8.48,0)(10,0)
\psline(-3,0)(0,3)
\psline(1.48,0)(0,3)
\psline(-3,0)(1.48,0)
\psline(-1,2)(3,0)
\psline(4,0)(7,3)
\psline(7,3)(8.48,0)
\psline(4,0)(8.48,0)
\psline(8.48,0)(10,0)
\psline(6,2)(10,0)
\psline(1.48,0)(3,0)
\rput[tl](-0.1,3.4){$A$}
\rput[tl](-3.18,-0.1){$B$}
\rput[tl](1.28,-0.1){$C$}
\rput[tl](1.05,1.4){$Q$}
\rput[tl](2.9,-0.1){$P$}
\rput[tl](-1.4,2.3){$R$}
\rput[tl](6.9,3.4){$A$}
\rput[tl](3.8,-0.1){$B$}
\rput[tl](8.3,-0.1){$C$}
\rput[tl](8.05,1.4){$Q$}
\rput[tl](5.6,2.3){$R$}
\rput[tl](9.9,-0.1){$P$}
\rput[tl](-1.56,-0.4){$(A,R,Q,P,C,B)$}
\rput[tl](5.2,-0.4){$(C,P,Q,R,A,B)$}
\psline(-1,2)(0,3)
\psline(0,3)(0.98,1.01)
\psline(0.98,1.01)(-1,2)
\psline(7.98,1.01)(8.48,0)
\psline(8.48,0)(10,0)
\psline(10,0)(7.98,1.01)
\begin{scriptsize}
\psdots[dotstyle=*](-3,0)
\psdots[dotstyle=*](0,3)
\psdots[dotstyle=*](1.48,0)
\psdots[dotstyle=*](-1,2)
\psdots[dotstyle=*](3,0)
\psdots[dotstyle=*](0.98,1.01)
\psdots[dotstyle=*](4,0)
\psdots[dotstyle=*](7,3)
\psdots[dotstyle=*](8.48,0)
\psdots[dotstyle=*](10,0)
\psdots[dotstyle=*](6,2)
\psdots[dotstyle=*](7.98,1.01)
\end{scriptsize}
\end{pspicture*}}
\end{center}

\begin{proof}
If $A$, $B$ and $C$ are not collinear, then this follows by using the Menelaus theorem in both
directions. If $A$, $B$ and $C$ are collinear, then let $A'$ be a
point outside the line $BC$ (for example, $A'A$ is perpendicular
to $BC$, as below) and let $Q'\in A'C$ and $R'\in A'B$ be such
that $A'A\parallel Q'Q\parallel R'R$. Then it remains to apply the
Thales theorem, and the Menelaus theorem in both directions.
\end{proof}

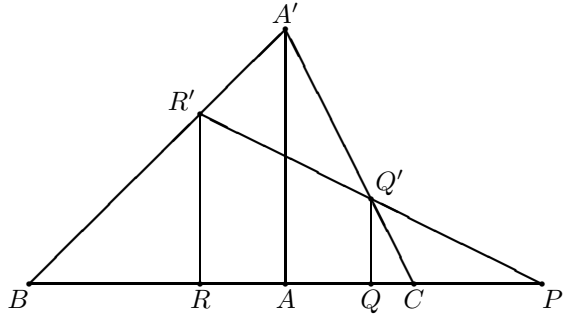
\begin{center}
\begin{picture}(200,120)(0,-17)
\unitlength.8pt
{\thicklines
\put(0,0){\line(1,0){180}} \put(180,0){\line(1,0){60}}
\put(0,0){\line(1,1){80}} \put(80,80){\line(1,1){40}}
\put(180,0){\line(-1,2){20}} \put(160,40){\line(-1,2){40}}
\put(240,0){\line(-2,1){80}} \put(160,40){\line(-2,1){80}}}

{\thinlines \put(80,0){\line(0,1){80}} \put(120,0){\line(0,1){120}} \put(160,0){\line(0,1){40}}
}

\put(0,0){\circle*{3}} \put(180,0){\circle*{3}}
\put(240,0){\circle*{3}} \put(80,80){\circle*{3}}
\put(120,120){\circle*{3}} \put(160,40){\circle*{3}} \put(120,0){\circle*{3}} \put(80,0){\circle*{3}} \put(160,0){\circle*{3}}

\put(0,-3){\makebox(0,0)[tr]{$B$}}
\put(180,-3){\makebox(0,0)[t]{$C$}}
\put(240,-3){\makebox(0,0)[tl]{$P$}}
\put(78,82){\makebox(0,0)[br]{$R'$}}
\put(120,123){\makebox(0,0)[b]{$A'$}}
\put(162,42){\makebox(0,0)[bl]{$Q'$}} \put(80,-3){\makebox(0,0)[t]{$R$}} \put(120,-3){\makebox(0,0)[t]{$A$}} \put(160,-3){\makebox(0,0)[t]{$Q$}}

\end{picture}
\end{center}

The symmetric group $S_6$, which acts naturally on sextuples,
contains a subgroup $G$ generated by the permutations
$s=(123)(456)$ and $t=(26)(35)$. The group $G$ is of order 24 and it is isomorphic to the \emph{octahedral group} (also isomorphic to $S_4$), presented by $\langle s,t\mid s^3,t^2,(st)^4\rangle$. By Remarks \ref{permutation} and \ref{switching}, it follows that if a sextuple makes a Menelaus configuration, then every sextuple from its $G$-orbit makes a Menelaus configuration, too.

\section{The Menelaus system}\label{system}
The aim of this section is to introduce a one-sided sequent
system, which deals with propositions of the form ``this sextuple
of points makes a Menelaus configuration''. An intuition behind
the sequents of our system is that an arbitrary formula in a
sequent is entailed by the remaining formulae of the sequent.

Probably the most prominent one-sided sequent system is the system
for the multiplicative fragment of \textit{linear logic},
introduced by Girard \cite{girard87}. It consists of one axiom
scheme $\vdash \varphi^\bot,\varphi$, where $\varphi^\bot$ denotes
the \emph{linear negation} of $\varphi$, two structural rules: cut
and exchange
\[
\f{\vdash\Gamma,\varphi \quad \vdash\Delta,\varphi^\bot}{\vdash\Gamma,\Delta} \quad\quad\quad \f{\vdash\Gamma,\varphi,\psi,\Delta}{\vdash\Gamma,\psi,\varphi,\Delta},
\]
and two rules for connectives
\[
\f{\vdash\Gamma,\varphi\quad \vdash\Delta,\psi}{\vdash\Gamma,\Delta,\varphi\otimes \psi} \quad\quad\quad \f{\Gamma,\varphi,\psi}{\vdash\Gamma, \varphi\wp \psi}.
\]

An alternative formal system for the same fragment of linear logic is the system of \textit{proof nets}.
There are many criteria of correctness to ensure that a given derivation is actually a proof net. One geometric criterion that differs from but is comparable to what we have mentioned about reasoning with triangulations of surfaces in Section~\ref{configurations} is the \textit{acyclic-connected correctness criterion} of Danos and Regnier, \cite{danos-regnier}.

Some further examples of one-sided sequent systems are the Gentzen-Sch\" utte-Tait system (see \cite{schutte} and \cite{tait}), which is a standard one-sided formulation of the propositional fragment of Gentzen's classical sequent calculus, and the one-sided sequent system called \textit{minimal sequent calculus} introduced in \cite{hughes10}.

In order to build the formal language, we introduce the following
set whose elements take the role of atomic formulae. For an
arbitrary countable set $\mathcal{W}$, let
\[
F^6(\mathcal{W})=\mathcal{W}^6-\{(X_1,\ldots,X_6)\in
\mathcal{W}^6\mid X_i=X_j\;\mbox{\rm for some}\; i\neq j\}.
\]

The \emph{atomic formulae} of our language are the elements of
$F^6(\mathcal{W})$. (It would be more convenient to write a
predicate symbol in front of a sextuple, but since we deal with
one predicate only, we will use no symbol for it.) The
\emph{formulae} are built out of atomic formulae by using the
connective $\diskon$, which plays the role of conjunction and
disjunction, simultaneously, and $\leftrightarrow$, which plays
the role of two implications, simultaneously. The metavariables we
use for formulae are $\varphi,\psi,\theta,\ldots$, possibly with
indices. A \emph{sequent} is a finite multiset of formulae, and
the sequent consisting of a multiset $\Gamma$ is denoted by
$\vdash \Gamma$.

The axiomatic sequents are formed in the following manner. For
every $\mathcal{M}$-complex $L$ such that $L_0\cup L_1\subseteq
\mathcal{W}$, let $\nu\colon L_2\to F^6(\mathcal{W})$ be defined
as

\[\nu x=(d_1d_2 x, d_0d_2 x, d_0d_0 x, d_0 x, d_1 x, d_2 x)\]
 (cf.\ the definition of the operator $\mu$ in
Section~\ref{configurations}).

The set of axiomatic sequents includes the sequents of the form
\begin{equation}\label{ax-com}
\vdash\{\nu x\mid x\in L_2\},
\end{equation}
for every $\mathcal{M}$-complex $L$ whose 0 and 1-cells belong to
$\mathcal{W}$. For example, let $L$ be the sphere $S^2$, together
with the $\mathcal{M}$-complex structure given by two 2-cells
having the same boundaries (see Section~\ref{configurations}, Example~\ref{Mcomplex}~(e)).
If the 0-cells are $A$, $B$, $C$, and the corresponding 1-cells are $P$, $Q$, $R$ respectively, then
\begin{equation}\label{identity}
\vdash (A,B,C,P,Q,R), (A,B,C,P,Q,R)
\end{equation}
is an axiomatic sequent playing the role of the \emph{identity} derivation.

Moreover, we have the following two axiom schemata corresponding to permutations of vertices and switching of triangles (cf.\ Section~\ref{permutations}).
\begin{equation}\label{perm-swit}
\begin{array}l
\vdash (A,B,C,P,Q,R), (B,C,A,Q,R,P),
\\[.5ex]
\vdash(A,B,C,P,Q,R), (A,R,Q,P,C,B).
\end{array}
\end{equation}

The rules of inference of the system are introduced as follows. Besides the \emph{cut} rules:
\begin{equation}\label{rule-cut}
\f{\vdash\Gamma,\varphi \quad \vdash\Delta,\varphi}{\vdash\Gamma,\Delta} \quad\quad\quad \f{\vdash\Gamma \quad \vdash\Delta}{\vdash\Gamma,\Delta}
\end{equation}
there are the following inference rules of $\diskon$-\emph{introduction} and $\leftrightarrow$-\emph{introduction}:
\begin{equation}\label{rule-inf}
\f{\vdash\Gamma,\varphi \quad
\vdash\Gamma,\psi}{\vdash\Gamma,\varphi\diskon\psi}
\quad\quad\quad \f{\vdash\Gamma,\varphi \quad
\vdash\Delta,\psi}{\vdash\Gamma,\Delta,\varphi\leftrightarrow\psi}.
\end{equation}
The formula $\varphi$ in the first cut rule is called the
\emph{cut formula}, and we refer to the second cut rule as the one
``whose cut formula is empty''. The second cut rule enables us to
take into account just connected complexes while creating the
family of $\mathcal{M}$-complexes.

 Analogously, the first cut rule enables us to build the axiomatic sequents not with respect to arbitrary $\mathcal{M}$-complexes, but to restrict this family to those not expressible as \emph{connected sums} of two simpler complexes (see Section \ref{sneg}, and in particular Proposition \ref{connected-cut}, see also Example~\ref{example-3}  in Section~\ref{reading} and  Example \ref{counterexample} in Section  \ref{long-normal-form}). We do not take advantage of this opportunity in the present paper.

Note that our connective $\diskon$ corresponds to the \textit{additive} connective $\&$ in linear logic (cf.\ \cite[Section~1.15]{girard87}). On the other hand, our $\leftrightarrow$ corresponds to the \textit{multiplicative} connective $\otimes$ (cf.\ the beginning of this section). The only difference between the cut rules is that, in linear logic, the cut formula in the right premise is linearly negated. Informally, in the Menelaus system, a formula $\varphi$ coincides with its linear negation $\varphi^{\bot}$.

\vspace{1ex}

\begin{rem} \label{less-syntactical}
Alternatively, we could introduce the system in a slightly less
syntactical manner. Instead of taking $F^6(\mathcal{W})$ as the
set of atomic formulae, we could start with the orbit set
$F^6(\mathcal{W})/G$, where $G$ is the octahedral group introduced in Section~\ref{permutations}, and omit the two axiom schemata  (\ref{perm-swit}). Similar
``strictifications'' are ubiquitous in logic, e.g.\ omitting
parentheses and neglecting the order of conjuncts in purely
conjunctive formulae. In such cases one works not with syntactical
objects but with classes of equivalence. We shall stick here  to
the non-quotiented syntax, but shall freely use implicitly the
more practical quotiented one for the examples of derivations
shown in Section \ref{reading}.
 \end{rem}

\section{The soundness}\label{secsound}
A \emph{Euclidean interpretation} is a function from $\mathcal{W}$
to $\mathbf{R}^2$, and we abuse the notation by denoting by $X$
the point which is the interpretation of $X\in \mathcal{W}$. We
say that an interpretation \emph{satisfies} the atomic formula
$(A,B,C,P,Q,R)$, when the sextuple $(A,B,C,P,Q,R)$ of points in
$\mathbf{R}^2$ makes a Menelaus configuration.

Let $\Gamma\models_E \varphi$ mean that every Euclidean interpretation that satisfies every formula in $\Gamma$ also satisfies $\varphi$, where every occurrence of $\diskon$ in $\Gamma$ is interpreted as disjunction $\vee$ and every occurrence of $\diskon$ in $\varphi$ is interpreted as conjunction $\wedge$. Concerning the connective $\leftrightarrow$, it is always interpreted as classical equivalence.

\begin{prop}[Soundness]\label{sound}
If $\vdash\Gamma,\varphi$ is derivable, then $\Gamma\models_E\varphi$.
\end{prop}

\begin{proof}
We proceed by induction on the complexity of a derivation of $\vdash\Gamma,\varphi$. Assume that an interpretation satisfies every formula in $\Gamma$.

If $\vdash\Gamma,\varphi$ is an axiomatic sequent derived from an $\mathcal{M}$-complex $L$, then we proceed as in the proof of Proposition~\ref{cycle}, with $c$ being an orientation of $L$. If $\vdash\Gamma,\varphi$ is an instance of one of the two axiomatic schemata, then we rely on Remarks~\ref{permutation} and~\ref{switching}.

If the last inference rule in the derivation of $\vdash\Gamma,\varphi$ is
\[
\f{\vdash\Gamma_1,\psi \quad \vdash\Gamma_2,\varphi,\psi}{\vdash\Gamma_1,\Gamma_2,\varphi},
\]
then by the induction hypothesis applied to $\vdash\Gamma_1,\psi$, the interpretation satisfies $\psi$. Thence, by the induction hypothesis  applied to $\vdash\Gamma_2,\varphi,\psi$, the interpretation satisfies $\varphi$. If the cut formula $\psi$ is empty, then the induction hypothesis applied to the right premise $\vdash\Gamma_2,\varphi$ is sufficient. We proceed analogously when $\varphi$ occurs in the left premise of the cut rule.

If the last inference rule is
\[
\f{\vdash\Gamma_1,\varphi,\psi \quad
\vdash\Gamma_1,\varphi,\theta}{\vdash\Gamma_1,\varphi,\psi\diskon\theta},
\]
then the interpretation that satisfies $\Gamma_1,\psi\diskon\theta$, by the definition of the interpretation of $\diskon$ in $\Gamma$, must satisfy $\Gamma_1,\psi$ or $\Gamma_1,\theta$, and by the induction hypothesis applied to the corresponding premise, one obtains that the interpretation satisfies $\varphi$ too.

If the last inference rule is
\[
\f{\vdash\Gamma,\varphi_1 \quad
\vdash\Gamma,\varphi_2}{\vdash\Gamma,\varphi_1\diskon\varphi_2},
\]
then by the induction hypothesis applied to both premisses, one obtains that both $\varphi_1$ and $\varphi_2$ are satisfied by the interpretation. Hence, $\varphi=\varphi_1\diskon\varphi_2$ is satisfied by the definition of the interpretation of $\diskon$ in $\varphi$.

If the last inference rule is
\[
\f{\vdash\Gamma_1,\psi \quad
\vdash\Gamma_2,\varphi,\theta}{\vdash\Gamma_1,\Gamma_2,\varphi, \psi\leftrightarrow\theta},
\]
then by the induction hypothesis applied to the left premise, one obtains that $\psi$ is satisfied by the interpretation, and since $\psi\leftrightarrow\theta$ is satisfied, $\theta$ must be satisfied too. Then by the induction hypothesis applied to the right premise, one obtains that $\varphi$ is satisfied. We proceed analogously when $\varphi$ occurs in the left premise of the rule.

Finally, if the last inference rule is
\[
\f{\vdash\Gamma_1,\varphi_1 \quad
\vdash\Gamma_2,\varphi_2}{\vdash\Gamma_1,\Gamma_2,\varphi_1\leftrightarrow \varphi_2},
\]
then by the induction hypothesis applied to both premisses, one obtains that both $\varphi_1$ and $\varphi_2$ are satisfied by the interpretation. Hence, $\varphi=\varphi_1\leftrightarrow \varphi_2$ is satisfied too.
\end{proof}

A multiset $\Delta$ of formulae is \emph{valid} when for every
$\delta\in \Delta$ we have that $$\Delta-\{\delta\}\models_E
\delta.$$ Proposition~\ref{sound} has the following reformulation.

\begin{prop}\label{sound2}
If $\vdash\Delta$ is derivable, then $\Delta$ is valid.
\end{prop}

The converse of Proposition~\ref{sound2}, which would be
completeness proper, however, does not hold. For a simple
counterexample consider the sequent
\begin{equation}\label{unprov1}
\vdash (A,B,P,C,X,R), (A,C,P,B,X,Q), (B,R,C,X,P,A), (A,R,C,X,Q,B)
\end{equation}
(One could take the following picture as an illustration.)
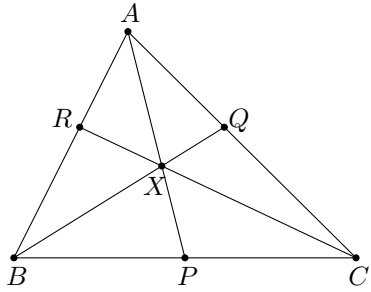
\begin{center}
\begin{tikzpicture}[scale=0.75][line cap=round,line join=round,>=triangle 45,x=1.0cm,y=1.0cm]
\draw (4.0,3.0)-- (2.0,-1.0);
\draw (4.0,3.0)-- (8.0,-1.0);
\draw (2.0,-1.0)-- (8.0,-1.0);
\draw (2.0,-1.0)-- (5.69,1.3099999999999996);
\draw (4.0,3.0)-- (5.0,-1.0);
\draw (3.1550000000000007,1.3100000000000007)-- (8.0,-1.0);
\draw (3.7,3.65) node[anchor=north west] {$A$};
\draw (1.7,-1) node[anchor=north west] {$B$};
\draw (7.7,-1) node[anchor=north west] {$C$};
\draw (4.7,-1) node[anchor=north west] {$P$};
\draw (5.6,1.8) node[anchor=north west] {$Q$};
\draw (2.5,1.8) node[anchor=north west] {$R$};
\draw (4.1,0.6) node[anchor=north west] {$X$};
\begin{scriptsize}
\draw [fill=black] (4.0,3.0) circle (1.5pt);
\draw [fill=black] (2.0,-1.0) circle (1.5pt);
\draw [fill=black] (8.0,-1.0) circle (1.5pt);
\draw [fill=black] (5.69,1.3099999999999996) circle (1.5pt);
\draw [fill=black] (5.0,-1.0) circle (1.5pt);
\draw [fill=black] (4.594024604569421,0.6239015817223198) circle (1.5pt);
\draw [fill=black] (3.1550000000000007,1.3100000000000007) circle (1.5pt);
\end{scriptsize}
\end{tikzpicture}
\end{center}
In order to prove that the sequent \ref{unprov1} is valid, assume
that an interpretation satisfies the first three formulae of that
sequent. By using the fact that for three mutually distinct collinear points $U$, $V$ and $W$
\[
(U,V;W)\cdot (V,W;U)\cdot (W,U;V)=1,
\]
it is straightforward to show that
\[
(A,R;B)\cdot (R,C;X)\cdot (C,A;Q)=-1,
\]
which means that the last formula in the sequent is satisfied by
the interpretation. We proceed analogously for the other three
cases in the proof that \ref{unprov1} is valid.

Our proof that the sequent \ref{unprov1} is not derivable, requires some results concerning decidability of the Menelaus system and it is postponed to the end of Section~\ref{secdecidability} (see Example~\ref{exunprov1}).

\section{The projective interpretation}
We define the projective interpretation and the satisfiability relation along the lines of the previous section. Then the corresponding soundness result is a corollary of Proposition~\ref{sound}.

A \emph{projective interpretation} is a function from
$\mathcal{W}$ to the \emph{projective plane} $\mathbf{RP}^2$.
Again, we denote by $X$ the point in $\mathbf{RP}^2$, which is the
interpretation of $X\in \mathcal{W}$.

We consider the points of $\mathbf{RP}^2$ as lines through the origin in $\mathbf{R}^3$. For a finite set $\mathcal{S}$ of such lines there is a plane $\alpha$ in $\mathbf{R}^3$, not containing the origin, which intersects all of them. (It is sufficient to choose a plane whose normal vector is not orthogonal to a direction vector of any line in $\mathcal{S}$, and such a plane exists since $\mathbf{R}^3$ cannot be covered by a finite number of planes.) In that case, we say that $\alpha$ \emph{intersects properly} the points from $\mathcal{S}$, and for every $A\in \mathcal{S}$ we denote the intersection of $A$ and $\alpha$ by~$A_\alpha$.

\begin{lem}
Let $\alpha$ and $\beta$ be two planes in $\mathbf{R}^3$ that intersect properly the points $A$, $B$, $C$, $P$, $Q$ and $R$ from $\mathbf{RP}^2$. If the sextuple $(A_\alpha,B_\alpha,C_\alpha,P_\alpha,Q_\alpha,R_\alpha)$ makes a Menelaus configuration, then $(A_\beta,B_\beta,C_\beta,P_\beta,Q_\beta,R_\beta)$ makes it too.
\end{lem}

\begin{proof}
If $A$, $B$ and $C$ are not collinear, then neither $A_\alpha$, $B_\alpha$ and $C_\alpha$, nor
$A_\beta$, $B_\beta$ and $C_\beta$ are collinear and since $P_\alpha$, $Q_\alpha$ and $R_\alpha$ are three collinear points lying on the lines $B_\alpha C_\alpha$, $C_\alpha A_\alpha$
and $A_\alpha B_\alpha$ respectively, we have that $P_\beta$,
$Q_\beta$ and $R_\beta$ are three collinear points lying on the lines $B_\beta C_\beta$, $C_\beta A_\beta$ and
$A_\beta B_\beta$ respectively, which means that
$(A_\beta,B_\beta,C_\beta,P_\beta,Q_\beta,R_\beta)$ makes a
Menelaus configuration.

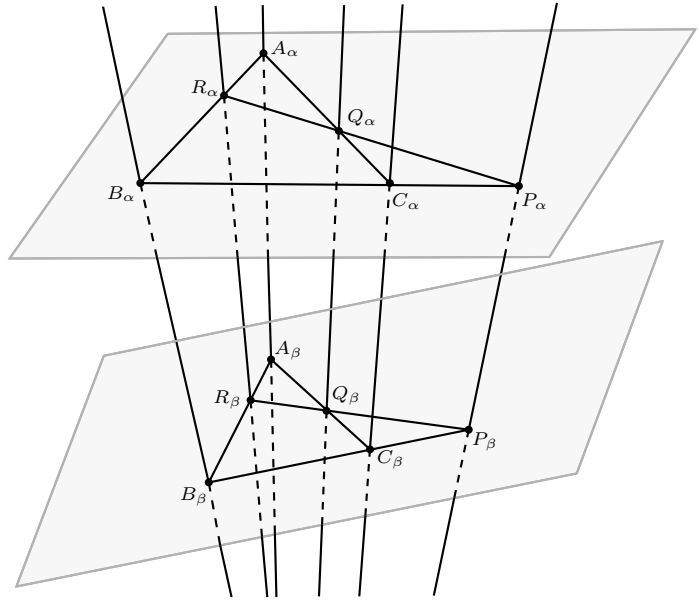
\newrgbcolor{zzttqq}{0.7 0.7 0.7}
\psset{xunit=1.0cm,yunit=1.0cm,algebraic=true,dimen=middle,dotstyle=o,dotsize=3pt 0,linewidth=0.8pt,arrowsize=3pt 2,arrowinset=0.25}
\begin{pspicture*}(0,-3.1)(21.26,5.5)
\pspolygon[linecolor=zzttqq,fillcolor=zzttqq,fillstyle=solid,opacity=0.1](3.5,4.72)(10.44,4.78)(8.52,1.76)(1.42,1.74)
\pspolygon[linecolor=zzttqq,fillcolor=zzttqq,fillstyle=solid,opacity=0.1](2.66,0.45)(10.01,1.97)(8.88,-1.11)(1.51,-2.61)
\psline[linecolor=zzttqq](3.5,4.72)(10.44,4.78)
\psline[linecolor=zzttqq](10.44,4.78)(8.52,1.76)
\psline[linecolor=zzttqq](8.52,1.76)(1.42,1.74)
\psline[linecolor=zzttqq](1.42,1.74)(3.5,4.72)
\psline(3.14,2.74)(8.12,2.7)
\psline(4.76,4.46)(3.14,2.74)
\psline(4.76,4.46)(6.42,2.74)
\psline(4.24,3.9)(8.12,2.7)
\psline(4.04,-1.23)(7.46,-0.53)
\psline(4.86,0.4)(6.16,-0.79)
\psline(4.86,0.4)(4.04,-1.23)
\psline(4.59,-0.14)(7.46,-0.53)
\psline(3.36,1.75)(4.04,-1.23)
\psline[linestyle=dashed,dash=3pt 3pt](3.14,2.74)(3.36,1.75)
\psline[linestyle=dashed,dash=3pt 3pt](4.24,3.9)(4.42,1.75)
\psline(4.42,1.75)(4.59,-0.14)
\psline[linestyle=dashed,dash=3pt 3pt](4.76,4.46)(4.82,1.75)
\psline(4.82,1.75)(4.86,0.4)
\psline[linestyle=dashed,dash=3pt 3pt](5.75,3.43)(5.68,1.75)
\psline(5.68,1.75)(5.59,-0.28)
\psline[linestyle=dashed,dash=3pt 3pt](6.42,2.74)(6.35,1.75)
\psline(6.35,1.75)(6.16,-0.79)
\psline(7.46,-0.53)(7.93,1.76)
\psline[linestyle=dashed,dash=3pt 3pt](8.12,2.7)(7.93,1.76)
\psline[linecolor=zzttqq](2.66,0.45)(10.01,1.97)
\psline[linecolor=zzttqq](10.01,1.97)(8.88,-1.11)
\psline[linecolor=zzttqq](8.88,-1.11)(1.51,-2.61)
\psline[linecolor=zzttqq](1.51,-2.61)(2.66,0.45)
\psline(4.75,5.1)(4.76,4.46)
\psline(4.13,5.09)(4.24,3.9)
\psline(2.65,5.08)(3.14,2.74)
\psline(5.82,5.1)(5.75,3.43)
\psline(6.59,5.11)(6.42,2.74)
\psline(8.62,5.13)(8.12,2.7)
\psline[linestyle=dashed,dash=3pt 3pt](4.04,-1.23)(4.19,-2.06)
\psline[linestyle=dashed,dash=3pt 3pt](4.59,-0.14)(4.74,-1.95)
\psline[linestyle=dashed,dash=3pt 3pt](4.86,0.4)(4.91,-1.91)
\psline[linestyle=dashed,dash=3pt 3pt](5.59,-0.28)(5.53,-1.79)
\psline[linestyle=dashed,dash=3pt 3pt](6.16,-0.79)(6.1,-1.67)
\psline[linestyle=dashed,dash=3pt 3pt](7.46,-0.53)(7.27,-1.43)
\psline(7.27,-1.43)(7,-2.73)
\psline(6.1,-1.67)(6.02,-2.74)
\psline(5.53,-1.79)(5.49,-2.74)
\psline(4.19,-2.06)(4.34,-2.75)
\psline(4.74,-1.95)(4.81,-2.75)
\psline(4.91,-1.91)(4.93,-2.75)
\begin{scriptsize}
\psdots[dotstyle=*](3.14,2.74)
\rput[bl](2.7,2.5){$B_\alpha$}
\psdots[dotstyle=*](8.12,2.7)
\rput[bl](8.15,2.4){$P_\alpha$}
\psdots[dotstyle=*](4.76,4.46)
\rput[bl](4.85,4.42){$A_\alpha$}
\psdots[dotstyle=*](6.42,2.74)
\rput[bl](6.45,2.4){$C_\alpha$}
\psdots[dotstyle=*](4.24,3.9)
\rput[bl](3.8,3.9){$R_\alpha$}
\psdots[dotstyle=*](5.75,3.43)
\rput[bl](5.86,3.5){$Q_\alpha$}
\psdots[dotstyle=*](4.04,-1.23)
\rput[bl](3.66,-1.53){$B_\beta$}
\psdots[dotstyle=*](4.86,0.4)
\rput[bl](4.9,0.4){$A_\beta$}
\psdots[dotstyle=*](7.46,-0.53)
\rput[bl](7.5,-0.8){$P_\beta$}
\psdots[dotstyle=*](6.16,-0.79)
\rput[bl](6.25,-1.05){$C_\beta$}
\psdots[dotstyle=*](4.59,-0.14)
\rput[bl](4.1,-0.25){$R_\beta$}
\psdots[dotstyle=*](5.59,-0.28)
\rput[bl](5.66,-0.2){$Q_\beta$}
\end{scriptsize}
\end{pspicture*}

If $A$, $B$ and $C$ are collinear, then for the plane $\beta_0$ such that $0\in\beta_0$ and
$\beta\parallel \beta_0$, denote by $b$ the intersection of
$\alpha$ and $\beta_0$. Hence $b$ is either a line in $\alpha$ or
the empty set. In the sequel we assume that all the chosen points
are outside $b$. Let $D_\alpha$ be a point in $\alpha$ outside the
line $A_\alpha B_\alpha$, and let $U_\alpha$ be a point on the
segment $A_\alpha D_\alpha$ such that the lines $R_\alpha
U_\alpha$ and $B_\alpha D_\alpha$, as well as the lines $Q_\alpha
U_\alpha$ and $C_\alpha D_\alpha$, intersect (we can always choose
such a point $U_\alpha$). Let $\{V_\alpha\}=R_\alpha U_\alpha \cap
B_\alpha D_\alpha$ and $\{W_\alpha\}=Q_\alpha U_\alpha \cap
C_\alpha D_\alpha$. Note that the sextuples $(A_\alpha, B_\alpha,
D_\alpha, V_\alpha, U_\alpha, R_\alpha)$ and $(A_\alpha, C_\alpha,
D_\alpha, W_\alpha, U_\alpha, Q_\alpha)$ make Menelaus
configurations.

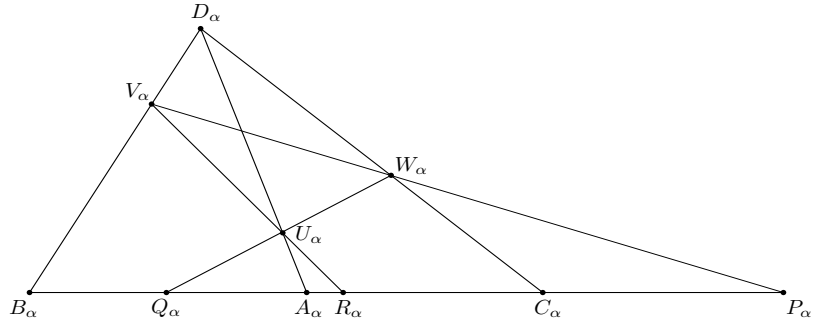
\begin{center}
\psscalebox{.75 .70} 
{
\begin{pspicture}(-4.5,-0.5)(12,5.5)

\psline[linewidth=0.02,dimen=outer](-3.0,0.0)(10.22,0.0)
\psline[linewidth=0.02,dimen=outer](-3.0,0.0)(0.0,5.0)
\psline[linewidth=0.02,dimen=outer](1.86,0.0)(0.0,5.0)
\psline[linewidth=0.02,dimen=outer](6.0,0.0)(0.0,5.0)
\psline[linewidth=0.02,dimen=outer](-0.6,0.0)(3.34,2.22)
\psline[linewidth=0.02,dimen=outer](2.5,0.0)(-0.86,3.57)
\psline[linewidth=0.01,dimen=outer](10.22,0.0)(-0.86,3.57)

\rput(-3.0,0.0){\pscircle*{.05}} \rput(6.0,0.0){\pscircle*{.05}} \rput(0.0,5.0){\pscircle*{.05}} \rput(2.5,0.0){\pscircle*{.05}} \rput(-0.6,0.0){\pscircle*{.05}} \rput(-0.86,3.57){\pscircle*{.05}} \rput(1.86,0.0){\pscircle*{.05}} \rput(1.44,1.13){\pscircle*{.05}} \rput(3.34,2.22){\pscircle*{.05}} \rput(10.22,0.0){\pscircle*{.05}}

\rput(-3.1,-0.3){\psscalebox{1.1}{$B_\alpha$}}
\rput(6.1,-0.3){\psscalebox{1.1}{$C_\alpha$}}
\rput(0.1,5.3){\psscalebox{1.1}{$D_\alpha$}}
\rput(-0.6,-0.3){\psscalebox{1.1}{$Q_\alpha$}}
\rput(2.6,-0.3){\psscalebox{1.1}{$R_\alpha$}}
\rput(-1.1,3.8){\psscalebox{1.1}{$V_\alpha$}}
\rput(1.9,-0.3){\psscalebox{1.1}{$A_\alpha$}}
\rput(1.9,1.1){\psscalebox{1.1}{$U_\alpha$}}
\rput(3.7,2.4){\psscalebox{1.1}{$W_\alpha$}}
\rput(10.5,-0.3){\psscalebox{1.1}{$P_\alpha$}}

\end{pspicture}
}
\end{center}

Consider the following axiomatic sequent derived from a tetrahedral triangulation of the sphere $S^2$ (cf.\ Section~\ref{reading}, Example~1)
\[
\vdash (A,B,C,P,Q,R), (A,B,D,V,U,R), (A,C,D,W,U,Q), (B,C,D,W,V,P),
\]
and an Euclidean interpretation that interprets $X\in\{A,B,C,D,P,Q,R,U,V,W\}$ as $X_\alpha$. This interpretation satisfies the first three sextuples, and by Proposition~\ref{sound} the sextuple $(B_\alpha,C_\alpha,D_\alpha,W_\alpha,V_\alpha,P_\alpha)$ makes a Menelaus configuration.

By appealing to the case of non-collinear $A$, $B$ and $C$, we know that an Euclidean
interpretation that interprets $X\in\{A,B,C,D,P,Q,R,U,V,W\}$ as
$X_\beta$ satisfies the last three sextuples, and by Proposition~\ref{sound} the sextuple
$(A_\beta,B_\beta,C_\beta,P_\beta,Q_\beta,R_\beta)$ makes a
Menelaus configuration.
\end{proof}

After this lemma, we can say that a sextuple $(A,B,C,P,Q,R)$ of points in $\mathbf{RP}^2$ \emph{makes a Menelaus configuration} when for some (or every) plane $\alpha$ in $\mathbf{R}^3$, which intersects properly all these points, the sextuple $(A_\alpha,B_\alpha,C_\alpha,P_\alpha,Q_\alpha,R_\alpha)$ makes a Menelaus configuration in the Euclidean sense. We say that a projective
interpretation \emph{satisfies} the atomic formula
$(A,B,C,P,Q,R)$, when the sextuple $(A,B,C,P,Q,R)$ of points in
$\mathbf{RP}^2$ makes a Menelaus configuration.

Let $\Gamma\models_P \varphi$ mean that every projective interpretation that satisfies every formula in $\Gamma$ also satisfies $\varphi$. As a corollary of Proposition~\ref{sound} we have the following result.

\begin{prop}[Projective Soundness]\label{prsound}
If $\vdash\Gamma,\varphi$ is derivable, then $\Gamma\models_P\varphi$.
\end{prop}

\section{From derivable sequents to incidence results}\label{reading}
A general pattern for extracting an incidence result (its formulation and a proof) from derivable sequents of our system is the following. One has to use interpretations that satisfy all but one formulae in a derivable sequent. By the soundness result, such an interpretation satisfies the last formula too. For some results it is enough to use just one sequent and one interpretation. On the other hand, some results require several interpretations and one derivable sequent, and perhaps, there are cases when several derivable sequents are involved in one incidence result.

Our general position is such that we treat the Menelaus system as
the syntax and the projective plane as the semantics. In order to
use a provable sequent, one has to choose one of its formulae that
should be treated as a consequence of the others. However, for
proving that the other formulae of the sequent are satisfied by
the assumptions used in the statement of the incidence result one
cannot rely exclusively on the Menelaus system. This system does
not deal with some negative statements (expressing that some
points are not collinear or that they are mutually distinct) or some positive statements
(expressing that some lines coincide), which are left to (the
axioms of) projective geometry. Hence, the Menelaus system is
somehow treated as an assistant for a larger system staying behind
the projective geometry.

In the examples below, we shall use the quotiented system (cf.
discussion in Remark \ref{less-syntactical}),  where no
distinction is made between a sextuple of elements of
$\mathcal{W}$ and any other member of its $G$-orbit.

\setcounter{example}{0}

\begin{example}\label{example-1} This example shows how the Desargues theorem could be derived by using one sequent obtained by $\diskon$-introduction and two interpretations.

\begin{thm}[Desargues] Let $ABC$ and $UVW$ be two triangles in $\mathbf{RP}^2$ such that $A\neq U$, $B\neq V$ and $C\neq W$. Let $BC\cap VW=\{P\}$, $AC\cap UW=\{Q\}$ and $AB\cap UV=\{R\}$. Then the lines $AU$, $BV$ and $CW$ are concurrent iff the points $P$, $Q$ and $R$ are
collinear.
\end{thm}

\begin{proof} For the \emph{only if} part of the theorem, one can use the following tetrahedral triangulation of~$S^2$ (see \cite[Section~3.2, second row, first scheme]{RG06}),

\begin{center}
\psset{xunit=1.0cm,yunit=1.0cm,algebraic=true,dimen=middle,dotstyle=o,dotsize=3pt 0,linewidth=0.8pt,arrowsize=3pt 2,arrowinset=0.25}
\begin{pspicture*}(-3.3,-2.4)(16.64,3.5)
\pspolygon[fillcolor=gray,fillstyle=solid,opacity=0.1](0.02,0.32)(2.66,3.06)(2.78,-1.32)
\pspolygon[fillcolor=lightgray,fillstyle=solid,opacity=0.1](2.66,3.06)(2.78,-1.32)(6.08,0.34)
\psline[linestyle=dashed,dash=2pt 2pt](0.02,0.32)(6.08,0.34)
\psline(0.02,0.32)(2.78,-1.32)
\psline(6.08,0.34)(2.78,-1.32)
\psline(2.66,3.06)(0.02,0.32)
\psline(2.66,3.06)(6.08,0.34)
\psline(2.66,3.06)(2.78,-1.32)
\psline(2.78,-1.32)(0.02,0.32)
\psline(2.78,-1.32)(6.08,0.34)
\psline(6.08,0.34)(2.66,3.06)
\psline(2.78,-1.32)(4.01,-2.04)
\psline(4.01,-2.04)(1.4,1.76)
\psline[linestyle=dashed,dash=2pt 2pt](1.4,1.76)(5.58,0.74)
\psline(5.58,0.74)(7.04,0.34)

\psline(4.01,-2.04)(8.94,1.84)
\psline(2.75,-0.18)(8.94,1.84)
\psline(6.08,0.34)(7.04,0.34)
\psline(6.08,0.34)(8.94,1.84)
\rput[tl](-0.25,0.2){$A$}
\rput[tl](2.58,-1.45){$B$}
\rput[tl](5.9,0.2){$C$}
\rput[tl](8.85,2.2){$P$}
\rput[tl](7.02,0.2){$Q$}
\rput[tl](3.86,-2.15){$R$}
\rput[tl](1.1,2.1){$U$}
\rput[tl](2.34,-0.2){$V$}
\rput[tl](5.5,1.15){$W$}
\rput[tl](2.5,3.4){$D$}
\begin{scriptsize}
\psdots[dotstyle=*](0.02,0.32)
\psdots[dotstyle=*](6.08,0.34)
\psdots[dotstyle=*](2.78,-1.32)
\psdots[dotstyle=*](2.66,3.06)
\psdots[dotstyle=*](4.01,-2.04)
\psdots[dotstyle=*](1.4,1.76)
\psdots[dotstyle=*](7.04,0.34)
\psdots[dotstyle=*](2.75,-0.18)
\psdots[dotstyle=*](5.58,0.74)
\psdots[dotstyle=*](8.94,1.84)
\end{scriptsize}
\end{pspicture*}
\end{center}
which delivers the axiomatic sequent
\begin{equation}\label{des1}
\vdash (A,B,D,V,U,R), (B,C,D,W,V,P), (A,C,D,W,U,Q), (A,B,C,P,Q,R).
\end{equation}
Consider a projective interpretation that maps $A,B,C,U,V,W,P,Q,R$ as indicated and $D$ is interpreted as the common point of  $AU$, $BV$ and $CW$. This interpretation satisfies
\[
(A,B,D,V,U,R), (B,C,D,W,V,P), \; \mbox{\rm and}\; (A,C,D,W,U,Q).
\]
By Proposition~7.2 applied to the sequent (\ref{des1}), this interpretation satisfies \linebreak $(A,B,C,P,Q,R)$. Therefore, $P$, $Q$ and $R$ are collinear.

For the \emph{if} part of the theorem, we proceed analogously with the following tetrahedral triangulation of $S^2$,
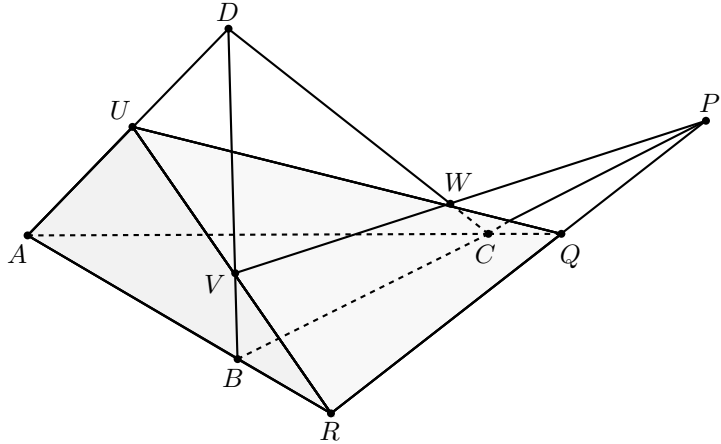
\begin{center}
\psset{xunit=1.0cm,yunit=1.0cm,algebraic=true,dimen=middle,dotstyle=o,dotsize=3pt 0,linewidth=0.8pt,arrowsize=3pt 2,arrowinset=0.25}
\begin{pspicture*}(-3.3,-2.4)(16.64,3.5)
\pspolygon[fillcolor=gray,fillstyle=solid,opacity=0.1](0.02,0.32)(1.4,1.76)(4.01,-2.04)
\pspolygon[fillcolor=lightgray,fillstyle=solid,opacity=0.1](1.4,1.76)(7.04,0.34)(4.01,-2.04)
\psline[linestyle=dashed,dash=2pt 2pt](0.02,0.32)(6.08,0.34)
\psline(0.02,0.32)(2.78,-1.32)
\psline[linestyle=dashed,dash=2pt 2pt](6.08,0.34)(2.78,-1.32)
\psline(2.66,3.06)(0.02,0.32)
\psline(2.66,3.06)(2.78,-1.32)

\psline(2.66,3.06)(5.58,0.74)
\psline[linestyle=dashed,dash=2pt 2pt](5.58,0.74)(6.08,0.34)

\psline(2.78,-1.32)(4.01,-2.04)
\psline(4.01,-2.04)(1.4,1.76)
\psline(1.4,1.76)(7.04,0.34)
\psline(4.01,-2.04)(8.94,1.84)
\psline(2.75,-0.18)(8.94,1.84)
\psline[linestyle=dashed,dash=2pt 2pt](6.08,0.34)(6.35,0.51)
\psline(6.35,0.51)(8.94,1.84)
\psline[linestyle=dashed,dash=2pt 2pt](6.08,0.34)(7.04,0.34)
\rput[tl](-0.25,0.2){$A$}
\rput[tl](2.58,-1.45){$B$}
\rput[tl](5.9,0.2){$C$}
\rput[tl](8.85,2.2){$P$}
\rput[tl](7.02,0.2){$Q$}
\rput[tl](3.86,-2.15){$R$}
\rput[tl](1.1,2.1){$U$}
\rput[tl](2.34,-0.2){$V$}
\rput[tl](5.5,1.15){$W$}
\rput[tl](2.5,3.4){$D$}
\psline(0.02,0.32)(1.4,1.76)
\psline(1.4,1.76)(4.01,-2.04)
\psline(4.01,-2.04)(0.02,0.32)
\psline(1.4,1.76)(7.04,0.34)
\psline(7.04,0.34)(4.01,-2.04)
\psline(4.01,-2.04)(1.4,1.76)
\begin{scriptsize}
\psdots[dotstyle=*](0.02,0.32)
\psdots[dotstyle=*](6.08,0.34)
\psdots[dotstyle=*](2.78,-1.32)
\psdots[dotstyle=*](2.66,3.06)
\psdots[dotstyle=*](4.01,-2.04)
\psdots[dotstyle=*](1.4,1.76)
\psdots[dotstyle=*](7.04,0.34)
\psdots[dotstyle=*](2.75,-0.18)
\psdots[dotstyle=*](5.58,0.74)
\psdots[dotstyle=*](8.94,1.84)
\end{scriptsize}
\end{pspicture*}
\end{center}
which leads to the axiomatic sequent
\begin{equation}\label{des2}
\vdash (A,R,U,V,D,B), (A,R,Q,P,C,B), (U,R,Q,P,W,V), (A,Q,U,W,D,C).
\end{equation}
Consider now a projective interpretation that maps again $A,B,C,U,V,W,P,Q,R$ as indicated, while $D$ is interpreted as the intersection point of $AU$ and $BV$. This interpretation satisfies
\[
(A,R,U,V,D,B), (A,R,Q,P,C,B), \; \mbox{\rm and}\; (U,R,Q,P,W,V),
\]
and, by Proposition 7.2 applied to the sequent (\ref{des2}), it
satisfies $(A,C,D,W,U,Q)$, i.e.\ $(A,Q,U,W,D,C)$. Therefore $W$,
$D$ and $C$ are collinear, which means that the lines $AU$, $BV$ and $CW$ are concurrent.
\end{proof}

The above proof shows how the Menelaus system may assist in
proving some concrete projective results. For one direction just
an axiomatic sequent is used and for the other direction, if we
stick to the original system, a cut applied to the axiomatic
sequent (\ref{des2}) and the axiomatic sequent
\[
\vdash (A,C,D,W,U,Q),(A,Q,U,W,D,C),
\]
coming from switching of triangles (cf.\ the second axiomatic sequent in (\ref{perm-swit})), is sufficient. However, our intention is to cover both directions of the Desargues theorem with a single sequent. In order to do this, consider the two axiomatic sequents (\ref{des1}) and (\ref{des2}). Since they share the three elements
\[
(A,B,D,V,U,R), (A,C,D,W,U,Q), (A,B,C,P,Q,R),
\]
one may derive the sequent
\begin{equation}\label{des3}
\begin{array}{rl}
\vdash (A,B,D,V,U,R),\!\!\!\!\! & (A,C,D,W,U,Q), (A,B,C,P,Q,R),
\\[1ex]
 & (B,C,D,W,V,P)\diskon (U,R,Q,P,W,V),
\end{array}
\end{equation}
which encapsulates all that we need to ``prove'' for the Desargues
theorem. It remains to use the appropriate interpretations. To see
that (\ref{des3}) delivers the \emph{only if} part, let
$A,B,C,U,V,W,P,Q,R$ be nine points in $\mathbf{RP}^2$ satisfying
the first sentence of the theorem and the condition that the lines
$AU$, $BV$ and $CW$ are concurrent. The projective interpretation
that maps $A,B,C,U,V,W,P,Q,R$ as indicated and $D$ to the common
point of $AU$, $BV$ and $CW$ satisfies
\[
(A,B,D,V,U,R), (A,C,D,W,U,Q), \; \mbox{\rm and}\; (B,C,D,W,V,P),
\]
and hence
\[
(A,B,D,V,U,R), (A,C,D,W,U,Q), \; \mbox{\rm and}\;
(B,C,D,W,V,P)\vee (U,R,Q,P,W,V).
\]
By Proposition~7.2 applied to the sequent (\ref{des3}), this interpretation satisfies \linebreak $(A,B,C,P,Q,R)$. Therefore, $P$, $Q$ and $R$ are collinear.

To see that (\ref{des3}) delivers the \emph{if} part, let $A,B,C,U,V,W,P,Q,R$ be nine points in $\mathbf{RP}^2$ satisfying the first sentence of the theorem and the condition that the points $P$, $Q$ and $R$ are collinear. The projective interpretation that maps
$A$, $B$, $C$, $U$, $V$, $W$, $P$, $Q$, $R$ as indicated and $D$ to the intersection point
of $AU$ and $BV$ satisfies
\[
(A,R,U,V,D,B), (A,R,Q,P,C,B), \; \mbox{\rm and}\; (U,R,Q,P,W,V),
\]
and hence
\[
(A,B,D,V,U,R), (A,B,C,P,Q,R),  \; \mbox{\rm and}\; (B,C,D,W,V,P)\vee (U,R,Q,P,W,V).
\]
By Proposition~7.2 applied to the sequent (\ref{des3}), this interpretation satisfies \linebreak $(A,C,D,W,U,Q)$, i.e.\ $(A,Q,U,W,D,C)$. Therefore $W$, $D$ and $C$ are collinear, which means that the lines $AU$, $BV$ and $CW$ are concurrent.

Hence, the role of $\diskon$ is just to ``pack'' two sequents into a single sequent sufficient for the proof of the theorem. As it is noted at the beginning of this section, single sequent proofs could be impossible for some results possessing proofs assisted with several provable sequents. However, all the examples that follow have single sequent proofs.
\end{example}

\begin{example}\label{example-2}

{\it
Let $AU$, $BV$ and $CW$ be three concurrent lines in $\mathbf{RP}^2$, and let $X$ and $E$ be such that $B$, $X$ and $E$ are collinear.
For
\[
\begin{array}{c}
\{P\}=BC\cap VW,\quad \{Q\}=AC\cap UW,\quad  \{R\}=AB\cap UV,
\\[1ex]
\{Y\}=AX\cap RE,\quad  \{Z\}=XC\cap EP,
\end{array}
\]
the points $Q$, $Y$ and $Z$ are collinear.
}

One can prove this incidence result by relying on a sequent obtained by a non-eliminable cut. Consider the following two axiomatic sequents obtained from tetrahedral triangulations $ABCD$ and $BRPE$ of two spheres.
\[
\begin{array}{l}
\vdash (A,B,D,V,U,R), (B,C,D,W,V,P), (A,C,D,W,U,Q), (A,B,C,P,Q,R)
\\[2ex]
\vdash (B,R,E,Y,X,A), (B,P,E,Z,X,C), (R,P,E,Z,Y,Q), (B,P,R,Q,A,C)
\end{array}
\]

Since $(A,B,C,P,Q,R)$ and $(B,P,R,Q,A,C)$ are identified, one may apply the cut rule to these two sequents in order to obtain the sequent
\[
\begin{array}{rl}
\vdash (A,B,D,V,U,R), (B,C,D,W,V,P), \!\!\!\!\! &  (A,C,D,W,U,Q), (B,R,E,Y,X,A),
\\[1ex]
 & (B,P,E,Z,X,C), (R,P,E,Z,Y,Q).
\end{array}
\] The use of the permutation axiom here is reflected
geometrically in the way the tetrahedra are glued together, so as
not to form an $\mathcal{M}$-complex.

With the indicated interpretation of the members of these sextuples, the conditions in the statement guarantee that the first five sextuples are satisfied. Hence, this interpretation satisfies $(R,P,E,Z,Y,Q)$, which means that $Q$, $Y$ and $Z$ are
 collinear.

\vspace{2ex}

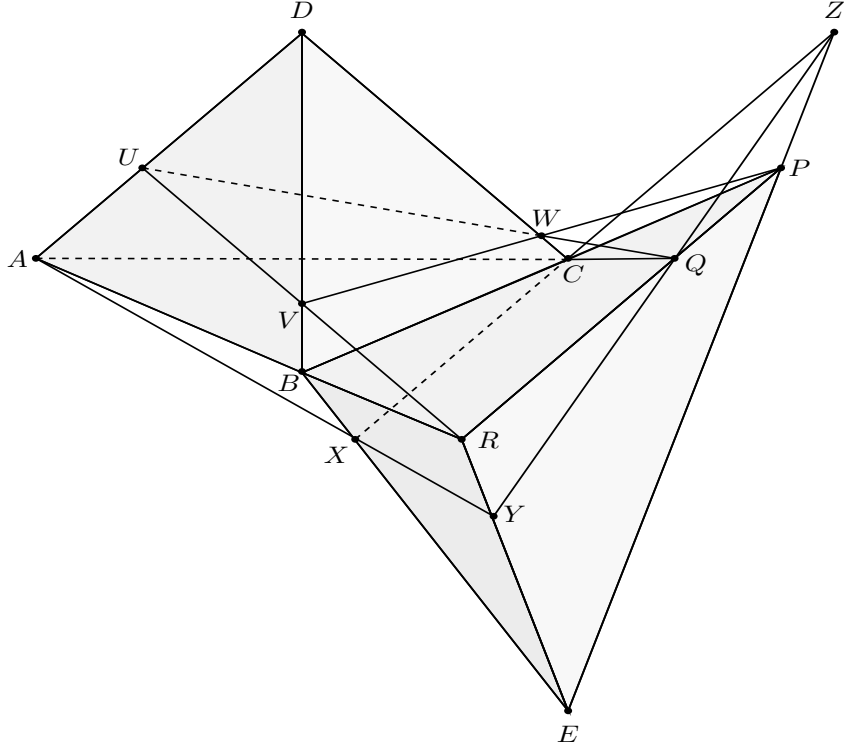
\psset{xunit=1cm,yunit=1cm,algebraic=true,dimen=middle,dotstyle=o,dotsize=5pt 0,linewidth=1.6pt,arrowsize=3pt 2,arrowinset=0.25}
\psscalebox{.35 .3}{
\begin{pspicture*}(-2,-2.7)(31,32)
\pspolygon[linewidth=2pt,fillcolor=lightgray,fillstyle=solid,opacity=0.1](10,29.95)(10,14.95)(20,19.95)
\pspolygon[linewidth=2pt,fillcolor=gray,fillstyle=solid,opacity=0.1](0,20)(10,29.95)(10,14.95)
\pspolygon[linewidth=2pt,fillcolor=lightgray,fillstyle=solid,opacity=0.1](28,24)(16,12)(20,0)
\pspolygon[linewidth=2pt,fillcolor=gray,fillstyle=solid,opacity=0.1](10,14.95)(16,12)(28,24)
\pspolygon[linewidth=2pt,fillcolor=darkgray,fillstyle=solid,opacity=0.1](10,14.95)(20,0)(16,12)
\psline[linewidth=2pt](0,20)(10,29.95)
\psline[linewidth=2pt](10,29.95)(20,19.95)
\psline[linewidth=2pt,linestyle=dashed,dash=7pt 7pt](0,20)(20,19.95)
\psline[linewidth=2pt](20,19.95)(10,14.95)
\psline[linewidth=2pt](0,20)(10,14.95)
\psline[linewidth=2pt](10,29.95)(10,14.95)
\psline[linewidth=2pt](4,24)(16,12)
\psline[linewidth=2pt,linestyle=dashed,dash=7pt 7pt](20,20)(12,12)
\psline[linewidth=2pt](20,20)(30,30)
\psline[linewidth=2pt](10,14.95)(20,0)
\psline[linewidth=2pt](16,12)(20,0)
\psline[linewidth=2pt](28,24)(20,0)
\psline[linewidth=2pt](28,24)(16,12)
\psline[linewidth=2pt](28,24)(20,19.95)
\psline[linewidth=2pt](28,24)(10,18)
\psline[linewidth=2pt](30,30)(17.2,8.6)
\psline[linewidth=2pt](0,20)(17.2,8.6)
\psline[linewidth=2pt](20,19.95)(24,20)
\psline[linewidth=2pt,linestyle=dashed,dash=7pt 7pt](4,24)(19,21)
\psline[linewidth=2pt](19,21)(24,20)
\psline[linewidth=2pt](10,14.95)(16,12)
\psline[linewidth=2pt](30,30)(28,24)

\rput(-0.7,20){\psscalebox{3}{$A$}} \rput(9.5,14.5){\psscalebox{3}{$B$}} \rput(20.2,19.3){\psscalebox{3}{$C$}} \rput(10,31){\psscalebox{3}{$D$}} \rput(3.5,24.5){\psscalebox{3}{$U$}} \rput(9.5,17.3){\psscalebox{3}{$V$}} \rput(19.2,21.8){\psscalebox{3}{$W$}} \rput(30,31){\psscalebox{3}{$Z$}} \rput(28.7,24){\psscalebox{3}{$P$}} \rput(24.8,19.7){\psscalebox{3}{$Q$}} \rput(17,12){\psscalebox{3}{$R$}} \rput(11.3,11.3){\psscalebox{3}{$X$}} \rput(18,8.7){\psscalebox{3}{$Y$}} \rput(20,-1){\psscalebox{3}{$E$}}

\rput(20,0){\pscircle*{.15}} \rput(17.2,8.6){\pscircle*{.15}} \rput(4,24){\pscircle*{.15}} \rput(12,12){\pscircle*{.15}} \rput(16,12){\pscircle*{.15}} \rput(10,15){\pscircle*{.15}} \rput(10,18){\pscircle*{.15}} \rput(0,20){\pscircle*{.15}} \rput(20,20){\pscircle*{.15}} \rput(24,20){\pscircle*{.15}} \rput(19,21){\pscircle*{.15}} \rput(47,24){\pscircle*{.15}} \rput(28,24){\pscircle*{.15}} \rput(10,30){\pscircle*{.15}} \rput(30,30){\pscircle*{.15}}
\end{pspicture*}}
\end{example}

\begin{example}\label{example-3} This example provides an incidence result tied to a more involved axiomatic sequent. Consider the following triangulation of the torus with two holes into ten triangles and three zero-cells in total. The corresponding axiomatic sequent is
\[
\begin{array}{rl}
\vdash (X,Y,Z,B,1,A),\!\!\!\!\! & (X,Y,Z,B,2,C), (X,Y,Z,D,3,C), (X,Y,Z,D,4,E)
\\[1ex]
 & (X,Y,Z,F,5,E), (X,Y,Z,F,1,G),(X,Y,Z,H,4,G),
\\[1ex]
 & (X,Y,Z,H,5,I), (X,Y,Z,J,2,I),(X,Y,Z,J,3,A).
\end{array}
\]
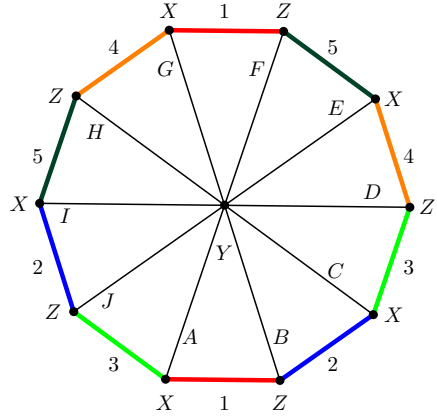
\begin{center}
\psscalebox{.4} 
{
\begin{pspicture}(3,-10)(17,5)

\psline[linewidth=0.05,dimen=outer](8,-8.2)(11.88,3.33)
\psline[linewidth=0.05,dimen=outer](11.76,-8.24)(8.12,3.37)
\psline[linewidth=0.05,dimen=outer](14.83,-6.06)(5.06,1.19)
\psline[linewidth=0.05,dimen=outer](16.03,-2.5)(3.86,-2.37)
\psline[linewidth=0.05,dimen=outer](4.98,-5.96)(14.9,1.09)

\psline[linewidth=0.15,dimen=outer,linecolor=red](8,-8.2)(11.76,-8.24)
\psline[linewidth=0.15,dimen=outer,linecolor=red](8.12,3.37)(11.88,3.33)

\psline[linewidth=0.15,dimen=outer,linecolor=blue](11.76,-8.24)(14.83,-6.06) \psline[linewidth=0.15,dimen=outer,linecolor=orange](5.06,1.19)(8.12,3.37)

\psline[linewidth=0.15,dimen=outer,linecolor=green](14.83,-6.06)(16.03,-2.5) \psline[linewidth=0.15,dimen=outer,linecolor=britishracinggreen](3.86,-2.37)(5.06,1.19)

\psline[linewidth=0.15,dimen=outer,linecolor=blue](4.98,-5.96)(3.86,-2.37) \psline[linewidth=0.15,dimen=outer,linecolor=orange](16.03,-2.5)(14.9,1.09)

\psline[linewidth=0.15,dimen=outer,linecolor=britishracinggreen](14.9,1.09)(11.88,3.33) \psline[linewidth=0.15,dimen=outer,linecolor=green](8,-8.2)(4.98,-5.96)

\rput(8,-9.0){\psscalebox{2}{$X$}} \rput(11.76,-9.0){\psscalebox{2}{$Z$}} \rput(15.5,-6.06){\psscalebox{2}{$X$}} \rput(16.6,-2.5){\psscalebox{2}{$Z$}} \rput(15.5,1.09){\psscalebox{2}{$X$}} \rput(11.88,4.0){\psscalebox{2}{$Z$}} \rput(8.12,4.0){\psscalebox{2}{$X$}} \rput(4.4,1.19){\psscalebox{2}{$Z$}} \rput(3.2,-2.37){\psscalebox{2}{$X$}} \rput(4.3,-5.96){\psscalebox{2}{$Z$}} \rput(9.94,-4.0){\psscalebox{2}{$Y$}}

\rput(9.94,-9.0){\psscalebox{2}{$1$}} \rput(9.94,4.0){\psscalebox{2}{$1$}}
\rput(13.5,-7.7){\psscalebox{2}{$2$}} \rput(3.8,-4.5){\psscalebox{2}{$2$}}
\rput(16.0,-4.5){\psscalebox{2}{$3$}}
\rput(6.3,-7.7){\psscalebox{2}{$3$}}
\rput(16.0,-0.8){\psscalebox{2}{$4$}}
\rput(6.3,2.8){\psscalebox{2}{$4$}}
\rput(13.5,2.8){\psscalebox{2}{$5$}}
\rput(3.8,-0.8){\psscalebox{2}{$5$}}

\rput(8.8,-6.8){\psscalebox{2}{$A$}} \rput(11.8,-6.8){\psscalebox{2}{$B$}} \rput(13.6,-4.6){\psscalebox{2}{$C$}}
\rput(14.8,-2.0){\psscalebox{2}{$D$}}
\rput(13.6,0.8){\psscalebox{2}{$E$}}
\rput(11.0,2.1){\psscalebox{2}{$F$}}
\rput(8.0,2.1){\psscalebox{2}{$G$}}
\rput(5.7,0.0){\psscalebox{2}{$H$}}
\rput(4.7,-2.8){\psscalebox{2}{$I$}}
\rput(6.1,-5.6){\psscalebox{2}{$J$}}

\rput(8,-8.2){\pscircle*{.15}} 
\rput(11.76,-8.24){\pscircle*{.15}} 
\rput(14.83,-6.06){\pscircle*{.15}} 
\rput(16.03,-2.5){\pscircle*{.15}} 
\rput(14.9,1.09){\pscircle*{.15}} 
\rput(11.88,3.33){\pscircle*{.15}} 
\rput(8.12,3.37){\pscircle*{.15}} 
\rput(5.06,1.19){\pscircle*{.15}} 
\rput(3.86,-2.37){\pscircle*{.15}} 
\rput(4.98,-5.96){\pscircle*{.15}} 
\rput(9.94,-2.43){\pscircle*{.15}} 

\end{pspicture}
}
\end{center}

On the other hand, this sequent could be obtained by the cut rule applied to two axiomatic sequents corresponding to the Pappus theorem, which are derived from triangulations of the torus in six triangles (cf.\ the next example and see \cite[Section~3.4]{RG06} for more details).


\newrgbcolor{qqzzff}{0 0.6 1}
\newrgbcolor{ffxfqq}{1 0.5 0}
\newrgbcolor{qqqqcc}{0 0 0.8}
\newrgbcolor{qqttcc}{0 0.2 0.8}
\newrgbcolor{qqwuqq}{0 0.39 0}
\newrgbcolor{ffwwzz}{1 0.4 0.6}
\newrgbcolor{ffqqtt}{1 0 0.2}
\newrgbcolor{qqccqq}{0 0.8 0}
\newrgbcolor{wwqqcc}{0.4 0 0.8}
\psset{xunit=1.0cm,yunit=1.0cm,algebraic=true,dimen=middle,dotstyle=o,dotsize=3pt 0,linewidth=0.8pt,arrowsize=3pt 2,arrowinset=0.25}
\begin{center}
\psscalebox{.8 .8} 
{
}
\end{center}

\[
\begin{array}{l}
\vdash (X,Y,Z,B,1,A), (X,Y,Z,B,2,C), (X,Y,Z,D,3,C), (X,Y,Z,J,2,I),
\\[1ex]
\mbox{\hspace{11em}} (X,Y,Z,J,3,A), (X,Y,Z,D,1,I),
\\[2ex]
\vdash (X,Y,Z,D,4,E), (X,Y,Z,F,5,E), (X,Y,Z,F,1,G),(X,Y,Z,H,4,G),
\\[1ex]
\mbox{\hspace{11em}} (X,Y,Z,H,5,I), (X,Y,Z,D,1,I).
\end{array}
\]
Note that the torus with two holes corresponding to the decagon is
obtained by the connected sum, with respect to the triangle whose
sides are $D$, 1 and $I$, from the two tori corresponding to the
hexagons. This illustrates the fact (mentioned in Section
\ref{system}) that some axiomatic sequents could be derived from
more primitive axiomatic sequents (more on this in Section
\ref{cyclic}). The following ``game'' may be treated as an
incidence result extracted from the sequent above.

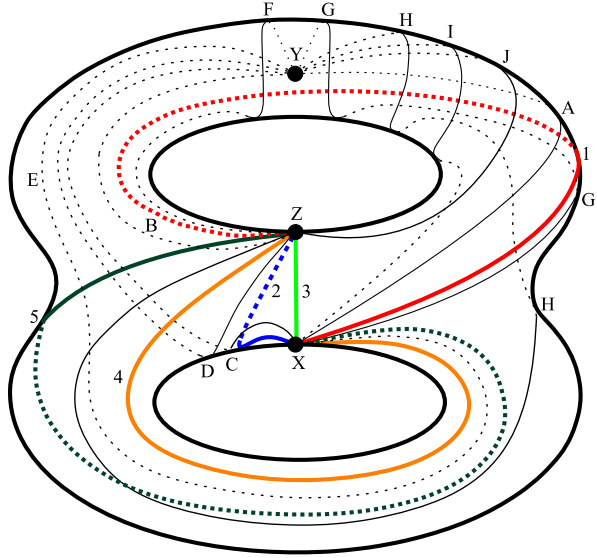
\begin{center}
%

\end{center}

\vspace{2ex}

{\it Let $p$, $q$ and $s$ be three non-concurrent lines in the projective plane. Choose a point $A$ on $p$, then choose $B$ on $q$ and let the line $AB$ intersects $s$ in 1. Then choose $C$ on $p$ and let $BC$ intersect $s$ in 2, and so on up to the point $F$ on $q$ and the intersection point 5 of $EF$ and $s$. Continue with this zigzag game, save that the path now crosses $s$ in the ``old'' points 1, 4, 5, 2, 3, respectively, i.e.\ the line $FG$ intersects $s$ in 1, $GH$ intersects $s$ in 4, and so on. Eventually, the last segment, which starts in $J$ and goes in the direction of 3 toward $p$, ends in the initial point~$A$.}

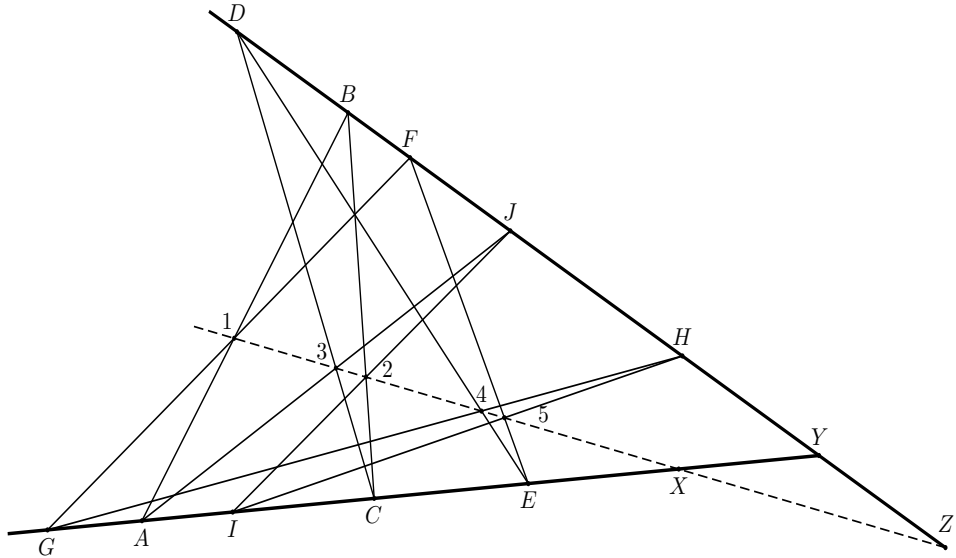
\begin{center}
\psscalebox{.18 .23} 
{
\begin{pspicture}(-15,-23)(52,10)

\psline[linewidth=0.2,dimen=outer](-16.55,-21.58)(42.73,-17.07)
\psline[linewidth=0.2,dimen=outer](-1.84,8.54)(51.97,-22.38)

\psline[linewidth=0.1,dimen=outer,linestyle=dashed, dash=20pt 10pt](-2.94,-9.63)(51.97,-22.38)

\psline[linewidth=0.1,dimen=outer](-6.74,-20.83)(8.32,2.69)
\psline[linewidth=0.1,dimen=outer](8.32,2.69)(10.23,-19.54)
\psline[linewidth=0.1,dimen=outer](10.23,-19.54)(0.2,7.36)
\psline[linewidth=0.1,dimen=outer](0.2,7.36)(21.49,-18.69)
\psline[linewidth=0.1,dimen=outer](21.49,-18.69)(12.85,0.1)
\psline[linewidth=0.1,dimen=outer](12.85,0.1)(-13.64,-21.36)
\psline[linewidth=0.1,dimen=outer](-13.64,-21.36)(32.75,-11.33)
\psline[linewidth=0.1,dimen=outer](32.75,-11.33)(-0.11,-20.33)
\psline[linewidth=0.1,dimen=outer](-0.11,-20.33)(20.18,-4.12)
\psline[linewidth=0.1,dimen=outer](20.18,-4.12)(-6.74,-20.83)

\rput(-13.7,-22.3){\psscalebox{4.5 4}{$G$}}
\rput(-6.8,-21.8){\psscalebox{4.5 4}{$A$}}
\rput(-0.1,-21.3){\psscalebox{4.5 4}{$I$}}
\rput(10.2,-20.5){\psscalebox{4.5 4}{$C$}}
\rput(21.5,-19.6){\psscalebox{4.5 4}{$E$}}
\rput(32.5,-18.8){\psscalebox{4.5 4}{$X$}}
\rput(42.8,-16){\psscalebox{4.5 4}{$Y$}}
\rput(52,-21){\psscalebox{4.5 4}{$Z$}}

\rput(-13.64,-21.36){\pscircle*{.15}} 
\rput(-6.74,-20.83){\pscircle*{.15}} 
\rput(-0.11,-20.33){\pscircle*{.15}} 
\rput(10.23,-19.54){\pscircle*{.15}} 
\rput(21.49,-18.69){\pscircle*{.15}} 

\rput(32.48,-17.85){\pscircle*{.15}} 
\rput(42.73,-17.07){\pscircle*{.15}} 
\rput(51.97,-22.38){\pscircle*{.15}} 

\rput(-0.5,-9.31){\psscalebox{4.5 4}{$1$}}
\rput(11.2,-12.1){\psscalebox{4.5 4}{$2$}}
\rput(6.4,-11.03){\psscalebox{4.5 4}{$3$}}
\rput(18.07,-13.51){\psscalebox{4.5 4}{$4$}}
\rput(22.6,-14.7){\psscalebox{4.5 4}{$5$}}

\rput(0,-10.31){\pscircle*{.15}} 
\rput(9.63,-12.54){\pscircle*{.15}} 
\rput(7.43,-12.03){\pscircle*{.15}} 
\rput(18.07,-14.51){\pscircle*{.15}} 
\rput(19.75,-14.89){\pscircle*{.15}} 

\rput(0.2,8.46){\psscalebox{4.5 4}{$D$}}
\rput(8.32,3.79){\psscalebox{4.5 4}{$B$}}
\rput(12.85,1.2){\psscalebox{4.5 4}{$F$}}
\rput(20.18,-3.02){\psscalebox{4.5 4}{$J$}}
\rput(32.75,-10.23){\psscalebox{4.5 4}{$H$}}

\rput(0.2,7.36){\pscircle*{.15}} 
\rput(8.32,2.69){\pscircle*{.15}} 
\rput(12.85,0.1){\pscircle*{.15}} 
\rput(20.18,-4.12){\pscircle*{.15}} 
\rput(32.75,-11.33){\pscircle*{.15}} 

\end{pspicture}
}
\end{center}
\end{example}

\begin{example}\label{example-4}
The proof of the following incidence result uses two axiomatic sequents connected by $\leftrightarrow$-introduction.

\begin{prop}\label{propex-4}
Consider the Pappus configuration consisting of two triples $(A,B,C)$ and $(D,E,F)$ of collinear points, all mutually distinct. Assume that for
$\{X\}=CD\cap AE$ and $\{Z\}=BE\cap CF$, the lines $AB$, $DE$ and
$XZ$ are not concurrent. For
\[
\begin{array}{lll}
\{K\}=BE\cap CD, & \{L\}=AF\cap CD, & \{M\}=AF\cap BE,
\\
\{U\}=AE\cap CF, & \{V\}=AE\cap BD, & \{W\}=CF\cap BD,
\end{array}
\]
the Pappus lines $KU$, $LV$ and $MW$ are concurrent.
\end{prop}

\begin{proof}
Let
$\{1\}=XZ\cap AB$, $\{2\}=AB\cap DE$, $\{3\}=XZ\cap DE$,
$\{O\}=KU\cap LV$, and $\{Y\}=13\cap BD$.

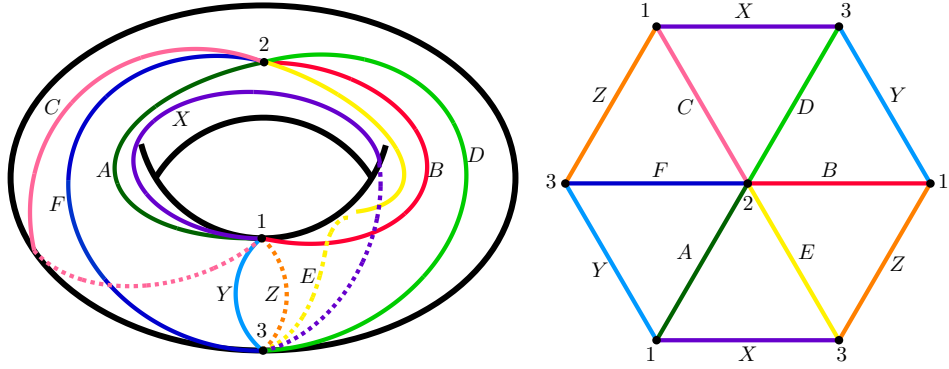
\newrgbcolor{qqzzff}{0 0.6 1}
\newrgbcolor{ffxfqq}{1 0.5 0}
\newrgbcolor{qqqqcc}{0 0 0.8}
\newrgbcolor{qqttcc}{0 0.2 0.8}
\newrgbcolor{qqwuqq}{0 0.39 0}
\newrgbcolor{ffwwzz}{1 0.4 0.6}
\newrgbcolor{ffqqtt}{1 0 0.2}
\newrgbcolor{qqccqq}{0 0.8 0}
\newrgbcolor{wwqqcc}{0.4 0 0.8}
\psset{xunit=1.0cm,yunit=1.0cm,algebraic=true,dimen=middle,dotstyle=o,dotsize=3pt 0,linewidth=0.8pt,arrowsize=3pt 2,arrowinset=0.25}
\begin{center}
\psscalebox{.8 .8} 
{
\begin{pspicture*}(-4.2,-3.2)(16.69,3.2)
\rput{0}(0,0){\psellipse[linewidth=2.8pt](0,0)(4.15,2.86)}
\parametricplot[linewidth=2.8pt]{3.3936126414032994}{6.021215152704372}{1*2.08*cos(t)+0*2.08*sin(t)+0.01|0*2.08*cos(t)+1*2.08*sin(t)+1.08}
\parametricplot[linewidth=2.8pt]{0.545463398049105}{2.5961292555406885}{1*2.08*cos(t)+0*2.08*sin(t)+0.01|0*2.08*cos(t)+1*2.08*sin(t)+-1.08}
\parametricplot[linewidth=2pt,linecolor=qqzzff]{2.279976774261442}{4.024668457434242}{1*1.22*cos(t)+0*1.22*sin(t)+0.77|0*1.22*cos(t)+1*1.22*sin(t)+-1.92}
\parametricplot[linewidth=2pt,linestyle=dashed,dash=2pt 2pt,linecolor=ffxfqq]{-0.7960284050386095}{0.8174883295547072}{1*1.29*cos(t)+0*1.29*sin(t)+-0.9|0*1.29*cos(t)+1*1.29*sin(t)+-1.94}
\parametricplot[linewidth=2pt,linecolor=qqqqcc]{4.876359128662692}{5.921249764843227}{-0.93*2.85*cos(t)+-0.37*2.37*sin(t)+-0.41|0.37*2.85*cos(t)+-0.93*2.37*sin(t)+-0.42}
\parametricplot[linewidth=2pt,linecolor=qqqqcc]{-0.36193554233635883}{0.28801900984834206}{-0.93*2.85*cos(t)+-0.37*2.37*sin(t)+-0.41|0.37*2.85*cos(t)+-0.93*2.37*sin(t)+-0.42}
\parametricplot[linewidth=2pt,linecolor=qqttcc]{0.28801900984834206}{0.9867712530003544}{-0.93*2.85*cos(t)+-0.37*2.37*sin(t)+-0.41|0.37*2.85*cos(t)+-0.93*2.37*sin(t)+-0.42}
\parametricplot[linewidth=2pt,linecolor=qqqqcc]{0.9867712530003544}{2.035888472841568}{-0.93*2.85*cos(t)+-0.37*2.37*sin(t)+-0.41|0.37*2.85*cos(t)+-0.93*2.37*sin(t)+-0.42}
\parametricplot[linewidth=2pt,linecolor=qqwuqq]{4.796287866793456}{5.513370746979559}{-0.99*2.96*cos(t)+0.16*1.46*sin(t)+0.5|-0.16*2.96*cos(t)+-0.99*1.46*sin(t)+0.52}
\parametricplot[linewidth=2pt,linecolor=qqwuqq]{-0.769814560200027}{0.09751614892685402}{-0.99*2.96*cos(t)+0.16*1.46*sin(t)+0.5|-0.16*2.96*cos(t)+-0.99*1.46*sin(t)+0.52}
\parametricplot[linewidth=2pt,linecolor=qqwuqq]{0.09751614892685402}{0.7823645134637963}{-0.99*2.96*cos(t)+0.16*1.46*sin(t)+0.5|-0.16*2.96*cos(t)+-0.99*1.46*sin(t)+0.52}
\parametricplot[linewidth=2pt,linecolor=qqwuqq]{0.7823645134637963}{1.3099205773499074}{-0.99*2.96*cos(t)+0.16*1.46*sin(t)+0.5|-0.16*2.96*cos(t)+-0.99*1.46*sin(t)+0.52}
\parametricplot[linewidth=2pt,linecolor=ffwwzz]{5.101666532076473}{5.812100631499277}{-0.73*2.99*cos(t)+-0.68*2.81*sin(t)+-0.93|0.68*2.99*cos(t)+-0.73*2.81*sin(t)+-0.75}
\parametricplot[linewidth=2pt,linecolor=ffwwzz]{-0.4710846756803093}{0.01597084132437098}{-0.73*2.99*cos(t)+-0.68*2.81*sin(t)+-0.93|0.68*2.99*cos(t)+-0.73*2.81*sin(t)+-0.75}
\parametricplot[linewidth=2pt,linecolor=ffwwzz]{0.01597084132437098}{0.4930772245259334}{-0.73*2.99*cos(t)+-0.68*2.81*sin(t)+-0.93|0.68*2.99*cos(t)+-0.73*2.81*sin(t)+-0.75}
\parametricplot[linewidth=2pt,linecolor=ffwwzz]{0.4930772245259334}{0.9378310323419501}{-0.73*2.99*cos(t)+-0.68*2.81*sin(t)+-0.93|0.68*2.99*cos(t)+-0.73*2.81*sin(t)+-0.75}
\parametricplot[linewidth=2pt,linestyle=dashed,dash=2pt 2pt,linecolor=ffwwzz]{0.29980171239023573}{0.8026092272871177}{-0.96*2.4*cos(t)+0.28*1.62*sin(t)+-1.7|-0.28*2.4*cos(t)+-0.96*1.62*sin(t)+-0.1}
\parametricplot[linewidth=2pt,linestyle=dashed,dash=2pt 2pt,linecolor=ffwwzz]{0.8026092272871177}{1.2856576751432676}{-0.96*2.4*cos(t)+0.28*1.62*sin(t)+-1.7|-0.28*2.4*cos(t)+-0.96*1.62*sin(t)+-0.1}
\parametricplot[linewidth=2pt,linestyle=dashed,dash=2pt 2pt,linecolor=ffwwzz]{1.2856576751432676}{1.764156086743084}{-0.96*2.4*cos(t)+0.28*1.62*sin(t)+-1.7|-0.28*2.4*cos(t)+-0.96*1.62*sin(t)+-0.1}
\parametricplot[linewidth=2pt,linestyle=dashed,dash=2pt 2pt,linecolor=ffwwzz]{1.764156086743084}{2.178177947197833}{-0.96*2.4*cos(t)+0.28*1.62*sin(t)+-1.7|-0.28*2.4*cos(t)+-0.96*1.62*sin(t)+-0.1}
\parametricplot[linewidth=2pt,linecolor=ffqqtt]{4.3610629972392605}{5.012668775187195}{-0.98*2.29*cos(t)+-0.17*1.48*sin(t)+0.44|0.17*2.29*cos(t)+-0.98*1.48*sin(t)+0.41}
\parametricplot[linewidth=2pt,linecolor=ffqqtt]{3.6701867790263365}{4.3610629972392605}{-0.98*2.29*cos(t)+-0.17*1.48*sin(t)+0.44|0.17*2.29*cos(t)+-0.98*1.48*sin(t)+0.41}
\parametricplot[linewidth=2pt,linecolor=ffqqtt]{2.6745386261246695}{3.6701867790263365}{-0.98*2.29*cos(t)+-0.17*1.48*sin(t)+0.44|0.17*2.29*cos(t)+-0.98*1.48*sin(t)+0.41}
\parametricplot[linewidth=2pt,linecolor=ffqqtt]{1.4802438664873026}{2.6745386261246695}{-0.98*2.29*cos(t)+-0.17*1.48*sin(t)+0.44|0.17*2.29*cos(t)+-0.98*1.48*sin(t)+0.41}
\parametricplot[linewidth=2pt,linecolor=qqccqq]{3.952050680584713}{4.532538608784681}{-0.92*3.03*cos(t)+0.39*2.34*sin(t)+0.41|-0.39*3.03*cos(t)+-0.92*2.34*sin(t)+-0.41}
\parametricplot[linewidth=2pt,linecolor=qqccqq]{3.2064712294409152}{3.952050680584713}{-0.92*3.03*cos(t)+0.39*2.34*sin(t)+0.41|-0.39*3.03*cos(t)+-0.92*2.34*sin(t)+-0.41}
\parametricplot[linewidth=2pt,linecolor=qqccqq]{2.1186643232226268}{3.2064712294409152}{-0.92*3.03*cos(t)+0.39*2.34*sin(t)+0.41|-0.39*3.03*cos(t)+-0.92*2.34*sin(t)+-0.41}
\parametricplot[linewidth=2pt,linecolor=qqccqq]{1.1163143309152501}{2.1186643232226268}{-0.92*3.03*cos(t)+0.39*2.34*sin(t)+0.41|-0.39*3.03*cos(t)+-0.92*2.34*sin(t)+-0.41}
\parametricplot[linewidth=2pt,linecolor=yellow]{2.458508275316489}{3.66283004762946}{-0.93*2.38*cos(t)+-0.38*1.07*sin(t)+0.07|0.38*2.38*cos(t)+-0.93*1.07*sin(t)+0.77}
\parametricplot[linewidth=2pt,linecolor=yellow]{3.66283004762946}{4.353411313885404}{-0.93*2.38*cos(t)+-0.38*1.07*sin(t)+0.07|0.38*2.38*cos(t)+-0.93*1.07*sin(t)+0.77}
\parametricplot[linewidth=2pt,linecolor=yellow]{4.353411313885404}{4.91451518639337}{-0.93*2.38*cos(t)+-0.38*1.07*sin(t)+0.07|0.38*2.38*cos(t)+-0.93*1.07*sin(t)+0.77}

\parametricplot[linewidth=2pt,linestyle=dashed,dash=2pt 2pt,linecolor=yellow]{3.1877435668202536}{4.224426726437027}{-0.2*0.55*cos(t)+0.98*0.27*sin(t)+1.3|-0.98*0.55*cos(t)+-0.2*0.27*sin(t)+-1.2}

\parametricplot[linewidth=2pt,linestyle=dashed,dash=2pt 2pt,linecolor=yellow]{4.224426726437027}{4.962155189093001}{-0.2*0.55*cos(t)+0.98*0.27*sin(t)+1.3|-0.98*0.55*cos(t)+-0.2*0.27*sin(t)+-1.2}

\parametricplot[linewidth=2pt,linestyle=dashed,dash=2pt 2pt,linecolor=yellow]{3.365534736074497}{3.4091646125566917}{-0.43415047669753043*28.79714620222122*cos(t)+-0.9008403652053492*5.674955749185209*sin(t)+-12.39|
0.9008403652053492*28.79714620222122*cos(t)+-0.43415047669753043*5.674955749185209*sin(t)+23.13}

\parametricplot[linewidth=2pt,linestyle=dashed,dash=2pt 2pt,linecolor=yellow]{3.317991240386494}{3.365534736074497}{-0.43415047669753043*28.79714620222122*cos(t)+-0.9008403652053492*5.674955749185209*sin(t)+-12.39|
0.9008403652053492*28.79714620222122*cos(t)+-0.43415047669753043*5.674955749185209*sin(t)+23.13}

\parametricplot[linewidth=2pt,linestyle=dashed,dash=2pt 2pt,linecolor=yellow]{3.2675312276066593}{3.317991240386494}{-0.43415047669753043*28.79714620222122*cos(t)+-0.9008403652053492*5.674955749185209*sin(t)+-12.39|
0.9008403652053492*28.79714620222122*cos(t)+-0.43415047669753043*5.674955749185209*sin(t)+23.13}
\parametricplot[linewidth=2pt,linestyle=dashed,dash=2pt 2pt,linecolor=yellow]{3.1193320943659657}{3.2675312276066593}{-0.43415047669753043*28.79714620222122*cos(t)+-0.9008403652053492*5.674955749185209*sin(t)+-12.39|
0.9008403652053492*28.79714620222122*cos(t)+-0.43415047669753043*5.674955749185209*sin(t)+23.13}

\parametricplot[linewidth=2pt,linecolor=wwqqcc]{3.2714122191065176}{4.775270846461568}{-1*2.03*cos(t)+-0.06*1.21*sin(t)+-0.1|0.06*2.03*cos(t)+-1*1.21*sin(t)+0.21}
\parametricplot[linewidth=2pt,linecolor=wwqqcc]{4.775270846461568}{5.736341849300105}{-1*2.03*cos(t)+-0.06*1.21*sin(t)+-0.1|0.06*2.03*cos(t)+-1*1.21*sin(t)+0.21}
\parametricplot[linewidth=2pt,linecolor=wwqqcc]{-0.546843457879481}{0.49921065979275303}{-1*2.03*cos(t)+-0.06*1.21*sin(t)+-0.1|0.06*2.03*cos(t)+-1*1.21*sin(t)+0.21}
\parametricplot[linewidth=2pt,linecolor=wwqqcc]{0.49921065979275303}{1.6476451344447252}{-1*2.03*cos(t)+-0.06*1.21*sin(t)+-0.1|0.06*2.03*cos(t)+-1*1.21*sin(t)+0.21}
\parametricplot[linewidth=2pt,linestyle=dashed,dash=2pt 2pt,linecolor=wwqqcc]{0.47503988185286944}{1.5219882611117619}{-0.39*2.37*cos(t)+0.92*1.72*sin(t)+0.1|-0.92*2.37*cos(t)+-0.39*1.72*sin(t)+-0.62}
\parametricplot[linewidth=2pt,linestyle=dashed,dash=2pt 2pt,linecolor=wwqqcc]{1.5219882611117619}{1.8448740078507204}{-0.39*2.37*cos(t)+0.92*1.72*sin(t)+0.1|-0.92*2.37*cos(t)+-0.39*1.72*sin(t)+-0.62}
\parametricplot[linewidth=2pt,linestyle=dashed,dash=2pt 2pt,linecolor=wwqqcc]{1.8448740078507204}{2.0832626535555896}{-0.39*2.37*cos(t)+0.92*1.72*sin(t)+0.1|-0.92*2.37*cos(t)+-0.39*1.72*sin(t)+-0.62}
\parametricplot[linewidth=2pt,linestyle=dashed,dash=2pt 2pt,linecolor=wwqqcc]{2.0832626535555896}{2.2608210007427445}{-0.39*2.37*cos(t)+0.92*1.72*sin(t)+0.1|-0.92*2.37*cos(t)+-0.39*1.72*sin(t)+-0.62}
\rput[tl](-0.07,2.33){$2$}
\rput[tl](-0.11,-0.6){$1$}
\rput[tl](-0.11,-2.4){$3$}
\psline[linewidth=2pt,linecolor=wwqqcc](6.47,2.51)(9.47,2.51)
\psline[linewidth=2pt,linecolor=qqzzff](9.47,2.51)(10.97,-0.09)
\psline[linewidth=2pt,linecolor=ffxfqq](10.97,-0.09)(9.47,-2.69)
\psline[linewidth=2pt,linecolor=wwqqcc](9.47,-2.69)(6.47,-2.69)
\psline[linewidth=2pt,linecolor=qqzzff](6.47,-2.69)(4.97,-0.09)
\psline[linewidth=2pt,linecolor=ffxfqq](4.97,-0.09)(6.47,2.51)
\psline[linewidth=2pt,linecolor=ffwwzz](6.47,2.51)(7.97,-0.09)
\psline[linewidth=2pt,linecolor=yellow](7.97,-0.09)(9.47,-2.69)
\psline[linewidth=2pt,linecolor=qqccqq](9.47,2.51)(7.97,-0.09)
\psline[linewidth=2pt,linecolor=qqwuqq](7.97,-0.09)(6.47,-2.69)
\psline[linewidth=2pt,linecolor=qqqqcc](4.97,-0.09)(7.97,-0.09)
\psline[linewidth=2pt,linecolor=ffqqtt](7.97,-0.09)(10.97,-0.09)
\rput[tl](7.75,2.9){$X$}
\rput[tl](7.81,-2.8){$X$}
\rput[tl](10.3,1.5){$Y$}
\rput[tl](5.4,-1.4){$Y$}
\rput[tl](5.4,1.5){$Z$}
\rput[tl](10.3,-1.3){$Z$}
\rput[tl](6.8,-1.1){$A$}
\rput[tl](9.17,0.25){$B$}
\rput[tl](6.81,1.3){$C$}
\rput[tl](8.77,1.3){$D$}
\rput[tl](8.8,-1.1){$E$}
\rput[tl](6.39,0.25){$F$}
\rput[tl](6.2,2.9){$1$}
\rput[tl](11.1,0.06){$1$}
\rput[tl](6.29,-2.8){$1$}
\rput[tl](7.89,-0.3){$2$}
\rput[tl](9.5,2.9){$3$}
\rput[tl](4.65,0.06){$3$}
\rput[tl](9.43,-2.8){$3$}

\rput[tl](-3.6,1.3){$C$}
\rput[tl](-3.52,-0.3){$F$}
\rput[tl](-2.75,0.25){$A$}
\rput[tl](-0.77,-1.8){$Y$}
\rput[tl](0.02,-1.8){$Z$}
\rput[tl](3.35,0.5){$D$}
\rput[tl](2.7,0.25){$B$}
\rput[tl](-1.5,1.1){$X$}
\rput[tl](0.6,-1.5){$E$}

\begin{scriptsize}
\psdots[dotsize=4pt 0,dotstyle=*](0,-2.86)
\psdots[dotsize=4pt 0,dotstyle=*](-0.02,-1)
\psdots[dotsize=4pt 0,dotstyle=*](0.02,1.92)
\psdots[dotsize=4pt 0,dotstyle=*](9.47,2.51)
\psdots[dotsize=4pt 0,dotstyle=*](6.47,2.51)
\psdots[dotsize=4pt 0,dotstyle=*](4.97,-0.09)
\psdots[dotsize=4pt 0,dotstyle=*](6.47,-2.69)
\psdots[dotsize=4pt 0,dotstyle=*](9.47,-2.69)
\psdots[dotsize=4pt 0,dotstyle=*](10.97,-0.09)
\psdots[dotsize=4pt 0,dotstyle=*](7.97,-0.09)
\end{scriptsize}
\end{pspicture*}}
\end{center}

The following axiomatic sequent is obtained by a triangulation of the torus in six triangles all having vertices $1$, $2$ and $3$.
\[
\begin{array}{rl}
\vdash  (1,2,3,E,X,A), (1,2,3,E,Z,B), \!\!\!\!\! &   (1,2,3,D,Y,B), (1,2,3,D,X,C),
\\[1ex]
 & (1,2,3,F,Z,C), (1,2,3,F,Y,A).
\end{array}
\]
On the other hand, as in Example~1, from a tetrahedral triangulation of the sphere given by a tetrahedron $UXZK$, we obtain the axiomatic sequent
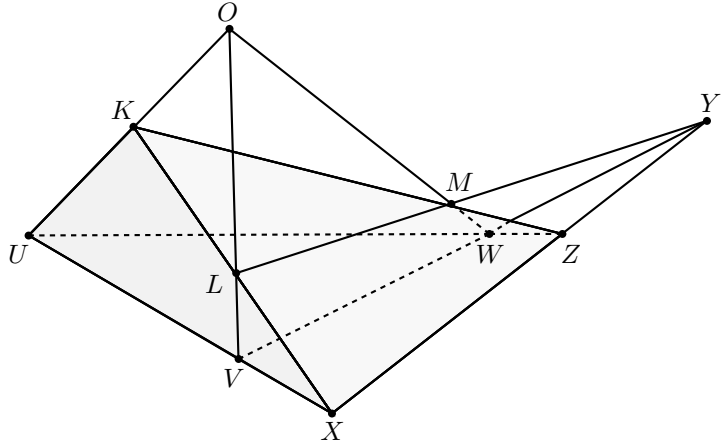
\[
\vdash (U,X,K,L,O,V), (U,X,Z,Y,W,V), (K,X,Z,Y,M,L), (U,Z,K,M,O,W).
\]
\begin{center}
\psset{xunit=1.0cm,yunit=1.0cm,algebraic=true,dimen=middle,dotstyle=o,dotsize=3pt 0,linewidth=0.8pt,arrowsize=3pt 2,arrowinset=0.25}
\begin{pspicture*}(-3.3,-2.5)(16.64,3.5)
\pspolygon[fillcolor=gray,fillstyle=solid,opacity=0.1](0.02,0.32)(1.4,1.76)(4.01,-2.04)
\pspolygon[fillcolor=lightgray,fillstyle=solid,opacity=0.1](1.4,1.76)(7.04,0.34)(4.01,-2.04)
\psline[linestyle=dashed,dash=2pt 2pt](0.02,0.32)(6.08,0.34)
\psline(0.02,0.32)(2.78,-1.32)
\psline[linestyle=dashed,dash=2pt 2pt](6.08,0.34)(2.78,-1.32)
\psline(2.66,3.06)(0.02,0.32)
\psline(2.66,3.06)(2.78,-1.32)

\psline(2.66,3.06)(5.58,0.74)
\psline[linestyle=dashed,dash=2pt 2pt](5.58,0.74)(6.08,0.34)

\psline(2.78,-1.32)(4.01,-2.04)
\psline(4.01,-2.04)(1.4,1.76)
\psline(1.4,1.76)(7.04,0.34)
\psline(4.01,-2.04)(8.94,1.84)
\psline(2.75,-0.18)(8.94,1.84)
\psline[linestyle=dashed,dash=2pt 2pt](6.08,0.34)(6.35,0.51)
\psline(6.35,0.51)(8.94,1.84)
\psline[linestyle=dashed,dash=2pt 2pt](6.08,0.34)(7.04,0.34)
\rput[tl](-0.25,0.2){$U$}
\rput[tl](2.58,-1.45){$V$}
\rput[tl](5.9,0.2){$W$}
\rput[tl](8.85,2.2){$Y$}
\rput[tl](7.02,0.2){$Z$}
\rput[tl](3.86,-2.15){$X$}
\rput[tl](1.1,2.1){$K$}
\rput[tl](2.34,-0.2){$L$}
\rput[tl](5.5,1.15){$M$}
\rput[tl](2.5,3.4){$O$}
\psline(0.02,0.32)(1.4,1.76)
\psline(1.4,1.76)(4.01,-2.04)
\psline(4.01,-2.04)(0.02,0.32)
\psline(1.4,1.76)(7.04,0.34)
\psline(7.04,0.34)(4.01,-2.04)
\psline(4.01,-2.04)(1.4,1.76)
\begin{scriptsize}
\psdots[dotstyle=*](0.02,0.32)
\psdots[dotstyle=*](6.08,0.34)
\psdots[dotstyle=*](2.78,-1.32)
\psdots[dotstyle=*](2.66,3.06)
\psdots[dotstyle=*](4.01,-2.04)
\psdots[dotstyle=*](1.4,1.76)
\psdots[dotstyle=*](7.04,0.34)
\psdots[dotstyle=*](2.75,-0.18)
\psdots[dotstyle=*](5.58,0.74)
\psdots[dotstyle=*](8.94,1.84)
\end{scriptsize}
\end{pspicture*}
\end{center}

By applying $\leftrightarrow$-introduction, one obtains the following nine-element sequent.
\[
\begin{array}{rl}
\vdash  (1,2,3,E,X,A), \!\!\!\!\! &    (1,2,3,E,Z,B), (1,2,3,D,Y,B), (1,2,3,D,X,C),
\\[1ex]
 & (1,2,3,F,Z,C), (1,2,3,F,Y,A)\leftrightarrow(K,X,Z,Y,M,L),
\\[1ex]
 & (U,X,K,L,O,V), (U,X,Z,Y,W,V), (U,Z,K,M,O,W).
\end{array}
\]

Consider now a projective interpretation that maps all points as indicated. (The pairs of corresponding sides of the triangles $KLM$ and $UVW$ are indicated in the picture.)

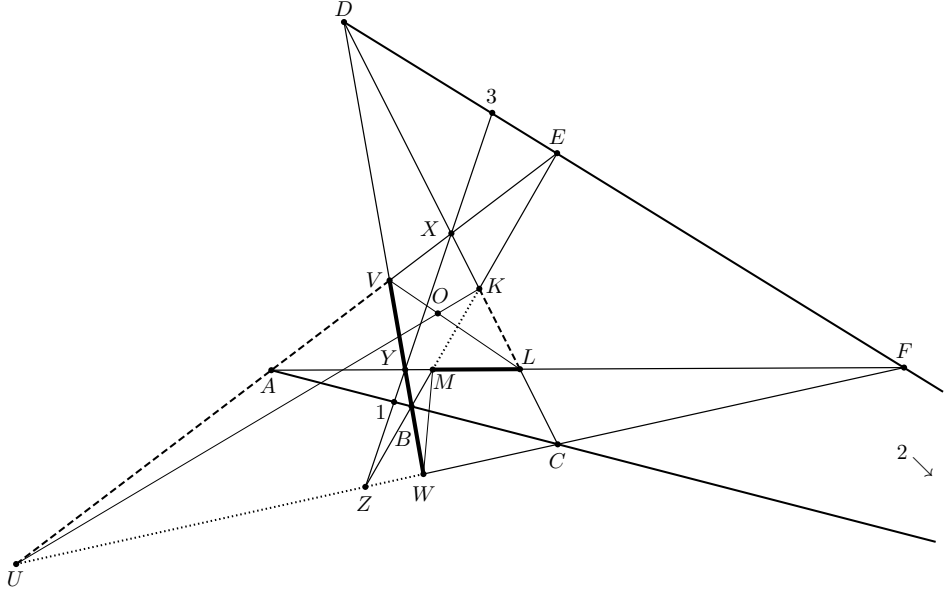
\begin{center}
\psscalebox{.27 .27} 
{
\begin{pspicture}(-8,-26)(38,3)

\psline[linewidth=0.1,dimen=outer,linestyle=dashed,dash=10pt 5pt](-7.66,-23.98)(10.53,-10.05)
\psline[linewidth=0.1,dimen=outer,linestyle=dashed,dash=10pt 5pt](14.9,-10.46)(16.88,-14.41)
\psline[linewidth=0.2,dimen=outer](10.53,-10.05)(12.18,-19.56)
\psline[linewidth=0.2,dimen=outer](16.88,-14.41)(12.63,-14.43)
\psline[linewidth=0.1,dimen=outer,linestyle=dashed,dash=2pt 4pt](-7.66,-23.98)(12.18,-19.56)
\psline[linewidth=0.1,dimen=outer,linestyle=dashed,dash=2pt 4pt](14.9,-10.46)(12.63,-14.43)

\psline[linewidth=0.06,dimen=outer](12.63,-14.43)(4.77,-14.46)
\psline[linewidth=0.06,dimen=outer](10.53,-10.05)(18.69,-3.8)
\psline[linewidth=0.06,dimen=outer](14.9,-10.46)(8.31,2.64)
\psline[linewidth=0.06,dimen=outer](12.63,-14.43)(9.35,-20.19)
\psline[linewidth=0.06,dimen=outer](12.18,-19.56)(35.59,-14.33)
\psline[linewidth=0.06,dimen=outer](10.53,-10.05)(8.31,2.64)
\psline[linewidth=0.06,dimen=outer](16.88,-14.41)(35.59,-14.33)
\psline[linewidth=0.06,dimen=outer](16.88,-14.41)(18.72,-18.1)

\psline[linewidth=0.03,dimen=outer](10.53,-10.05)(16.88,-14.41)
\psline[linewidth=0.03,dimen=outer](12.18,-19.56)(12.63,-14.43)
\psline[linewidth=0.03,dimen=outer](-7.66,-23.98)(14.9,-10.46)

\psline[linewidth=0.1,dimen=outer](4.77,-14.46)(37.13,-22.89)
\psline[linewidth=0.1,dimen=outer](8.31,2.64)(37.49,-15.52)
\psline[linewidth=0.06,dimen=outer](9.35,-20.19)(15.53,-1.82)
\psline[linewidth=0.06,dimen=outer](18.69,-3.8)(14.9,-10.46)

\rput(4.6,-15.2){\psscalebox{3}{$A$}} \rput(11.2,-17.8){\psscalebox{3}{$B$}} \rput(18.7,-18.9){\psscalebox{3}{$C$}}
\rput(8.3,3.3){\psscalebox{3}{$D$}}
\rput(18.7,-3){\psscalebox{3}{$E$}}
\rput(35.6,-13.5){\psscalebox{3}{$F$}}
\rput(12.5,-7.5){\psscalebox{3}{$X$}}
\rput(10.5,-13.85){\psscalebox{3}{$Y$}}
\rput(9.3,-21){\psscalebox{3}{$Z$}}
\rput(10.1,-16.5){\psscalebox{3}{$1$}}
\rput(35.5,-18.5){\psscalebox{3}{$2$}}
\rput(36.5,-19.2){\psscalebox{3}{$\searrow$}}
\rput(15.5,-1){\psscalebox{3}{$3$}}
\rput(15.7,-10.3){\psscalebox{3}{$K$}}
\rput(17.3,-13.8){\psscalebox{3}{$L$}}
\rput(13.2,-15){\psscalebox{3}{$M$}}
\rput(-7.7,-24.7){\psscalebox{3}{$U$}}
\rput(9.8,-10){\psscalebox{3}{$V$}}
\rput(12.2,-20.5){\psscalebox{3}{$W$}}
\rput(13,-10.8){\psscalebox{3}{$O$}}

\rput(4.77,-14.46){\pscircle*{.15}} 
\rput(11.6,-16.24){\pscircle*{.15}} 
\rput(18.72,-18.1){\pscircle*{.15}} 
\rput(8.31,2.64){\pscircle*{.15}} 
\rput(18.69,-3.8){\pscircle*{.15}} 
\rput(35.59,-14.33){\pscircle*{.15}} 
\rput(13.54,-7.74){\pscircle*{.15}} 
\rput(11.29,-14.43){\pscircle*{.15}} 
\rput(9.35,-20.19){\pscircle*{.15}} 
\rput(10.75,-16.02){\pscircle*{.15}} 
\rput(15.53,-1.82){\pscircle*{.15}} 
\rput(14.9,-10.46){\pscircle*{.15}} 
\rput(16.88,-14.41){\pscircle*{.15}} 
\rput(12.63,-14.43){\pscircle*{.15}} 
\rput(-7.66,-23.98){\pscircle*{.15}} 
\rput(10.53,-10.05){\pscircle*{.15}} 
\rput(12.18,-19.56){\pscircle*{.15}} 
\rput(12.88,-11.67){\pscircle*{.15}} 

\end{pspicture}
}
\end{center}

Let us prove that this interpretation satisfies the first eight formulae of this sequent. The proofs of the negative statements that we mention below require a lot of space and go beyond the scope of this paper. They are secondary for incidence results, and are usually assumed as a kind of general position of points involved in such results. However, we pay attention to all the positive statements involved in this example. The Menelaus theorem is tacitly used during the proof.

In order to show that our interpretation satisfies $(1,2,3,E,X,A)$, note that 1, 2 and 3 are not collinear and that 1, 2, 3, $E$, $X$, $A$ are mutually distinct.
The points $E$, 2, 3 are collinear by the definition of 2 and 3, the points $X$, 1, 3 are collinear by the definition of 1 and 3, and the points $A$, 1, 2
are collinear by the definition of 1 and 2. Finally, the points $E$, $X$, $A$ are collinear by the definition of $X$. We proceed analogously for the next four
formulae of the sequent.

In order to show that $(1,2,3,F,Y,A)\leftrightarrow(K,X,Z,Y,M,L)$ is satisfied by the interpretation, note that, by reasoning as in the preceding paragraph, the left-hand side of this equivalence is satisfied iff the points $F$, $Y$, $A$ are collinear. By the definition of $L$ and $M$, we have that the lines $AF$ and $LM$ coincide. Hence, the above condition is equivalent to the statement that $L$, $Y$, $M$ are collinear.

By reasoning as in the preceding paragraph, it is
evident that $L$, $Y$, $M$ are collinear iff the interpretation
satisfies $(K,L,M,Y,Z,X)$. This means that the interpretation
satisfies the above equivalence.

In order to show that our interpretation satisfies $(O,U,V,X,L,K)$, note that $O$, $U$, $V$ are not collinear
and that $O$, $U$, $V$, $X$, $L$, $K$ are mutually
distinct. The lines $UV$ and $AE$ coincide, hence, the definition
of $X$ implies that $X$, $U$, $V$ are collinear.
The definition of $O$ implies that $K$, $O$, $U$ as well as $L$,
$O$, $V$ are collinear. Finally, since the lines $CD$ and $KL$ coincide, the
definition of $X$ implies that $K$, $L$, $X$ are collinear.

In order to show that our interpretation satisfies $(U,V,W,Y,Z,X)$, note that $U$, $V$, $W$ are not collinear
and that $U$, $V$, $W$, $Y$, $Z$, $X$ are mutually
distinct. The lines $VW$ and $BD$ coincide, hence, the definition
of $Y$ implies that $Y$, $V$, $W$ are collinear. The lines $UW$ and $CF$ coincide, hence, the definition of $Z$
implies that $Z$, $U$, $W$ are collinear. As in the preceding paragraph, we have that $X$, $U$, $V$ are collinear.
Finally, since the lines $XZ$ and $13$ coincide, the
definition of $Y$ implies that $X$, $Y$, $Z$ are collinear.

By the soundness result, our interpretation satisfies $(U,Z,K,M,O,W)$, which means that the line ${W\!M}$ contains the intersection point of $KU$ and $LV$, and these three lines are concurrent.
\end{proof}
\end{example}

\section{Decidability}\label{secdecidability}
The aim of this section is to prove the following result.
\begin{prop}\label{decidability}
The Menelaus system is decidable.
\end{prop}

We say that a derivation in the Menelaus system is \emph{normal}, when neither $\diskon$-in\-tro\-duc\-tion, nor $\leftrightarrow$-introduction  precedes an application of a cut rule in this derivation. By a formula in a derivation, we mean here a particular \emph{occurrence} of this formula in the derivation.

\begin{lem}\label{normalisation1}
If the last inference rule in a derivation is
\[
\f{\vdash\Gamma \quad \vdash\Delta}{\vdash\Gamma,\Delta}
\]
and there are no other applications of cut in this derivation
preceded by some $\diskon$-introduction or
$\leftrightarrow$-introduction, then there is a normal derivation
of the sequent $\vdash\Gamma,\Delta$.
\end{lem}

\begin{proof}
We proceed by induction on the number $n\geq 0$ of $\diskon$-introductions and $\leftrightarrow$-introductions in this derivation.
If $n=0$, then the derivation is already normal. If $n>0$, then by the assumption, one of the premisses is obtained by either $\diskon$-introduction or $\leftrightarrow$-introduction. If the end of the derivation is of the form
\[
\f{\f{\vdash\Gamma',\gamma_1\quad \vdash\Gamma',\gamma_2}{\vdash\Gamma',\gamma_1\diskon \gamma_2}\quad \afrac{\vdash\Delta}}{\vdash\Gamma',\gamma_1\diskon\gamma_2,\Delta},
\]
then we transform it into the derivation ending as
\[
\f{\f{\vdash\Gamma',\gamma_1\quad \vdash\Delta}{\vdash\Gamma',\gamma_1,\Delta}\quad \f{\vdash\Gamma',\gamma_2 \quad \vdash\Delta}{\vdash\Gamma',\gamma_2,\Delta}}{\vdash\Gamma', \gamma_1\diskon\gamma_2,\Delta},
\]
where we can apply the induction hypothesis to the subderivations ending with $\vdash\Gamma',\gamma_1,\Delta$ and $\vdash\Gamma',\gamma_2,\Delta$.

If the end of the derivation is of the form
\[
\f{\f{\vdash\Gamma',\gamma_1\quad \vdash\Gamma'',\gamma_2}{\vdash\Gamma',\Gamma'',\gamma_1\leftrightarrow \gamma_2}\quad \afrac{\vdash\Delta}}{\vdash\Gamma',\Gamma'',\gamma_1\leftrightarrow\gamma_2, \Delta},
\]
then we transform it into the derivation ending as
\[
\f{\f{\vdash\Gamma',\gamma_1\quad \vdash\Delta}{\vdash\Gamma',\gamma_1, \Delta}\quad \afrac{\vdash\Gamma'',\gamma_2}}{\vdash\Gamma',\Gamma'',\gamma_1\leftrightarrow\gamma_2, \Delta},
\]
where we can apply the induction hypothesis to the subderivation ending with $\vdash\Gamma',\gamma_1, \Delta$.
\end{proof}

The following definitions consider the cut whose cut formula is \emph{not} empty. Let the \emph{degree} of a cut be the number of occurrences of $\diskon$ and $\leftrightarrow$ in the cut formula. For the cut rule and $\leftrightarrow$-introduction, every formula of the lower sequent, except the \emph{principal} formula $\varphi\leftrightarrow\psi$ in the case of $\leftrightarrow$-introduction, has a unique \emph{successor}, an occurrence of the same formula, in the upper sequent. In the case of $\diskon$-introduction, every formula of the lower sequent, except the \emph{principal} formula $\varphi\diskon\psi$, has two \emph{successors}, occurrences of the same formula, in the upper sequent. Let the \emph{rank} of a formula in a derivation be the number of formulae that are related to this formula by the reflexive and transitive closure of the successor relation. Let the \emph{rank} of a cut rule in a derivation be the sum of the ranks of the cut formulae in both premisses of this cut.

For the proof of the following lemma, we use a procedure akin to the cut-elimination procedure introduced by Gentzen, \cite{G35}, which corresponds to cut-dis\-in\-te\-gra\-tion of Do\v sen, \cite[Section~1.8.1]{D99}.

\begin{lem}\label{normalisation2}
If the last inference rule in a derivation is
\[
\f{\vdash\Gamma,\varphi \quad \vdash\Delta,\varphi}{\vdash\Gamma,\Delta}
\]
and there are no other applications of cut in this derivation preceded by some $\diskon$-introduction or $\leftrightarrow$-introduction, then there is a normal derivation of the sequent $\vdash\Gamma,\Delta$.
\end{lem}

\begin{proof}
We proceed by induction on the lexicographically ordered pairs $(d,r)$, where $d\geq 0$ is the degree and $r\geq 2$ is the rank of this cut.
The basis of this induction, i.e.\ the case when $(d,r)=(0,2)$, holds: in this case the last inference rule is not preceded by $\diskon$-introduction or $\leftrightarrow$-introduction, since both premisses must be axiomatic sequents and the derivation is already normal.

If $r>2$, then it is possible that the end of our derivation is of the form
\[
\f{\f{\vdash\Gamma',\varphi,\psi\quad \vdash\Gamma',\varphi,\chi}{\vdash \Gamma',\psi\diskon\chi,\varphi}\quad \afrac{\vdash\Delta,\varphi}}{\vdash\Gamma',\psi\diskon\chi,\Delta},
\]
and we transform it into the derivation ending as
\[
\f{\f{\vdash\Gamma',\varphi,\psi \quad \vdash\Delta,\varphi}{\vdash\Gamma',\psi,\Delta} \quad \f{\vdash\Gamma',\varphi,\chi \quad \vdash\Delta,\varphi}{\vdash\Gamma',\chi,\Delta}}{\vdash\Gamma', \psi\diskon\chi,\Delta},
\]
where both applications of cut have the same degree but lower ranks. If the end of our derivation is of the form
\[
\f{\f{\vdash\Gamma',\varphi,\psi\quad \vdash\Gamma'',\chi}{\vdash \Gamma',\Gamma'',\psi\leftrightarrow\chi,\varphi}\quad \afrac{\vdash\Delta,\varphi}}{\vdash\Gamma',\Gamma'',\psi \leftrightarrow\chi,\Delta},
\]
then we transform it into the derivation ending as
\[
\f{\f{\vdash\Gamma',\varphi,\psi\quad \vdash\Delta,\varphi}{\vdash \Gamma',\Delta,\psi}\quad \afrac{\vdash\Gamma'',\chi}}{\vdash\Gamma',\Gamma'',\psi \leftrightarrow\chi,\Delta},
\]
where the new cut has the same degree but a lower rank.

On the other hand, if we assume that the end of our derivation is of the form
\[
\f{\f{\vdash\Gamma_1,\varphi,\psi\quad \vdash \Gamma_2,\psi}{\vdash \Gamma_1,\Gamma_2,\varphi}\quad \afrac{\vdash\Delta,\varphi}}{\vdash\Gamma_1,\Gamma_2,\Delta} \quad \mbox{\rm or} \quad
\f{\f{\vdash\Gamma_1,\varphi\quad \vdash \Gamma_2}{\vdash \Gamma_1,\Gamma_2,\varphi}\quad \afrac{\vdash\Delta,\varphi}}{\vdash\Gamma_1,\Gamma_2,\Delta},
\]
then, since $\diskon$-introduction and $\leftrightarrow$-introduction are not applied in the subderivations ending with $\vdash\Gamma_1,\Gamma_2,\varphi$, the formula $\varphi$ must be atomic. If our derivation is not normal, then the subderivation ending with $\vdash\Delta,\varphi$ must have $\diskon$-introduction or $\leftrightarrow$-introduction as the last rule and we are again in the former situation.

If $r=2$ and $d>0$, then we have two possibilities. If the end of our derivation is of the form
\[
\f{\f{\vdash\Gamma,\varphi_1\quad \vdash\Gamma,\varphi_2}{\vdash\Gamma,\varphi_1\diskon\varphi_2}\quad \f{\vdash\Delta,\varphi_1\quad \vdash\Delta,\varphi_2}{\vdash\Delta,\varphi_1\diskon\varphi_2}}{\vdash \Gamma,\Delta},
\]
then we transform it into the derivation whose end is the cut
\[
\f{\Gamma,\varphi_1\quad \vdash\Delta,\varphi_1}{\vdash\Gamma,\Delta}
\]
with a lower degree. If the end of our derivation is of the form
\[
\f{\f{\vdash\Gamma',\varphi_1\quad \vdash\Gamma'',\varphi_2}{\vdash\Gamma',\Gamma'',\varphi_1 \leftrightarrow\varphi_2}\quad \f{\vdash\Delta',\varphi_1\quad \vdash\Delta'',\varphi_2}{\vdash\Delta',\Delta'',\varphi_1 \leftrightarrow\varphi_2}}{\vdash\Gamma',\Gamma'',\Delta',\Delta''},
\]
then we transform it into the derivation whose end is of the form
\[
\f{\f{\vdash\Gamma',\varphi_1\quad \vdash\Delta',\varphi_1}{\vdash\Gamma',\Delta'}\quad \f{\vdash\Gamma'',\varphi_2\quad \vdash\Delta'',\varphi_2}{\vdash\Gamma'',\Delta''}}{\vdash\Gamma',\Gamma'', \Delta',\Delta''},
\]
where both upper cuts are of lower degree. By the induction hypothesis, there are two normal derivations ending with $\vdash\Gamma',\Delta'$ and $\vdash\Gamma'',\Delta''$, and, by Lemma~\ref{normalisation1}, there is a normal derivation of $\vdash\Gamma',\Gamma'', \Delta',\Delta''$.
\end{proof}

\begin{cor}\label{atomisation}
For every derivation of a sequent, there is a normal derivation of the same sequent.
\end{cor}

For a multiset $\Gamma$ of formulae, let $\lambda(\Gamma)$ be the set of elements of $W$ occurring in~$\Gamma$ and let $\kappa(\Gamma)$ be the number of elements (possibly with repetition) of $\Gamma$.

\begin{lem}\label{materijal}
If $\vdash\Delta,\varphi$ is derivable, then $\lambda(\{\varphi\})\subseteq \lambda(\Delta)$.
\end{lem}
\begin{proof}
Note that this property holds for the axiomatic sequents and the only interesting case in an inductive proof of this fact is when a derivation of $\vdash\Delta,\varphi$ ends as
\[
\f{\vdash\Delta',\psi\quad \vdash\Delta'',\varphi,\psi}{\vdash\Delta',\Delta'',\varphi}.
\]
If $A\in\lambda(\{\varphi\})$ and $A\not\in\lambda(\Delta'')$, then $A\in\lambda(\{\psi\})$, and by the induction hypothesis applied to the subderivation ending with $\vdash\Delta',\psi$, we have that $A\in\lambda(\Delta')$.
\end{proof}

\begin{lem}\label{materijal2}
For every sequent $\vdash\Delta$ that occurs in a derivation of $\vdash \Gamma$, we have that $\lambda(\Delta)\subseteq \lambda(\Gamma)$ and $2\leq\kappa(\Delta)\leq\kappa(\Gamma)$.
\end{lem}
\begin{proof}
By Lemma~\ref{materijal}, every application of cut preserves the letters from the premisses, and it is obvious that the other rules satisfy this property. Also, all the rules are such that $2\leq\kappa(\Delta)\leq\kappa(\Gamma)$ holds for $\vdash\Delta$ being a premise of the rule and $\vdash\Gamma$ being its conclusion.
\end{proof}

The following is usually presupposed for a formal system.

\begin{lem}\label{decidaxiom}
The set of axiomatic sequents is decidable.
\end{lem}
\begin{proof}
It is straightforward to check that the properties (0)-(5) of
$\Delta$-complexes that define the $\mathcal{M}$-complexes are
decidable. Hence the set of axiomatic sequents coming from
$\mathcal{M}$-complexes is decidable. Also, it is straightforward
to check whether a sequent is an instance of the identity sequent
\eqref{identity} or of the two axiomatic schemata
\eqref{perm-swit} corresponding to permutation of vertices and
switching of triangles. Therefore, the set of all axiomatic
sequents is recursive.
\end{proof}

Let the \emph{atomic} Menelaus system be defined as the original system, save that $\diskon$ and $\leftrightarrow$ are omitted from the language, and $\diskon$-introduction and $\leftrightarrow$-introduction are omitted from the set of rules. The sequents of this system are called \emph{atomic}.

\begin{lem}
\label{decidatom}
The atomic Menelaus system is decidable.
\end{lem}
\begin{proof}
Let $\vdash\Gamma$ be an atomic sequent. By Lemma~\ref{materijal2}, if $\vdash\Delta$ occurs in a derivation of $\vdash\Gamma$, then $\lambda(\Delta)\subseteq\lambda(\Gamma)$ and $\kappa(\Delta)\leq \kappa(\Gamma)$. Let $S$ be the set of atomic sequents
\[
\{\vdash\Delta\mid \lambda(\Delta)\subseteq\lambda(\Gamma)\;\mbox{\rm and } \kappa(\Delta)\leq \kappa(\Gamma)\}.
\]
Note that $S$ is finite. Then the decision procedure may be carried out in the following way (cf.\ \cite[Section IV.1.2]{G35}).

Let $S_0\subseteq S$ be the set of axiomatic sequents in $S$. (Recall that the set of all axiomatic sequents is decidable by Lemma~\ref{decidaxiom}.) If $(\vdash\Gamma)\in S_0$, then we are done, and $\vdash\Gamma$ is derivable. If not, then let $S_1$ contain all the elements of $S_0$ and all the sequents from $S$ obtained from two $S_0$ sequents by a single application of cut. If $S_1=S_0$, then $\vdash\Gamma$ is not derivable, otherwise we proceed in this manner until either $\vdash\Gamma$ appears as a member of some $S_i$, in which case it is derivable, or the procedure yields no more derivable sequents. In the later case the sequent $\vdash \Gamma$ is not derivable in the atomic Menelaus system.
\end{proof}

We say that the formula $A$ is a \emph{subformula} of $\Gamma$ if $A$ is a subformula of some formula in $\Gamma$.

\begin{proof}[Proof of Proposition~\ref{decidability}] Let $\vdash\Gamma$ be a sequent. By Corollary~\ref{atomisation}, for every derivation of $\vdash\Gamma$ there exists a normal derivation of the same sequent, i.e.\ a derivation that can be divided into \textit{atomic} and \textit{non-atomic} part. Note that if $\vdash\Delta$ occurs in a non-atomic part of such a derivation, then all the formulae in $\Delta$ are subformulae of $\Gamma$. Moreover, by Lemma~\ref{materijal2}, we know that $\kappa(\Delta)\leq\kappa(\Gamma)$.

Let $S$ be the following set of sequents
\[
\{\vdash\Delta\mid\; \mbox{\rm every formula in } \Delta \; \mbox{\rm is a subformula of }\Gamma\; \mbox{\rm and }\kappa(\Delta)\leq\kappa(\Gamma)\}.
\]
Note that $S$ is again finite and we proceed as in the proof of Lemma~\ref{decidatom}, save that for $S_0$ we take not just the axiomatic sequents but all the sequents from $S$ derivable in the atomic Menelaus system. (By Lemma~\ref{decidatom}, this set is decidable.) If $\vdash\Gamma$ is not in $S_i$, then $S_{i+1}$ contains all the elements of $S_i$ and all the sequents from $S$ obtained from two $S_i$ sequents by a single application of $\diskon$-introduction or $\leftrightarrow$-introduction.
\end{proof}

Thanks to the results of this section, we can now give a formal
proof of the non-derivability of sequent (\ref{unprov1}).

\begin{example}\label{exunprov1}
Let us recall here this sequent:
\[ \vdash (A,B,P,C,X,R),
(A,C,P,B,X,Q), (B,R,C,X,P,A), (A,R,C,X,Q,B).
\]

Our proof of underivability does not follow the algorithm developed in this section completely, but it uses some observations shortening the procedure.

According to Corollary~\ref{atomisation}, derivability of
\ref{unprov1} implies existence of a normal derivation for it,
which in this case means an atomic derivation of this sequent.
Note that every sequent derivable in the atomic Menelaus system is
even, since the axiomatic sequents are such, and the cut rules
preserve this property. Also, every two element sequent must
contain elements from the same $G$-orbit (see the end of
Section~\ref{permutations}). Since no pair of elements of
\ref{unprov1} belongs to the same $G$-orbit, the last rule in such
a derivation of \ref{unprov1} cannot be a cut with empty cut
formula. We conclude that either \ref{unprov1} is an axiomatic
sequent or it is obtained from a four element premise and a two
element premise by a cut whose cut formula is not empty. This
means that \ref{unprov1} is, up to the action of $G$, an axiomatic
sequent, which is not possible for the following reasons.

There are 7 elements of $W$ involved in the sequent \ref{unprov1}. If that sequent, up to the action of $G$, is delivered from an $\mathcal{M}$-complex $L$, then we have that $|L_2|=4$ and $|L_1|+|L_0|=7$. Since the Euler characteristic $|L_2|-|L_1|+|L_0|$ of the geometric realisation of $L$ (cf.\ Proposition~\ref{surface}) is equal to $2-2g$, where $g$ is the genus of the obtained surface, we have that $|L_2|-|L_1|+|L_0|$ is even. Therefore, $|L_0|-|L_1|$ is even, which means that $|L_1|$ and $|L_0|$ cannot be integers since $|L_1|+|L_0|=7$.
\end{example}

\begin{example}
Consider the following incidence result.

\vspace{1ex}

\noindent\textit{For a triangle $ABC$ and points $D,E,F,G,H,K$ on its sides, as it is illustrated in the following picture, let $AB\cap FK=\{N\}$, $AB\cap GH=\{Q\}$, $BC\cap EH=\{L\}$, $BC\cap DK=\{P\}$, $AC\cap DG=\{M\}$ and $AC\cap EF=\{O\}$. Then the points $M$, $L$ and $N$ are collinear iff the points $O$, $P$ and $Q$ are collinear.
}

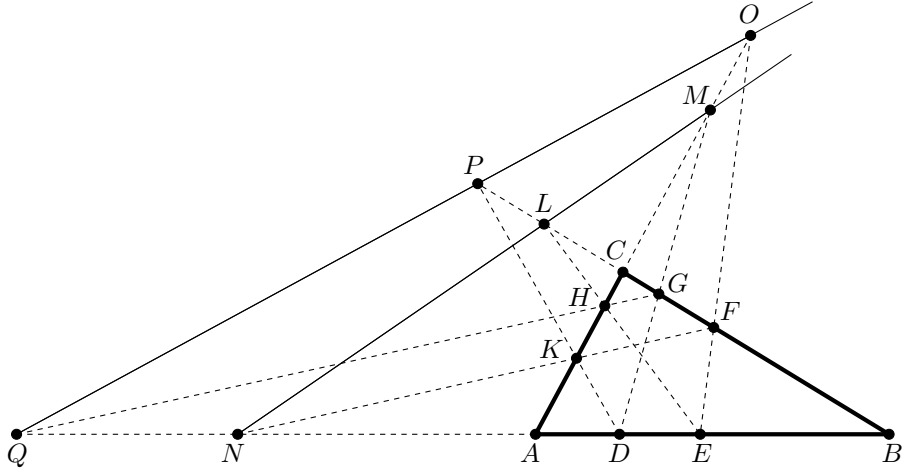
\begin{center}
\begin{tikzpicture}[scale=1.25][line cap=round,line join=round,>=triangle 45,x=1.0cm,y=1.0cm]
\draw [line width=1.6pt] (2.46,-3.0)-- (3.38,-1.28);
\draw [line width=1.6pt] (6.18,-3.0)-- (3.38,-1.28);
\draw [line width=1.6pt] (2.46,-3.0)-- (6.18,-3.0);
\draw [dash pattern=on 2pt off 2pt] (3.38,-1.28)-- (4.723199999999999,1.2311999999999994);
\draw [dash pattern=on 2pt off 2pt] (4.723199999999999,1.2311999999999994)-- (4.193333333333335,-3.0);
\draw [dash pattern=on 2pt off 2pt] (3.34439184746877,-3.0000000000000004)-- (4.3,0.4399999999999995);
\draw [dash pattern=on 2pt off 2pt] (3.38,-1.28)-- (1.8523718791064396,-0.34159986859395547);
\draw [dash pattern=on 2pt off 2pt] (3.34439184746877,-3.0000000000000004)-- (1.8523718791064396,-0.34159986859395547);
\draw [dash pattern=on 2pt off 2pt] (2.5502395964691047,-0.7702900378310215)-- (4.193333333333335,-3.0);
\draw (-0.6733333333333333,-3.0)-- (4.3,0.4399999999999995);
\draw (-3.0,-3.0)-- (4.723199999999999,1.2311999999999994);
\draw [dash pattern=on 2pt off 2pt] (-3.0,-3.0)-- (2.46,-3.0);
\draw [dash pattern=on 2pt off 2pt] (-3.0,-3.0)-- (3.7575108707909206,-1.5118995349144224);
\draw [dash pattern=on 2pt off 2pt] (-0.6733333333333333,-3.0)-- (4.334807207686513,-1.8665244275788577);
\draw (-3.0,-3.0)-- (5.374928407829859,1.5882531954642758);
\draw (5.15266629919752,1.0297799066835385)-- (-0.6733333333333333,-3.0);
\draw (2.2,-3) node[anchor=north west] {$A$};
\draw (6,-3) node[anchor=north west] {$B$};
\draw (3.1,-0.85) node[anchor=north west] {$C$};
\draw (3.12,-3) node[anchor=north west] {$D$};
\draw (4,-3) node[anchor=north west] {$E$};
\draw (4.3,-1.5) node[anchor=north west] {$F$};
\draw (3.75,-1.2) node[anchor=north west] {$G$};
\draw (2.7,-1.35) node[anchor=north west] {$H$};
\draw (2.4,-1.9) node[anchor=north west] {$K$};
\draw (-0.95,-3) node[anchor=north west] {$N$};
\draw (2.35,-0.35) node[anchor=north west] {$L$};
\draw (3.9,0.8) node[anchor=north west] {$M$};
\draw (-3.2,-3) node[anchor=north west] {$Q$};
\draw (1.6,0.05) node[anchor=north west] {$P$};
\draw (4.5,1.65) node[anchor=north west] {$O$};

\begin{scriptsize}
\draw [fill=black] (2.46,-3.0) circle (1.5pt);
\draw [fill=black] (6.18,-3.0) circle (1.5pt);
\draw [fill=black] (3.38,-1.28) circle (1.5pt);
\draw [fill=black] (-0.6733333333333333,-3.0) circle (1.5pt);
\draw [fill=black] (4.3,0.4399999999999995) circle (1.5pt);
\draw [fill=black] (-3.0,-3.0) circle (1.5pt);
\draw [fill=black] (4.723199999999999,1.2311999999999994) circle (1.5pt);
\draw [fill=black] (1.8523718791064396,-0.34159986859395547) circle (1.5pt);
\draw [fill=black] (2.5502395964691047,-0.7702900378310215) circle (1.5pt);
\draw [fill=black] (4.193333333333335,-3.0) circle (1.5pt);
\draw [fill=black] (3.34439184746877,-3.0000000000000004) circle (1.5pt);
\draw [fill=black] (2.891560444295605,-2.193169604143001) circle (1.5pt);
\draw [fill=black] (3.1889958791616024,-1.637094660697874) circle (1.5pt);
\draw [fill=black] (3.7575108707909206,-1.5118995349144224) circle (1.5pt);
\draw [fill=black] (4.334807207686513,-1.8665244275788577) circle (1.5pt);
\end{scriptsize}
\end{tikzpicture}
\end{center}

By reasoning as for validity of the sequent~\ref{unprov1}, one can prove that the following sequent is valid.
\begin{equation*}
\begin{array}{rl}
\vdash\!\!\!\!\!\! & (A,B,C,L,M,N), (A,B,C,L,H,E), (A,B,C,P,K,D), (A,B,C,G,M,D),
\\
 & (A,B,C,F,O,E), (A,B,C,G,H,Q), (A,B,C,F,K,N), (A,B,C,P,O,Q)
\end{array}
\end{equation*}
This suffices for the proof of the above result. However, by
reasoning as in the proof of underivability of \ref{unprov1}, the
above sequent is not derivable. This provides a geometrically more
interesting example of a valid and underivable sequent. The above
incidence result is equivalent to Pascal's theorem about six
points on a conic (for more details see, for example, \cite{DB}).
It witnesses that such incidence results are not derivable in the
Menelaus system.
\end{example}

\section{The Menelaus cyclic operad} \label{cyclic}
\noindent

In this section, we study the algebraic structure on  ${\mathcal
M}$-complexes induced by the connected sum construction. Connected
sums are widely used in topology. Here we express them
combinatorially. We show that the connected sum of two ${\mathcal
M}$-complexes is again an ${\mathcal M}$-complex, and that this
gives rise to a cyclic operad structure. We establish a tight link
between the connected sum construction and the first cut rule of
the Menelaus system.

We first recall the notion of cyclic operad, then we show how to
endow ${\mathcal M}$-complexes with a  cyclic operad structure,
for which we provide a presentation by generators and relations.
The generators are the indecomposable ${\mathcal M}$-complexes
i.e., those ${\mathcal M}$-complexes that cannot be  expressed as
a connected sum  of two simpler complexes. We  show that these
generators do not freely generate the structure, and discuss the
proof-theoretical significance of this negative result.  Finally,
we complete the description of the presentation, in two steps:  we
provide a combinatorial description of indecomposable ${\mathcal
M}$-complexes, and we identify those ${\mathcal M}$-complexes that
give rise to inherently diffferent decompositions and the
corresponding relations.

 \subsection{Cyclic operads} \label{cyclic-operad-def}
Cyclic operads are mathematical objects that were introduced by
Getzler and Kapranov in \cite{Getzler:1994pn}   with the goal of
encoding algebraic structures whose operations carry no
qualitative difference between the notion of inputs and that of
output.  Their original purpose  stemmed from   algebraic
topology, but recently, cyclic operads have found applications in
many other areas of modern mathematics, such as combinatorics,
homotopy theory, higher category theory, proof theory and
mathematical physics. In categorical logic, they hide under the
name of cyclic multicategories \cite{cgr}.  We refer to
\cite[Definition 1]{co} for a formal definition, contenting
ourselves here with the  explanations given below, which should be
sufficient for a self-contained account of our results in the
Menelaus context.

\smallskip

Combinatorially, a cyclic operad  consists of operations  that
have an {\em arity}, i.e., a finite set of entries, and that can
be composed in the following way. Two operations $f$ and $g$ with
arities $X$ and $Y$, respectively,  can be composed by choosing an
entry $x$ of $f$ and an entry $y$ of  $g$, yielding an operation
$f {_x\circ_y} \,g$ with arity $(X -\{x\}) \cup (Y-\{y\})$ (where
it is assumed that $X -\{x\}$ and  $Y-\{y\}$ are disjoint), as
illustrated below:
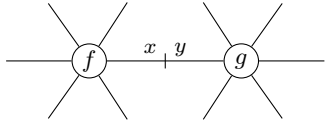
\begin{center}
\begin{tikzpicture}
 \node (f) [circle,fill=none,draw=black,minimum size=2mm,inner sep=0.4mm]  at (-1,0) {\small $f$};
\node (g) [circle,fill=none,draw=black,minimum size=2mm,inner
sep=0.7mm]  at (1,0) {\small $g$}; \node (a)
[circle,fill=none,draw=none,minimum size=2mm,inner sep=0mm]  at
(-1.6,0.85) {}; \node (w) [circle,fill=none,draw=none,minimum
size=2mm,inner sep=0mm]  at (1.6,0.85) {}; \node (b)
[circle,fill=none,draw=none,minimum size=2mm,inner sep=0mm]  at
(-2.2,0) {}; \node (c) [circle,fill=none,draw=none,minimum
size=2mm,inner sep=0mm]  at (-1.6,-0.85) {}; \node (d)
[circle,fill=none,draw=none,minimum size=2mm,inner sep=0mm]  at
(2.2,0) {}; \node (e) [circle,fill=none,draw=none,minimum
size=2mm,inner sep=0mm]  at (1.6,-0.85) {}; \node (i)
[label={[xshift=-0.2cm, yshift=-0.17cm,]{\footnotesize
$x$}},label={[xshift=0.2cm, yshift=-0.22cm,]{\footnotesize
$y$}},circle,fill=none,draw=none,minimum size=0mm,inner sep=0mm]
at (0,0.15) {}; \draw (-0.5,-0.775)--(f)--(-0.5,0.775); \draw
(0.5,-0.775)--(g)--(0.5,0.775); \draw (f) -- (g); \draw (f)--(a);
\draw (f)--(b); \draw (f)--(c); \draw (g)--(d); \draw (g)--(e);
\draw (g)--(w); \draw (0,-0.1)--(0,0.1);
\end{tikzpicture}
\end{center}
The picture features an unrooted tree with two nodes, one edge
(composed of  {\em half-edges} $x$ and $y$), and {\em open}
half-edges that are ready for being composed with other
operations. Indeed, more generally, the axioms defining a cyclic
operad guarantee that any unrooted tree  whose nodes and
half-edges are decorated by operations together with their entries
can be composed  by iteratively picking an edge (made of a pair of
half-edges) and composing the two nodes to which these half-edges
pertain, in such a way that the overall composition  is
well-defined, i.e., does not depend on the order in which edges
have been picked. It turns out that the   associativity and
commutativity axioms
$$(f\, {_{x}\circ_{y}}\,\, g)\,\,{_{u}\circ_z}\, h =  f {_{x}\circ_{z}}\,\, (g\,\,{_{u}\circ_{z}}\, h) \quad \mbox{ and } \quad f{_x\circ_y} \,g=g{_y\circ_x} f,$$
respectively, are necessary and sufficient to guarantee that the
order in which each such tree is composed does not matter. There
are also unit operations satisfying the expected laws. The full
definition  includes  one last ingredient: for each bijection
$\sigma:Z\rightarrow X$ and an operation $f$ of arity $X$, there
is an operation $f^{\sigma}$ of arity $Z$ (the action of $\sigma$
on $f$). These bijections   are required to satisfy the equality
$(f^\sigma)^\tau=f^{\sigma\circ\tau}$ and to  be compatible with
compositions and units.

Given some collection of generating operations (of certain arity),
one can build the free cyclic operad over that collection by
taking as operations all the unrooted trees built as above,
decorating   the nodes  by the generating operations. The question
of whether a given cyclic operad is {\em free} (for some
collection of generators) can  be understood as the question of
whether all its operations have a unique (up to associativity and
commutativity)  decomposition in terms of the generating
operations. The importance of  free cyclic operads lies in the
fact that any cyclic operad can be presented as the quotient of a
free operad under relations that equate precisely all the possible
decompositions, for each of the operations.

\smallskip
 \subsection{The cyclic operad of ${\mathcal M}$-complexes}\label{sneg}
 In this section, we define a cyclic operad, which we call the {\em Menelaus cyclic operad}, whose operations  are the ${\mathcal M}$-complexes, each having as arity its set of $2$-cells. More precisely, the operations of   the Menelaus cyclic operad are the  $<\!\!2$-isomorphism classes $[K]$ of ${\mathcal M}$-complexes, where a $<\!\!2$-isomorphism is a morphism  (in the category of $2$-dimensional homogeneous $\Delta$-complexes) whose $0$-th and $1$-st components are bijections and whose  $2$-nd component is the identity function (informally, the names of 1-cells or 0-cells do not matter).
%
The composition of  ${\mathcal M}$-complexes is defined in terms
of the connected sum construction  that we define below for
arbitrary  homogeneous $n$-dimensional $\Delta$-complexes.

\smallskip

 Let $X$ be an $n$-dimensional homogeneous $\Delta$-complex and let $x\in X_n$. The $\Delta$-complex $X\ominus x$ is obtained from $X$ by deleting $x$ from $X_n$ and restricting the faces $d^n_i$ to $X_n-\{x\}$.

The \emph{boundary complex} $\partial \Delta^n$ of the $n$-simplex
$\Delta^n$ is the $\Delta$-complex such that for $0\leq k\leq
n-1$, we have
\[
(\partial\Delta^n)_k=\{U\subseteq\{0,\ldots,n\}\mid |U|=k+1\},
\]
and for $U=\{j_0,\ldots,j_k\}$, such that $0\leq
j_0<\ldots<j_k\leq n$,
\[
d^k_i U=\{j_0,\ldots,j_{i-1},j_{i+1},\ldots,j_k\}.
\]

For an $n$-dimensional homogeneous $\Delta$-complex $X$ and $x\in
X_n$, let $f^x\colon \partial\Delta^n \to X\ominus x$ be a family
of functions \[ \{f^x_k\colon (\partial\Delta^n)_k\to (X\ominus
x)_k\mid 0\leq k\leq n-1\}
\]
defined so that for $i_1<\ldots<i_{n-k}$,
\[
f^x_k(\{0,\ldots,n\}-\{i_1,\ldots,i_{n-k}\})=d_{i_1}\ldots
d_{i_{n-k}} x.
\]
It is straightforward to check that $f^x$ is a morphism of
$\Delta$-complexes.

Let $X$ and $Y$ be two disjoint $n$-dimensional homogeneous
$\Delta$-complexes and let $x\in X_n$ and $y\in Y_n$. For
$f^x\colon \partial\Delta^n \to X\ominus x$ and $f^y\colon
\partial\Delta^n \to Y\ominus y$ defined as above, let $X
{_x\circ_y} Y$ be the $n$-dimensional $\Delta$-complex defined as
\[
(X {_x\circ_y} Y)_n=(X_n-\{x\})\cup(Y_n-\{y\}),\quad (X
{_x\circ_y} Y)_k=(X_k\cup Y_k)_{\sim_k}, 0\leq k\leq n-1,
\]
where the equivalence relations $\sim_k$ are generated by
$f^x_k(U)\sim_k f^y_k(U)$ for every $U\in (\partial\Delta^n)_k$.
Moreover, the faces of $X {_x\circ_y} Y$ are defined so that for
$z_{\sim_k}\in (X {_x\circ_y} Y)_k$ we have $d_i z_{\sim_k}=(d_i
z)_{\sim_k}$.

The complex $X {_x\circ_y} Y$, which is the \emph{connected sum}
of $X$ and $Y$ with respect to $x$ and $y$, is the pushout
$(X\ominus x)\bigsqcup_{\partial\Delta^n} (Y\ominus y)$ in the
category of $\Delta$-complexes. It is homogeneous, and, as we show
in Lemma \ref{close}, in the case when $X$ and $Y$ are
$\mathcal{M}$-complexes, $X {_x\circ_y} Y$ is an
$\mathcal{M}$-complex too. Informally, the connected sum is the
complex obtained by gluing the complexes $X\ominus x$ and
$Y\ominus y$ along the boundary of $x$ and the border of $y$, as
indicated in the picture below.
  \begin{center}
 \resizebox{5cm}{!}
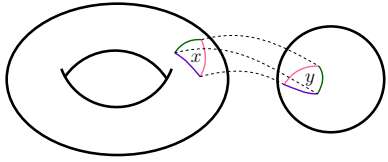
{\begin{tikzpicture}
 \draw[line width=2.8pt] (0,0) ellipse (4.15cm and 2.86cm);
 \draw[line width=2.8pt] (-2.1,0.4)  arc[radius = 2.2cm, start angle= 200, end angle= 340];
 \draw[line width=2.8pt] (1.85,0.015)  arc[radius =2.28cm, start angle= 32, end angle= 145];
  \draw[line width=2.8pt] (8,0) circle (2cm);
  \draw[draw=qqwuqq,line width=1.5pt] (2.15,1) to[bend left] (3.15,1.5);
  \draw[draw=ffwwzz,line width=1.5pt] (3.15,1.5) to[out=-70, in=60] (3.1,0.1);
  \draw[draw=wwqqcc,line width=1.5pt] (2.15,1) to[out=-30, in=130] (3.1,0.1);
    \draw[draw=ffwwzz,line width=1.5pt] (6.25,-0.2) to[bend left] (7.55,0.5);
  \draw[draw=qqwuqq,line width=1.5pt] (7.55,0.5) to[out=-775, in=60] (7.5,-0.55);
  \draw[draw=wwqqcc,line width=1.5pt] (6.25,-0.2) to[out=-20, in=180] (7.5,-0.55);
  \draw[dashed] (3.15,1.5) to [out=15, in=150] (7.55,0.5);
  \draw[dashed] (2.15,1) to [out=15, in=150] (7.5,-0.55);
   \draw[dashed] (3.1,0.1) to [out=15, in=150] (6.25,-0.2);
   \node (x) [circle,draw=none, inner sep=0,minimum size=1.15mm] at (2.95,0.8) {\Huge $x$};
      \node (y) [circle,draw=none, inner sep=0,minimum size=1.15mm] at (7.22,0) {\Huge $y$};
 \end{tikzpicture}}
\end{center}
We are now ready to define a  cyclic operadic structure on
homogeneous $n$-dimen\-sional $\Delta$-complexes. Composition is
given by the connected sums, and the   unit  operations are
$\Delta$-complexes given by two $n$-cells  having the same
boundaries. The bijections renaming the arities  act as follows.
Let $\sigma:Z_n\rightarrow X_n$ be a bijection and $K$ be an
$n$-dimensional homogeneous $\Delta$-complex such that  $K_n=X_n$.
We define a representative $K'$ of $[K]^\sigma$ as follows:
$K'_n=Z_n$, $K'_i=K_i$ for all $i<n$, and the face maps of $K'$
coincide with those of $K$ except for  the maps $d_i^n$ which for
$K'$ are defined by  $d_i^nu=d_i^n(\sigma(u))$. The  associativity
and commutativity laws  of cyclic operads are immediate
consequences of our ability to   choose the representatives of the
$<\!\!n$-isomorphism classes involved in the composition in such a
way that their collections of cells in all dimensions are
disjoint, and of the observation that union of disjoint sets is
commutative and associative ``on the nose''.

\begin{lem}\label{close}The  $\mathcal{M}$-complexes are closed under the connected
sum construction.\end{lem}
\begin{proof}
The connected sum construction clearly preserves the  properties
(0)-(4) from the definition of $\mathcal{M}$-complexes. In order
to prove that the orientability of ${\mathcal M}$-complexes $X$
and $Y$ implies the orientability of the ${\mathcal M}$-complex $X
{_x\circ_y} Y$,   let $c_1$ (resp. $c_2$) be the fundamental cycle
of  $X$ (resp. of $Y$), and let $\alpha$ (resp. $\beta$) be the
coefficient of $x$ in $c_1$ (resp. $y$ in $c_2$). Then, by the
orientability of $X$ and $Y$, the boundary of the linear
combination $(c_1-\alpha x)\pm (c_2-\beta y)$, where $\pm$ is the
Kronecker symbol of $(\alpha,\beta)$, is zero. This proves that
$H_2(X {_x\circ_y} Y;{\bf Z})={\bf Z}$, and, in particular,
determines an orientation of $X {_x\circ_y} Y$.
\end{proof}
The   Menelaus cyclic operad is the suboperad of the cyclic operad
of homogeneous $2$-dimensional $\Delta$-complexes determined by
the $\mathcal{M}$-complexes.

\medskip
The following proposition shows the link between the connected sum
construction on $\mathcal{M}$-complexes and the first cut rule of
the Menelaus system.

\begin{prop} \label{connected-cut}
Let $K$, $K_1$ and $K_2$ be ${\mathcal M}$-complexes. If
$K=K_1{_x\circ_y}K_2$, then there is a choice of representatives
in $[K]$, $[K_1]$ and $[K_2]$, giving rise to the axiomatic
sequents $\;\vdash \Gamma,  \vdash \Gamma_1$ and  $\vdash
\Gamma_2$, respectively, such that
  $\vdash \Gamma$ is obtained from  $\vdash \Gamma_1$ and  $\vdash \Gamma_2$  by applying the first cut rule of the Menelaus system.
\end{prop}
\begin{proof}
We refer to the notations of Section \ref{system}. It is easy to
see that we can choose representatives of $[K_1]$ and  $[K_2]$
(which we still call $K_1$ and $K_2$) such that their $0$-cells
and $1$-cells are all in $\mathcal{W}$, and such that the
$0$-cells and $1$-cells  that are glued with one another in the
connected sum are equal  (and are the only common cells of $K_1$
and $K_2$), so that  $\nu x=\nu y$.  Then the connected sum $K=K_1
{_x\circ_y} K_2$ is simply obtained by taking the union of the
cells of  $K_1\ominus x$ and $K_2\ominus y$.  It follows that the
axiomatic sequents associated with  $K_1$, $K_2$ and $K$ have the
form $\vdash \Gamma,\varphi$, $\vdash \Delta,\varphi$ and $\vdash
\Gamma,\Delta$, respectively, i.e. are the hypotheses and
conclusion of a first cut rule.
\end{proof}
Similarly, it is immediate to see that the units of the Menelaus
cyclic operad match the identity axioms (cf. Equation
\ref{identity}).

\medskip
Our goal is now to exhibit a presentation of the Menelaus cyclic
operad by generators and relations.  We say that an ${\mathcal
M}$-complex is {\em proper} if it is not an identity operation,
and that a  proper ${\mathcal M}$-complex is {\em indecomposable}
if  it cannot be obtained as connected sum of two proper
${\mathcal M}$-complexes.
\begin{prop}  \label{generation}
Every proper ${\mathcal M}$-complex can be obtained as the result
of composing an unrooted tree whose nodes are decorated with
indecomposable ${\mathcal M}$-complexes.
\end{prop}
\begin{proof} The statement is a consequence of the following remarks. It is easy to see that the identities are the only ${\mathcal M}$-complexes with two $2$-cells. It follows that proper ${\mathcal M}$-complexes have at least three $2$-cells, and hence, for  every (proper) decomposition $K=K_1 {_x\circ_y} K_2$, if $n,n_1$ and $n_2$ are the numbers of $2$-cells of $K$, $K_1$ and $K_2$, respectively, we deduce from
$n=n_1+n_2-2$, $n_1>2$ and $n_2>2$ that $n_1<n$ and $n_2<n$.
\end{proof}
It follows from Proposition \ref{generation} that the Menelaus
cyclic operad is generated by the indecomposable  ${\mathcal
M}$-complexes.  We shall next prove that the Menelaus cyclic
operad is not free over this collection of generators.

\subsection{A counter-example} \label{long-normal-form}

\begin{example}\label{counterexample}
Here is  an example of an  ${\mathcal M}$-complex $M$ which is
$<\!\!2$-isomorphic to two different connected sums  of two
indecomposable ${\mathcal M}$-complexes.  We define four
${\mathcal M}$-complexes $K$, $U$, $L$ and $V$ as follows:
\begin{itemize}
\item the ${\mathcal M}$-complex $K$ is defined by
$$\begin{array}{lll}
K_2=\{\alpha,\beta,\gamma,\delta\} && \left\{  \begin{array}{l} d_0^2\alpha =1\quad d_1^2\alpha=4 \quad d_2^2\alpha=3\\
d_0^2\beta=2\quad d_1^2\beta=4\quad d_2^2\beta=3\\
d_0^2\gamma=2 \quad d_1^2\gamma=6\quad d_2^2\gamma=5\\
d_0^2\delta=1\quad d_1^2\delta=6\quad d_2^2\delta=5\\
\end{array} \right.\\[1cm]
K_1=\{1,2,3,4,5,6\} && \left\{ \begin{array}{lll}
d_0^11=c\quad d_1^11=d && d_0^12=c\quad d_1^12=d\\
d_0^13=b\quad d_1^13=c && d_0^14=b\quad d_1^14=d\\
d_0^15=a\quad d_1^15=c && d_0^16=a\quad d_1^16=d
\end{array}\right.\\
K_0=\{a,b,c,d\}
\end{array}$$

\item the ${\mathcal M}$-complex $L$ is defined by
$$\quad\quad\enspace\begin{array}{lll}
L_2=\{\alpha',\beta',\gamma',\delta'\} && \left\{  \begin{array}{l} d_0^2\alpha' =4'\quad d_1^2\alpha'=2' \quad d_2^2\alpha'=6'\\
d_0^2\beta'=5'\quad d_1^2\beta'=3'\quad d_2^2\beta'=6'\\
d_0^2\gamma'=1' \quad d_1^2\gamma'=5'\quad d_2^2\gamma'=4'\\
d_0^2\delta'=1'\quad d_1^2\delta'=3'\quad d_2^2\delta'=2'\\
\end{array} \right.\\[1cm]
L_1=\{1',2',3',4',5',6'\} && \left\{ \begin{array}{l}
d_0^11'=c'\;\; d_1^11'=d'\;\; d_0^12'=a'\;\; d_1^12'=c'\\
d_0^13'=a'\;\; d_1^13'=d'\;\; d_0^14'=b'\;\; d_1^14'=c'\\
d_0^15'=b'\;\; d_1^15'=d'\;\; d_0^16'=a'\;\; d_1^16'=b'
\end{array}\right.\\
L_0=\{a',b',c',d'\}
\end{array}$$
 \begin{center}
\resizebox{3.5cm}{!}{\begin{tikzpicture} \draw[line
width=0.18em,draw=gray] (0,0) circle (3cm); \draw[line
width=0.18em,draw=qqzzff] (0,3).. controls (-1.2,1) and (-1.2,-1)
..(0,-3) node[midway,xshift=-0.2cm] {$1$}; \draw[line
width=0.18em,draw=ffxfqq] (0,3).. controls (1,1) and (1,-1)
..(0,-3) node[midway,xshift=0.2cm] {$2$}; \draw[line
width=0.18em,draw=qqqqcc] (0,3).. controls (-0.24,1) ..(-0.25,0);
\draw[line width=0.18em,draw=ffwwzz] (-0.25,0).. controls
(-0.24,-1) ..(0,-3); \node (5) at (1.7,1.5) {$5$}; \node (6) at
(1.7,-1.5) {$6$}; \node (alpha) at (-0.6,0) {$\alpha$}; \node
(beta) at (0.4,0) {$\beta$}; \node (gamma) at (2.5,0) {$\gamma$};
\node (delta) at (-2,0) {$\delta$}; \node (K) at (0,-4) {\large
$K$}; \node(A) at (0,3) [circle,draw=black,fill=black,minimum
size=4,inner sep=0,label={[xshift=0cm,yshift=0.0cm]{\large $c$}}]
{}; \node(B) at (0,-3) [circle,draw=black,fill=black,minimum
size=4,inner sep=0,label={[xshift=0cm,yshift=-0.7cm]{\large $d$}}]
{}; \node(C) at (-0.25,0) [circle,draw=black,fill=black,minimum
size=4,inner sep=0,label={[xshift=0.25cm,yshift=-0.3cm]{\large
$b$}}]  {}; \node(D) at (1.75,0)
[circle,draw=black,fill=black,minimum size=4,inner
sep=0,label={[xshift=0.25cm,yshift=-0.3cm]{\large $a$}}]  {};
\draw[line width=0.18em,dashed,draw=red] (D) to[out=270,in=25] (B)
node[midway,yshift=-1cm] {$4$}; \draw[line
width=0.18em,dashed,draw=qqccqq] (D) to[out=90,in=-25] (A)
node[midway,yshift=1cm] {$3$};
\end{tikzpicture}} \quad\quad
 \resizebox{3.5cm}{!}{\begin{tikzpicture}
\draw[line width=0.18em,draw=gray] (0,0) circle (3cm); \draw[line
width=0.18em,draw=pink] (0,3).. controls (-1.2,1) and (-1.2,-1)
..(0,-3) node[midway,xshift=-0.2cm] {$1'$}; \draw[line
width=0.18em,dashed,draw=lime] (0,3).. controls (0.5,1).. (0.55,0)
node[midway,yshift=0cm,xshift=-0.4cm] {$2'$};; \draw[line
width=0.18em,dashed,draw=qqwuqq]  (0.55,0) .. controls (0.5,-1)..
(0,-3) node[midway,yshift=0cm,xshift=-0.4cm] {$3'$}; \draw[line
width=0.18em,dashed,draw=violet] (0,3) to[out=-35,in=95] (1.35,0)
node[midway,yshift=1.3cm,xshift=1.45cm] {$4'$}; \draw[line
width=0.18em,dashed,draw=teal] (1.35,0) to[out=-95,in=35] (0,-3)
node[midway,yshift=-1.3cm,xshift=1.45cm] {$5'$}; \node (alpha) at
(0.9,1) {$\alpha'$}; \node (beta) at (0.9,-1) {$\beta'$}; \node
(gamma) at (2.5,0) {$\gamma'$}; \node (delta) at (-2,0)
{$\delta'$}; \node (L) at (0,-4) {\large $L$}; \node(A) at (0,3)
[circle,draw=black,fill=black,minimum size=4,inner
sep=0,label={[xshift=0cm,yshift=0.0cm]{\large $c'$}}]  {};
\node(B) at (0,-3) [circle,draw=black,fill=black,minimum
size=4,inner sep=0,label={[xshift=0cm,yshift=-0.7cm]{\large
$d'$}}]  {}; \node(C) at (0.55,0)
[circle,draw=black,fill=black,minimum size=4,inner
sep=0,label={[xshift=-0.25cm,yshift=-0.3cm]{\large $a'$}}]  {};
\node(D) at (1.35,0) [circle,draw=black,fill=black,minimum
size=4,inner sep=0,label={[xshift=0.275cm,yshift=-0.3cm]{\large
$b'$}}]  {}; \draw[line width=0.18em,dashed,draw=olive] (D)-- (C)
node[midway,above] {$6'$};
\end{tikzpicture}}
\end{center}

\item the ${\mathcal M}$-complex $U$ is defined by $U=K^{\sigma}$,
where $\sigma$ renames the $2$-cells  $\gamma$ and $\delta$  of
$K$ into $\gamma'$ and $\varphi$, respectively.

\item   the ${\mathcal M}$-complex $V$ is defined by $V=L^{\tau}$,
where $\tau$ renames the $2$-cells $\gamma'$ and $\delta'$  of $L$
into   $\psi$ and $\delta$, respectively.

\end{itemize}
One checks easily that $K$ and $L$ (and hence also $U$ and $V$)
are indecomposable. The compositions $K \,
_\gamma\!\circ_{\delta'} \,L$ and $U \, _\varphi\!\circ_{\psi}\,V$
are then the same operations of the Menelaus cyclic operad:
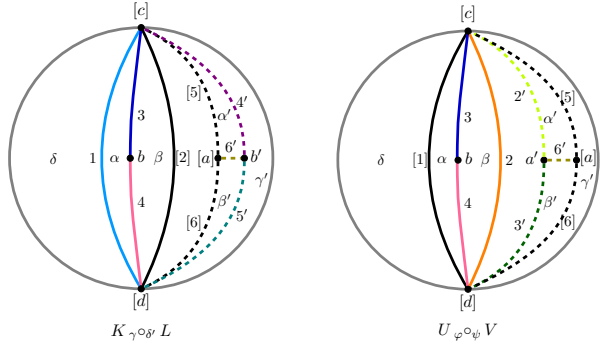
\begin{center}
\resizebox{3.5cm}{!}{\begin{tikzpicture} \draw[line
width=0.18em,draw=gray]  (0,0) circle (3cm); \draw[line
width=0.18em,draw=qqzzff]  (0,3).. controls (-1.2,1) and (-1.2,-1)
..(0,-3) node[midway,xshift=-0.2cm] {$1$}; \draw[line
width=0.18em,draw=black]  (0,3).. controls (1,1) and (1,-1)
..(0,-3) node[midway,xshift=0.2cm] {$[2]$}; \draw[line
width=0.18em,draw=qqqqcc] (0,3).. controls (-0.24,1) ..(-0.25,0);
\draw[line width=0.18em,draw=ffwwzz] (-0.25,0).. controls
(-0.24,-1) ..(0,-3); \node (5) at (1.2,1.5) {$[5]$}; \node (6) at
(1.2,-1.5) {$[6]$}; \node (alpha) at (1.9,1) {$\alpha'$}; \node
(beta) at (1.9,-1) {$\beta'$}; \node (alpha) at (-0.6,0)
{$\alpha$}; \node (beta) at (0.4,0) {$\beta$}; \node (gamma) at
(2.75,-0.5) {$\gamma'$}; \node (delta) at (-2,0) {$\delta$}; \node
(K) at (0,-4) {\large $K \,{_{\gamma}\circ_{\delta'}}\,L$};
\node(S) at (2.35,0) [circle,draw=black,fill=black,minimum
size=4,inner sep=0,label={[xshift=0.275cm,yshift=-0.3cm]{\large
$b'$}}]  {}; \node(A) at (0,3)
[circle,draw=black,fill=black,minimum size=4,inner
sep=0,label={[xshift=0cm,yshift=0.0cm]{\large $[c]$}}]  {};
\node(B) at (0,-3) [circle,draw=black,fill=black,minimum
size=4,inner sep=0,label={[xshift=0cm,yshift=-0.7cm]{\large
$[d]$}}]  {}; \node(C) at (-0.25,0)
[circle,draw=black,fill=black,minimum size=4,inner
sep=0,label={[xshift=0.25cm,yshift=-0.3cm]{\large $b$}}]  {};
\node(D) at (1.75,0) [circle,draw=black,fill=black,minimum
size=4,inner sep=0,label={[xshift=-0.275cm,yshift=-0.4cm]{\large
$[a]$}}]  {}; \draw[line width=0.18em,dashed,draw=black] (D)
to[out=270,in=25] (B) node[midway,yshift=-1cm] {$4$}; \draw[line
width=0.18em,dashed,,draw=black] (D) to[out=90,in=-25] (A)
node[midway,yshift=1cm] {$3$}; \draw[line
width=0.18em,dashed,draw=olive] (D)-- (S) node[midway,above]
{$6'$}; \draw[line width=0.18em,dashed,draw=teal] (S)
to[out=270,in=25] (B) node[midway,xshift=2.3cm,yshift=-1.3cm]
{$5'$}; \draw[line width=0.18em,dashed,draw=violet] (S)
to[out=90,in=-25] (A) node[midway,xshift=2.3cm,yshift=1.3cm]
{$4'$};
\end{tikzpicture}}   \quad\quad
 \resizebox{3.5cm}{!}{\begin{tikzpicture}
\draw[line width=0.18em,draw=gray] (0,0) circle (3cm); \draw[line
width=0.18em,draw=black] (0,3).. controls (-1.2,1) and (-1.2,-1)
..(0,-3) node[midway,xshift=-0.2cm] {$[1]$}; \draw[line
width=0.18em,draw=ffxfqq] (0,3).. controls (1,1) and (1,-1)
..(0,-3) node[midway,xshift=0.2cm] {$2$}; \draw[line
width=0.18em,draw=qqqqcc] (0,3).. controls (-0.24,1) ..(-0.25,0);
\draw[line width=0.18em,draw=ffwwzz] (-0.25,0).. controls
(-0.24,-1) ..(0,-3); \node (5) at (1.2,1.5) {$2'$}; \node (6) at
(1.2,-1.5) {$3'$}; \node (alpha) at (1.9,1) {$\alpha'$}; \node
(beta) at (1.9,-1) {$\beta'$}; \node (alpha) at (-0.6,0)
{$\alpha$}; \node (beta) at (0.4,0) {$\beta$}; \node (gamma) at
(2.75,-0.5) {$\gamma'$}; \node (delta) at (-2,0) {$\delta$}; \node
(K) at (0,-4) {\large $U \,{_{\varphi}\circ_{\psi}}\,V$}; \node(S)
at (2.5,0) [circle,draw=black,fill=black,minimum size=4,inner
sep=0,label={[xshift=0.275cm,yshift=-0.35cm]{\large $[a]$}}]  {};
\node(A) at (0,3) [circle,draw=black,fill=black,minimum
size=4,inner sep=0,label={[xshift=0cm,yshift=0.0cm]{\large
$[c]$}}]  {}; \node(B) at (0,-3)
[circle,draw=black,fill=black,minimum size=4,inner
sep=0,label={[xshift=0cm,yshift=-0.7cm]{\large $[d]$}}]  {};
\node(C) at (-0.23,0) [circle,draw=black,fill=black,minimum
size=4,inner sep=0,label={[xshift=0.25cm,yshift=-0.3cm]{\large
$b$}}]  {}; \node(D) at (1.75,0)
[circle,draw=black,fill=black,minimum size=4,inner
sep=0,label={[xshift=-0.275cm,yshift=-0.3cm]{\large $a'$}}]  {};
\draw[line width=0.18em,dashed,draw=qqwuqq] (D) to[out=270,in=25]
(B) node[midway,yshift=-1cm] {$4$}; \draw[line
width=0.18em,dashed,draw=lime] (D) to[out=90,in=-25] (A)
node[midway,yshift=1cm] {$3$}; \draw[line
width=0.18em,dashed,draw=olive] (D)-- (S) node[midway,above]
{$6'$}; \draw[line width=0.18em,dashed,draw=black] (S)
to[out=270,in=25] (B) node[midway,xshift=2.3cm,yshift=-1.45cm]
{$[6]$}; \draw[line width=0.18em,dashed,draw=black] (S)
to[out=90,in=-25] (A) node[midway,xshift=2.3cm,yshift=1.45cm]
{$[5]$};
\end{tikzpicture}}
\end{center}
Indeed, taking into account the identifications
$$2\sim 1'\enspace\enspace 5\sim 4' \enspace\enspace 6\sim 5'\quad\quad  a\sim a'\enspace\enspace c\sim c'\enspace\enspace d\sim d',$$
induced by the connected sum $K \; _\gamma\!\circ_{\delta'} \;L$,
and the identifications
$$1\sim 1'\enspace\enspace 6\sim 3'\enspace\enspace 5\sim 2' \quad\quad  a\sim b'\enspace\enspace c\sim c'\enspace\enspace d\sim d',$$ induced by the connected sum $U \,{_{\varphi}\circ_{\psi}}\,V$,  the witnessing $<\!\!2$-isomorphism between $K \, _\gamma\!\circ_{\delta'} \,L$ and $U \, _\varphi\!\circ_{\psi}\,V$  is defined as follows (where it is not the identity)
$$\begin{array}{l}
\! 1\stackrel{1}{\mapsto} [1] \quad [2]\stackrel{2}{\mapsto} 2  \quad [5]\stackrel{2'}{\mapsto} 2'\quad [6]\stackrel{3'}{\mapsto} 3' \quad  4'\stackrel{4'}{\mapsto} [5] \quad 5'\stackrel{5'}{\mapsto} [6]\\
\! [a]\stackrel{a'}{\mapsto} a'\quad b'\stackrel{b'}{\mapsto} [a]
\end{array},$$
where, above the $\mapsto$ signs, we have specified yet another
representative of the  $<\!\!2$-isomorphism class, that we shall
put to use below. It follows that $K \, _\gamma\!\circ_{\delta'}
\,L = U \, _\varphi\!\circ_{\psi}\,L$ is a relation, preventing
the Menelaus cyclic operad to be free.

In the above pictures, we have highlighted in black two distinct
triangles drawn on the sphere, which witness the respective
decompositions $K \, _\gamma\!\circ_{\delta'} \,L$ and  $U \,
_\varphi\!\circ_{\psi}\,L$. We shall study such triangles in the
rest of this section.

\smallskip

In the formalism of unrooted trees, the ${\mathcal M}$-complexes
$K$, $L$, $U$ and $V$  are represented by the corollas
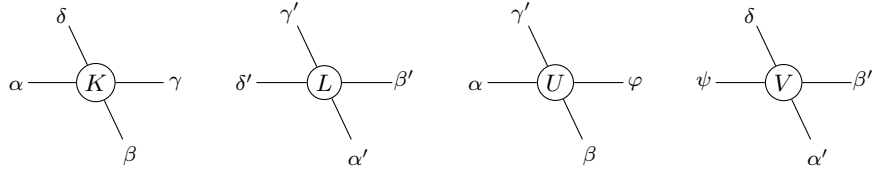
\begin{center}
\begin{tikzpicture}
 \node (f) [circle,fill=none,draw=black,minimum size=2mm,inner sep=0.5mm]  at (-1,0) {\small $K$};
\node (a) [label={[xshift=-0.05cm, yshift=-0.29cm,]{\footnotesize
$\delta$}},circle,fill=none,draw=none,minimum size=2mm,inner
sep=0mm]  at (-1.4,0.85) {}; \node (b)
[label={[xshift=-0.05cm,yshift=-0.33cm,]{\footnotesize
$\alpha$}},circle,fill=none,draw=none,minimum size=2mm,inner
sep=0mm]  at (-2,0) {}; \node (c) [label={[xshift=0.05cm,
yshift=-0.475cm,]{\footnotesize
$\beta$}},circle,fill=none,draw=none,minimum size=2mm,inner
sep=0mm]  at (-0.6,-0.85) {}; \node (d) [label={[xshift=0.05cm,
yshift=-0.33cm,]{\footnotesize
$\gamma$}},circle,fill=none,draw=none,minimum size=2mm,inner
sep=0mm]  at (0,0) {}; \draw (f)--(a); \draw (f)--(b); \draw
(f)--(c); \draw (f)--(d);
\end{tikzpicture} \quad \begin{tikzpicture}
 \node (f) [circle,fill=none,draw=black,minimum size=2mm,inner sep=0.5mm]  at (-1,0) {\small $L$};
\node (a) [label={[xshift=-0.05cm, yshift=-0.29cm,]{\footnotesize
$\gamma'$}},circle,fill=none,draw=none,minimum size=2mm,inner
sep=0mm]  at (-1.4,0.85) {}; \node (b)
[label={[xshift=-0.05cm,yshift=-0.33cm,]{\footnotesize
$\delta'$}},circle,fill=none,draw=none,minimum size=2mm,inner
sep=0mm]  at (-2,0) {}; \node (c) [label={[xshift=0.05cm,
yshift=-0.475cm,]{\footnotesize
$\alpha'$}},circle,fill=none,draw=none,minimum size=2mm,inner
sep=0mm]  at (-0.6,-0.85) {}; \node (d) [label={[xshift=0.05cm,
yshift=-0.33cm,]{\footnotesize
$\beta'$}},circle,fill=none,draw=none,minimum size=2mm,inner
sep=0mm]  at (0,0) {}; \draw (f)--(a); \draw (f)--(b); \draw
(f)--(c); \draw (f)--(d);
\end{tikzpicture} \quad \begin{tikzpicture}
 \node (f) [circle,fill=none,draw=black,minimum size=2mm,inner sep=0.5mm]  at (-1,0) {\small $U$};
\node (a) [label={[xshift=-0.05cm, yshift=-0.29cm,]{\footnotesize
$\gamma'$}},circle,fill=none,draw=none,minimum size=2mm,inner
sep=0mm]  at (-1.4,0.85) {}; \node (b)
[label={[xshift=-0.05cm,yshift=-0.33cm,]{\footnotesize
$\alpha$}},circle,fill=none,draw=none,minimum size=2mm,inner
sep=0mm]  at (-2,0) {}; \node (c) [label={[xshift=0.05cm,
yshift=-0.475cm,]{\footnotesize
$\beta$}},circle,fill=none,draw=none,minimum size=2mm,inner
sep=0mm]  at (-0.6,-0.85) {}; \node (d) [label={[xshift=0.05cm,
yshift=-0.33cm,]{\footnotesize
$\varphi$}},circle,fill=none,draw=none,minimum size=2mm,inner
sep=0mm]  at (0,0) {}; \draw (f)--(a); \draw (f)--(b); \draw
(f)--(c); \draw (f)--(d);
\end{tikzpicture} \quad  \begin{tikzpicture}
 \node (f) [circle,fill=none,draw=black,minimum size=2mm,inner sep=0.5mm]  at (-1,0) {\small $V$};
\node (a) [label={[xshift=-0.05cm, yshift=-0.29cm,]{\footnotesize
$\delta$}},circle,fill=none,draw=none,minimum size=2mm,inner
sep=0mm]  at (-1.4,0.85) {}; \node (b)
[label={[xshift=-0.05cm,yshift=-0.33cm,]{\footnotesize
$\psi$}},circle,fill=none,draw=none,minimum size=2mm,inner
sep=0mm]  at (-2,0) {}; \node (c) [label={[xshift=0.05cm,
yshift=-0.475cm,]{\footnotesize
$\alpha'$}},circle,fill=none,draw=none,minimum size=2mm,inner
sep=0mm]  at (-0.6,-0.85) {}; \node (d) [label={[xshift=0.05cm,
yshift=-0.33cm,]{\footnotesize
$\beta'$}},circle,fill=none,draw=none,minimum size=2mm,inner
sep=0mm]  at (0,0) {}; \draw (f)--(a); \draw (f)--(b); \draw
(f)--(c); \draw (f)--(d);
\end{tikzpicture}
\end{center}
respectively. The tree-wise representation of the compositions $K
\, _\gamma\!\circ_{\delta'} \,L$ and $U \, _\varphi\!\circ_{\psi
}\,V$ is   obtained by grafting the corresponding corollas along
the half-edges indicated in the two insertions, and their
identification   is reflected by imposing the equality
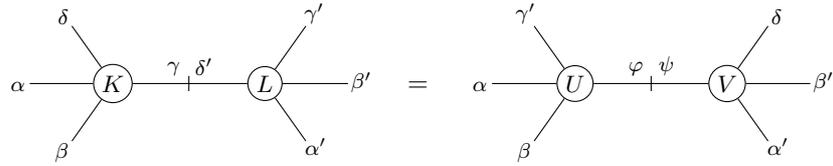
\begin{center}
\begin{tikzpicture}
 \node (f) [circle,fill=none,draw=black,minimum size=2mm,inner sep=0.5mm]  at (-1,0) {\small $K$};
\node (g) [circle,fill=none,draw=black,minimum size=2mm,inner
sep=0.5mm]  at (1,0) {\small $L$}; \node (a)
[label={[xshift=-0.05cm, yshift=-0.29cm,]{\footnotesize
$\delta$}},circle,fill=none,draw=none,minimum size=2mm,inner
sep=0mm]  at (-1.6,0.85) {}; \node (w) [label={[xshift=0.05cm,
yshift=-0.29cm,]{\footnotesize
$\gamma'$}},circle,fill=none,draw=none,minimum size=2mm,inner
sep=0mm]  at (1.6,0.85) {}; \node (b)
[label={[xshift=-0.05cm,yshift=-0.33cm,]{\footnotesize
$\alpha$}},circle,fill=none,draw=none,minimum size=2mm,inner
sep=0mm]  at (-2.2,0) {}; \node (c) [label={[xshift=-0.07cm,
yshift=-0.42cm,]{\footnotesize
$\beta$}},circle,fill=none,draw=none,minimum size=2mm,inner
sep=0mm]  at (-1.6,-0.85) {}; \node (d) [label={[xshift=0.07cm,
yshift=-0.35cm,]{\footnotesize
$\beta'$}},circle,fill=none,draw=none,minimum size=2mm,inner
sep=0mm]  at (2.2,0) {}; \node (e) [label={[xshift=0.07cm,
yshift=-0.33cm,]{\footnotesize
$\alpha'$}},circle,fill=none,draw=none,minimum size=2mm,inner
sep=0mm]  at (1.6,-0.85) {}; \node (i) [label={[xshift=-0.2cm,
yshift=-0.17cm,]{\footnotesize $\gamma$}},label={[xshift=0.2cm,
yshift=-0.17cm,]{\footnotesize
$\delta'$}},circle,fill=none,draw=none,minimum size=0mm,inner
sep=0mm]  at (0,0.15) {}; \draw (f) -- (g); \draw (f)--(a); \draw
(f)--(b); \draw (f)--(c); \draw (g)--(d); \draw (g)--(e); \draw
(g)--(w); \draw (0,-0.1)--(0,0.1);
\end{tikzpicture}\quad \raisebox{3em}{$=$} \quad\begin{tikzpicture}
 \node (f) [circle,fill=none,draw=black,minimum size=2mm,inner sep=0.5mm]  at (-1,0) {\small $U$};
\node (g) [circle,fill=none,draw=black,minimum size=2mm,inner
sep=0.5mm]  at (1,0) {\small $V$}; \node (a)
[label={[xshift=-0.05cm, yshift=-0.29cm,]{\footnotesize
$\gamma'$}},circle,fill=none,draw=none,minimum size=2mm,inner
sep=0mm]  at (-1.6,0.85) {}; \node (w) [label={[xshift=0.05cm,
yshift=-0.29cm,]{\footnotesize
$\delta$}},circle,fill=none,draw=none,minimum size=2mm,inner
sep=0mm]  at (1.6,0.85) {}; \node (b)
[label={[xshift=-0.05cm,yshift=-0.33cm,]{\footnotesize
$\alpha$}},circle,fill=none,draw=none,minimum size=2mm,inner
sep=0mm]  at (-2.2,0) {}; \node (c) [label={[xshift=-0.07cm,
yshift=-0.42cm,]{\footnotesize
$\beta$}},circle,fill=none,draw=none,minimum size=2mm,inner
sep=0mm]  at (-1.6,-0.85) {}; \node (d) [label={[xshift=0.07cm,
yshift=-0.35cm,]{\footnotesize
$\beta'$}},circle,fill=none,draw=none,minimum size=2mm,inner
sep=0mm]  at (2.2,0) {}; \node (e) [label={[xshift=0.07cm,
yshift=-0.33cm,]{\footnotesize
$\alpha'$}},circle,fill=none,draw=none,minimum size=2mm,inner
sep=0mm]  at (1.6,-0.85) {}; \node (i) [label={[xshift=-0.2cm,
yshift=-0.17cm,]{\footnotesize $\varphi$}},label={[xshift=0.2cm,
yshift=-0.17cm,]{\footnotesize
$\psi$}},circle,fill=none,draw=none,minimum size=0mm,inner
sep=0mm]  at (0,0.15) {}; \draw (f) -- (g); \draw (f)--(a); \draw
(f)--(b); \draw (f)--(c); \draw (g)--(d); \draw (g)--(e); \draw
(g)--(w); \draw (0,-0.1)--(0,0.1);
\end{tikzpicture}
\end{center}
of the resulting unrooted trees.
\end{example}
Example \ref{counterexample} has a proof-theoretic consequence, as
we explain now.  Let us first display the third representative of
$K \, _\gamma\!\circ_{\delta'} \,L$ and $U \,
_\varphi\!\circ_{\psi}\,L$ specified above.
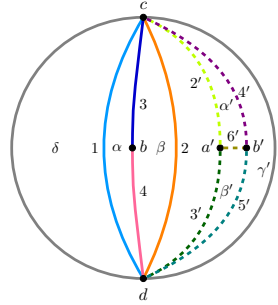
\begin{center}
\resizebox{3.5cm}{!}{\begin{tikzpicture} \draw[line
width=0.18em,draw=gray]  (0,0) circle (3cm); \draw[line
width=0.18em,draw=qqzzff]  (0,3).. controls (-1.2,1) and (-1.2,-1)
..(0,-3) node[midway,xshift=-0.2cm] {$1$}; \draw[line
width=0.18em,draw=ffxfqq]  (0,3).. controls (1,1) and (1,-1)
..(0,-3) node[midway,xshift=0.2cm] {$2$}; \draw[line
width=0.18em,draw=qqqqcc] (0,3).. controls (-0.24,1) ..(-0.25,0);
\draw[line width=0.18em,draw=ffwwzz] (-0.25,0).. controls
(-0.24,-1) ..(0,-3); \node (5) at (1.2,1.5) {$2'$}; \node (6) at
(1.2,-1.5) {$3'$}; \node (alpha) at (1.9,1) {$\alpha'$}; \node
(beta) at (1.9,-1) {$\beta'$}; \node (alpha) at (-0.6,0)
{$\alpha$}; \node (beta) at (0.4,0) {$\beta$}; \node (gamma) at
(2.75,-0.5) {$\gamma'$}; \node (delta) at (-2,0) {$\delta$};
\node(S) at (2.35,0) [circle,draw=black,fill=black,minimum
size=4,inner sep=0,label={[xshift=0.275cm,yshift=-0.3cm]{\large
$b'$}}]  {}; \node(A) at (0,3)
[circle,draw=black,fill=black,minimum size=4,inner
sep=0,label={[xshift=0cm,yshift=0.0cm]{\large $c$}}]  {}; \node(B)
at (0,-3) [circle,draw=black,fill=black,minimum size=4,inner
sep=0,label={[xshift=0cm,yshift=-0.7cm]{\large $d$}}]  {};
\node(C) at (-0.25,0) [circle,draw=black,fill=black,minimum
size=4,inner sep=0,label={[xshift=0.25cm,yshift=-0.3cm]{\large
$b$}}]  {}; \node(D) at (1.75,0)
[circle,draw=black,fill=black,minimum size=4,inner
sep=0,label={[xshift=-0.275cm,yshift=-0.3cm]{\large $a'$}}]  {};
\draw[line width=0.18em,dashed,draw=qqwuqq] (D) to[out=270,in=25]
(B) node[midway,yshift=-1cm] {$4$}; \draw[line
width=0.18em,dashed,,draw=lime] (D) to[out=90,in=-25] (A)
node[midway,yshift=1cm] {$3$}; \draw[line
width=0.18em,dashed,draw=olive] (D)-- (S) node[midway,above]
{$6'$}; \draw[line width=0.18em,dashed,draw=teal] (S)
to[out=270,in=25] (B) node[midway,xshift=2.3cm,yshift=-1.3cm]
{$5'$}; \draw[line width=0.18em,dashed,draw=violet] (S)
to[out=90,in=-25] (A) node[midway,xshift=2.3cm,yshift=1.3cm]
{$4'$};
\end{tikzpicture}}
\end{center}
Based on this picture (and on the representatives that it induces
for $K,L,U,V$), and by Proposition \ref{connected-cut}, we get the
following derivations in the Menelaus system, where, in order to
save horizontal space and to  help the reader to place the
sextuples spatially, we write $(b,c,d,1,4,3)$ as
$\underbracket{bcd143}_{\alpha}$, etc.:

{\small
$$\f{\f{}{\underbracket{bcd143}_{\alpha}\quad\underbracket{bcd243}_{\beta}\quad\underbrace{a'cd23'2'}_{\gamma}\quad\underbracket{a'cd13'2'}_{\delta}}\quad\quad\f{}{\underbracket{a'b'c4'2'6'}_{\alpha'}\quad\underbracket{a'b'd5'3'6'}_{\beta'}\quad\underbracket{b'cd25'4'}_{\gamma'}\quad\underbrace{a'cd23'2'}_{\delta'}}}{\underbracket{bcd143}_{\alpha}\quad\underbracket{bcd243}_{\beta}\quad\underbracket{a'cd13'2'}_{\delta}\quad\underbracket{a'b'c4'2'6'}_{\alpha'}\quad\underbracket{a'b'd5'3'6'}_{\beta'}\quad\underbracket{b'cd25'4'}_{\gamma'}}$$
$$\f{\f{}{\underbracket{bcd143}_{\alpha}\quad\underbracket{bcd243}_{\beta}\quad\underbracket{b'cd25'4'}_{\gamma'}\quad\underbrace{b'cd15'4'}_{\varphi}}\quad\quad\f{}{\underbracket{a'b'c4'2'6'}_{\alpha'}\quad\underbracket{a'b'd5'3'6'}_{\beta'}\quad\underbrace{b'cd15'4'}_{\psi}\quad\underbracket{a'cd13'2'}_{\delta}}}{\underbracket{bcd143}_{\alpha}\quad\underbracket{bcd243}_{\beta}\quad\underbracket{b'cd25'4'}_{\gamma'}\quad\underbracket{a'b'c4'2'6'}_{\alpha'}\quad\underbracket{a'b'd5'3'6'}_{\beta'}\quad\underbracket{a'cd13'2'}_{\delta}}$$}
These two derivations can be contrasted with a third one, obtained
by directly applying one axiom:
$$\f{}{\underbracket{bcd143}_{\alpha}\quad\underbracket{bcd243}_{\beta}\quad\underbracket{b'cd25'4'}_{\gamma'}\quad\underbracket{a'b'c4'2'6'}_{\alpha'}\quad\underbracket{a'b'd5'3'6'}_{\beta'}\quad\underbracket{a'cd13'2'}_{\delta}}$$
Recall the notion of normal derivation from Section
\ref{secdecidability}. The process lying behind the proofs of the
two lemmata leading to Corollary \ref{atomisation} is that of a
normalisation, i.e. a gradual transformation from a given
derivation to one in normal form.  Let us now call {\em long
normal derivation} a normal derivation in which moreover the only
axioms coming from ${\mathcal M}$-complexes are those coming from
the indecomposable ${\mathcal M}$-complexes.  Thanks to
Proposition \ref{connected-cut} and Proposition \ref{generation},
we can transform any normal derivation into a long normal
derivation. Note that this process introduces cuts, while the
process leading to a normal derivation can be viewed as a
cut-elimination process. With these glasses, the first two
derivations are distinct  long normal forms for the third one.
Thus, a fortiori, we can draw the following proof-theoretic
conclusion from Example \ref{counterexample}: the process of
transforming a given derivation (normal or not) into a long normal
derivation is non-deterministic.

\smallskip
In the rest of the section, we shall circonscribe the situations
of the kind shown in Example \ref{counterexample}, by examining
which of the ${\mathcal M}$-complexes  induce and which do not
induce   relations, so as to give a presentation of the Menelaus
cyclic operad. We first characterize indecomposable  ${\mathcal
M}$-complexes.

\subsection{Irreducibe ${\mathcal M}$-complexes}

Let $K$ be an $\mathcal{M}$-complex and let $T=\{e_0,e_1,e_2\}$
$\subseteq K_1$ be such that $\partial_1(e_0-e_1+e_2)=0$, i.e.\
$e_0-e_1+e_2$ is a 1-cycle (i.e., $e_0,e_1,e_2$ form a triangle).
Consider the binary relation on $K_2$ of sharing an edge from
$K_1-T$. Let $\tau$ be the transitive closure of this relation. We
say that $T$ is a \emph{cut-triangle}, when $\tau$ is an
equivalence relation with exactly two classes. If $K$ contains a
cut-triangle, then we say that it is \emph{reducible}, otherwise
it is called \emph{irreducible}. Note that if $T$ is the set of
edges of a 2-cell of $K$, then $\tau$ is not reflexive.

We invite the reader to check that the triangles drawn in black in
Example \ref{counterexample} are cut triangles. In fact, as we
shall see, cut triangles detect decomposability.

We set up a bit of terminology. Let $K$ be a reducible ${\mathcal
M}$-complex and let $T=\{e_0,e_1,e_2\}$  be a cut-triangle of $K$.
We shall write $K^l_{T}$ and $K^r_T$ for the equivalence classes
of $2$-cells with respect to the relation of sharing an edge from
$K_1-T$. We shall denote this relation with $\sim_{T}$.
 A {\em path} between two 2-cells $u,v$ is a sequence of 2-cells starting from $u$ and ending with $v$ such that any two consecutive 2-cells  share a 1-cell, and we say that the path {\em crosses} these 1-cells.

A key tool in this section is the following lemma, which  makes
explicit the argument used in the proof of Proposition 2.9.

\begin{lem}[Camembert lemma] \label{camembert}
The axiom (4) in the definition of ${\mathcal M}$-complex can be
reinforced as folows. For every $0$-cell $w$, the faces of $L_w$
can be displayed without repetition around $w$ in a circle.
Formally, we can arrange a cyclic order  on
$L_w=\{u_1,\ldots,u_n\}$  in such a way that, for all $i$, $u_i$
and $u_{i+1}$ are $w$-neighbours, modulo $n$. We refer to $L_{w}$
equipped with this cyclic order as the camembert of  $w$.
\end{lem}
 \begin{proof} We first exhibit a circle like in the statement, but we do not prove  yet that all  elements of $L_w$ occur in it.  By axiom (1), we can pick $u_1\in L_w$. We define a $u_1$-sequence to be   a path  $v_1,\ldots,v_{m}$ of elements of $L_w$ such that $v_1=u_1$,   $v_i$ and $v_{i+1}$ are $w$-neighbours for all $1\leq i<m$, and the path $v_{i-1},v_i,v_{i+1}$ crosses two different edges of $v_i$ for all $2\leq i\leq m-1$. By axiom (3), we can construct an infinite $u_1$-sequence  $v_1,\ldots,v_{m},\ldots$, in the following way: choose $v_2$ to be a $w$-neighbour  of $u_1$  in $L_w$ ($v_2$ might be the unique such $2$-cell, in which case $L_w=\{v_1,v_2\}$, or there might be a choice of exactly two such $2$-cells), and then at each $v_i$ continue with the $w$-neighbour of $v_i$  along the edge of $v_i$ incident to $w$  that is not crossed  from $v_{i-1}$ to $v_i$. If $L_w$ contains more than two $2$-cells, then, by axiom (0), there exist $i$ and $j$ such that $i<j$, $v_i=v_j$, and all $2$-cells $v_i,v_{i+1},\ldots,v_{j-1}$ are distinct. We next show the following property (P): every two  $v_i$-sequences that share their first two elements are prefixes of one another. Indeed, if the two sequences diverge at some $2$-cell $v$, this would display a $1$-cell shared by at least three $2$-cells of $K$, contradicting axiom (3). We now show that all elements of $L_w$ appear among $v_i,v_{i+1},\ldots,v_{j-1}$. Let $u\in L_w$. By axiom (4), there exists a $v_i$-sequence leading to $u$. By property (P), this sequence is a prefix of a sufficiently long prefix of $v_i,v_{i+1},\ldots,v_{j-1},v_j,v_{i+1},\ldots,v_{j-1},v_j,\ldots$ or of $v_i,v_{j-1},\ldots,v_{i+1},v_j,v_{j-1},\ldots,v_{ji+1},v_i,\ldots$.  Thus $u=v_k$, for some $i\leq k\leq j-1$.
  \end{proof}

\begin{rem}\label{rem1}   The Camembert lemma provides us with the following reasoning concerning the paths between $2$-cells: if the path crosses at least two edges of a single cut-triangle $T$, using the Camembert lemma, we shall be able to transform that path into one that  crosses strictly less edges of $T$. We shall tacitly use an induction on this number of crossings.
\end{rem}

 The following lemma provides a method for showing that three $1$-cells of an ${\mathcal M}$-complex do {\em not} make a cut-triangle.

 \begin{lem}\label{fff}
If  $T=\{e_0,e_1,e_2\}$ is a cut-triangle of $K$, then  for each
$e_i$, $0\leq i\leq 2$, each of the classes $K^l_{T}$ and $K^r_T$
contains  exactly one $2$-cell having $e_i$ as a $1$-cell.
\end{lem}
\begin{proof} Fix an $i\in\{0,1,2\}$.
By the axiom (3) for $K$,   there  exist exactly two $2$-cells of
$K$ having $e_i$ as an edge; denote them with $u^l$ and $u^r$. We
show that, if   $u^l$ and $u^r$ are both contained in  $K^l_{T}$,
then all the 2-cells of $K$ must also belong to $K^l_{T}$, thereby
contradicting the fact that $T$ is a cut-triangle of $K$.

 Let $u$ and $v$ be two arbitrary distinct $2$-cells of $K$. As a consequence of the connectedness property  and axiom (4) for $K$, there exists a path of $2$-cells starting from $u$ and ending in $v$. If the 1-cells crossed on this path all belong to $K_1\setminus T$, then $u$ and $v$ are equivalent by definition. If the path  crosses $e_i$, then, by the assumption that $u^l$ and $u^r$ are equivalent, this crossing can be replaced by a path between  $u^l$ and $u^r$ that does not cross $e_i$. If the path crosses $e_j$ for $e_j\in T\setminus\{e_i\}$, then  let $u_1$ and $u_2$ be the faces sharing $e_j$ on the path. By the Camembert lemma, we can display
all of $u^l,u^r,u_1,u_2$ in a circle of $2$-cells associated with
the link of the 0-cell $w$ common to $e_i$ and $e_j$, in such a
way that $u_1,u_2$ (resp. $u^l,u^r$) are neighbours. Then
following the cells in clockwise or anticlockwise way, we see that
there is a path from $u_1$ to  $u^l$ (or $u^r$) and a path from
$u_2$ to $u^r$ (or $u^l$) that do not cross $e_i$ nor $e_j$ (nor
the third 1-edge of $T$, which is not incident to $w$). Therefore,
again, we can replace the length 1 path from $u_1$ to $u_2$ by the
concatenation of three paths from, say, $u_1$ to $u^l$ to $u^r$ to
$u_2$, witnessing that $u_1$ and $u_2$ are equivalent. This
concludes the proof by contradiction.
\end{proof}

  \begin{example}
The unique cut-triangle in the torus with two holes of Example 3
is the triangle with edges $D$, $1$ and $I$. Indeed, the two
equivalence classes of $2$-cells with respect to the relation of
sharing an edge other than $D$, $1$ and $I$ are given by
\begin{enumerate}
\item $B1A$, $B2C$, $D3C$, $J2I$, $J3A$, and \item $D4E$, $F1G$,
$F5E$, $H4G$, $H5I$.
\end{enumerate}
\end{example}

\begin{example}
Consider the triangulation of the torus with two holes into ten
$2$-cells and three zero-cells in total, obtained by identifying
the opposite sides of a decagon:
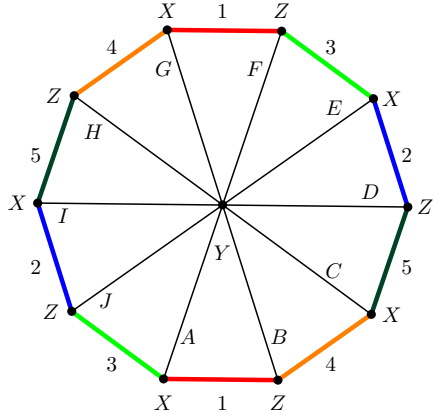
\begin{center}
\psscalebox{.4} 
{
\begin{pspicture}(3,-10)(17,5)
\psline[linewidth=0.05,dimen=outer](8,-8.2)(11.88,3.33)
\psline[linewidth=0.05,dimen=outer](11.76,-8.24)(8.12,3.37)
\psline[linewidth=0.05,dimen=outer](14.83,-6.06)(5.06,1.19)
\psline[linewidth=0.05,dimen=outer](16.03,-2.5)(3.86,-2.37)
\psline[linewidth=0.05,dimen=outer](4.98,-5.96)(14.9,1.09)

\psline[linewidth=0.15,dimen=outer,linecolor=red](8,-8.2)(11.76,-8.24)
\psline[linewidth=0.15,dimen=outer,linecolor=red](8.12,3.37)(11.88,3.33)

\psline[linewidth=0.15,dimen=outer,linecolor=orange](11.76,-8.24)(14.83,-6.06)
\psline[linewidth=0.15,dimen=outer,linecolor=orange](5.06,1.19)(8.12,3.37)

\psline[linewidth=0.15,dimen=outer,linecolor=britishracinggreen](14.83,-6.06)(16.03,-2.5)
\psline[linewidth=0.15,dimen=outer,linecolor=britishracinggreen](3.86,-2.37)(5.06,1.19)

\psline[linewidth=0.15,dimen=outer,linecolor=blue](4.98,-5.96)(3.86,-2.37)
\psline[linewidth=0.15,dimen=outer,linecolor=blue](16.03,-2.5)(14.9,1.09)

\psline[linewidth=0.15,dimen=outer,linecolor=green](14.9,1.09)(11.88,3.33)
\psline[linewidth=0.15,dimen=outer,linecolor=green](8,-8.2)(4.98,-5.96)

\rput(8,-9.0){\psscalebox{2}{$X$}}
\rput(11.76,-9.0){\psscalebox{2}{$Z$}}
\rput(15.5,-6.06){\psscalebox{2}{$X$}}
\rput(16.6,-2.5){\psscalebox{2}{$Z$}}
\rput(15.5,1.09){\psscalebox{2}{$X$}}
\rput(11.88,4.0){\psscalebox{2}{$Z$}}
\rput(8.12,4.0){\psscalebox{2}{$X$}}
\rput(4.4,1.19){\psscalebox{2}{$Z$}}
\rput(3.2,-2.37){\psscalebox{2}{$X$}}
\rput(4.3,-5.96){\psscalebox{2}{$Z$}}
\rput(9.94,-4.0){\psscalebox{2}{$Y$}}

\rput(9.94,-9.0){\psscalebox{2}{$1$}}
\rput(9.94,4.0){\psscalebox{2}{$1$}}
\rput(13.5,-7.7){\psscalebox{2}{$4$}}
\rput(3.8,-4.5){\psscalebox{2}{$2$}}
\rput(16.0,-4.5){\psscalebox{2}{$5$}}
\rput(6.3,-7.7){\psscalebox{2}{$3$}}
\rput(16.0,-0.8){\psscalebox{2}{$2$}}
\rput(6.3,2.8){\psscalebox{2}{$4$}}
\rput(13.5,2.8){\psscalebox{2}{$3$}}
\rput(3.8,-0.8){\psscalebox{2}{$5$}}

\rput(8.8,-6.8){\psscalebox{2}{$A$}}
\rput(11.8,-6.8){\psscalebox{2}{$B$}}
\rput(13.6,-4.6){\psscalebox{2}{$C$}}
\rput(14.8,-2.0){\psscalebox{2}{$D$}}
\rput(13.6,0.8){\psscalebox{2}{$E$}}
\rput(11.0,2.1){\psscalebox{2}{$F$}}
\rput(8.0,2.1){\psscalebox{2}{$G$}}
\rput(5.7,0.0){\psscalebox{2}{$H$}}
\rput(4.7,-2.8){\psscalebox{2}{$I$}}
\rput(6.1,-5.6){\psscalebox{2}{$J$}}

\rput(8,-8.2){\pscircle*{.15}} 
\rput(11.76,-8.24){\pscircle*{.15}} 
\rput(14.83,-6.06){\pscircle*{.15}} 
\rput(16.03,-2.5){\pscircle*{.15}} 
\rput(14.9,1.09){\pscircle*{.15}} 
\rput(11.88,3.33){\pscircle*{.15}} 
\rput(8.12,3.37){\pscircle*{.15}} 
\rput(5.06,1.19){\pscircle*{.15}} 
\rput(3.86,-2.37){\pscircle*{.15}} 
\rput(4.98,-5.96){\pscircle*{.15}} 
\rput(9.94,-2.43){\pscircle*{.15}} 

\end{pspicture}
}
\end{center}
In order to see that this triangulation is irreducible, notice
first that the symmetric nature of the triangulation allows us to
look for a cut-triangle by requiring that it contains a fixed side
of the decagon, say $1$, without  loss of generality. Let us first
examine the potential cut-triangles that contain edges $1$ and
$A$:
$$1AD,\enspace 1AF,\enspace 1AH,\enspace \mbox{ and } \enspace 1AJ.$$
(Notice that, since $1AB$ is a $2$-cell of the triangulation, it
is not a candidate for a cut-triangle.) By Lemma \ref{fff}, in
order for   $1AD$ to be a cut-triangle, the $2$-cells $1AB$ and
$1GF$ must belong to different equivalence classes induced by
$1AD$. However, the  sequence of $2$-cells  $1AB, 4BC, 4GH, 1GF$,
in which each two successive members share an edge outside of
$1AD$, witnesses that this is not possible. By the same argument,
the sequence $1AB, 4BC, 5CD, 2DE, 3EF, 3AJ$ shows that $1AF$
cannot be a cut-triangle, the sequence $4HG, 1GF, 3FE, 2ED, 5DC,
5IH$ shows that $1AH$ cannot be a cut-triangle, and the sequence
$3AJ,3EF,1FG,4GH,5IH,2IJ$ shows that $1AJ$ cannot be a
cut-triangle. By  symmetry again, we can conclude that a
cut-triangle cannot contain edges $1$ and $B$. The remaining
candidates for a cut-triangle  are   $1ID$,   $1EJ$ and $1EH$. By
the symmetry of the triangulation again, we can reduce the
analysis to  $1ID$ and $1EJ$. By Lemma \ref{fff}, in order for
$1ID$ to be a cut-triangle, the $2$-cells $1AB$ and $1GF$ must
belong to different equivalence classes induced by  $1ID$.
However, the  sequence of $2$-cells  $1AB, 3AJ, 3FE, 1GF$, in
which each two successive members share an edge outside of  $1ID$,
witnesses that this is not possible. Similarly, the sequence of
$2$-cells $1AB, 4BC, 5CD, 5HI, 4HG, 1GF$ witnesses that $1EJ$
cannot be a cut-triangle.
\end{example}

 We need one more property of cut-triangles in order to show, in Proposition \ref{theprop}, that  the irreducible ${\mathcal M}$-complexes are the indecomposable ${\mathcal M}$-complexes.

\begin{lem} \label{disjointness-lemma}
If $T$ is a cut-triangle for $K$, then the two equivalence classes
$K^l_T$ and $K^r_T$ of 2-cells do not share  lower dimensional
faces other than those of $T$.
\end{lem}
\begin{proof}
We just have to check that if $u,v$ are two 2-cells of $K$ sharing
a 1-cell  in $K_1\setminus T$ or a cell in $K_0$ which is not a
face of any element of $T$, then they are equivalent. This is
obvious for 1-cells, by the very definition of the equivalence.
Suppose now that $u,v$ share a  0-cell $w$ outside of $T$. These
cells belong to $L_w$, hence by  condition (4) applied to $K$ and
$w$, there is a path from $u$ to $v$ in $L_w$. Such a path does
not cross any of the $e_i$'s, since none of these 1-cells  is
incident to $w$. Hence $u$ and $v$ are equivalent.
\end{proof}

\begin{prop}\label{theprop}
An ${\mathcal M}$-complex $K$ is reducible if and only if it is
decomposable.
\end{prop}
\begin{proof}
Suppose that $K$ is reducible and  let $T=\{e_0,e_1,e_2\}\subseteq
K_1$ be a cut-triangle of $K$. Define, for $o\in \{l,r\}$ and
$0\leq j\leq 2$, the sets $(\hat{K}^o_T)_j$ as follows:
\begin{center}
$(\hat{K}^o_T)_2=K^o_{T}\cup \{t\}$,\quad
$(\hat{K}^o_T)_1=\displaystyle\bigcup_{x\in (\hat{K}^o_T)_2-\{t\}}
\{d^2_0x,d^2_1x,d^2_2x\} $ \quad and\quad
$(\hat{K}^o_T)_0=\displaystyle\bigcup_{y\in (K^o_T)_1}
\{d^1_0y,d^1_1y\}$,
\end{center}
for $t\not\in K_T^o$. Therefore, for $j<2$, $(\hat{K}^o_T)_j$ is
the set of all $j$-cells contained in the equivalence class
$K_T^o$, while, for $j=2$, $(\hat{K}^o_T)_2$ additionally contains
$t$ as  a $2$-cell. Define, moreover, for $1\leq i\leq 2$, $1\leq
j\leq 2$ and $0\leq k\leq j$,  functions $(d_T^o)_k^j:
(\hat{K}_T^o)_j\rightarrow (\hat{K}_T^o)_{j-1}$ as the appropriate
restrictions of the corresponding faces $d_k^j:K_j\rightarrow
K_{j-1}$ of
$K$, except on the triangle $t$, for which the  action of
$(d_T^o)_k^2$ is defined by  $(d_T^o)_k^2(t)=e_k$. We prove that,
by taking the elements of $(\hat{K}^o_T)_j$ as $j$-cells and
functions $(d_T^o)_k^j$ as the appropriate faces, the equivalence
classes $K_T^l$ and $K^r_T$ get  turned into   ${\mathcal
M}$-complexes $\hat{K}_T^l$ and $\hat{K}^r_T$, respectively. By
construction, $\hat{K}_T^l$ and $\hat{K}^r_T$ satisfy  the first
two axioms of an ${\mathcal M}$-complex. We verify below the
remaining four for $\hat{K}_T^l$.

\begin{itemize}
\item[(2)] The axiom (2) follows by the fact that, in the
cut-triangle $T$, all $0$-cells and $1$-cells are mutually
distinct. Indeed,  by the regularity of $K$, the two $0$-cells of
each $e_i$, $i\in\{0,1,2\}$, are mutually different, and,
therefore, the three $0$-cells of the $2$-cell $t$ must be
mutually distinct. Consequently,  the $1$-cells of $t$ must also
be mutually distinct. \item[(3)] Note that, for all the $1$-cells
$e$ of  $(\hat{K}_T^l)_1-\{e_0,e_1,e_2\}$, this axiom holds by the
same axiom for $K$. (Observe that   Lemma \ref{disjointness-lemma}
disallows us to apply the same argument twice on $e$, which would
produce four $2$-cells adjacent with $e$.) As for the 1-cells
$\{e_0,e_1,e_2\}$, each  of them belongs to $t$ and, by Lemma
\ref{fff}, to exactly one  $2$-cell of $\hat{K}_T^l$, which proves
the claim. \item[(4)] Note that, for all the $0$-cells of
$(\hat{K}_T^l)_0-\{d^1_{0}e_0,d^1_{1}e_0,d_1^1e_1\}$, this axiom
holds by the same axiom for $K$ (and again by Lemma
\ref{disjointness-lemma}). Suppose therefore that $w\in
\{d^1_{0}e_0,d^1_{1}e_0,d_1^1e_1\}$, say $w=d^1_{0}e_0$, and let
$u,v\in L_{w}$. Suppose, moreover, that both $u$ and $v$ are
different from $t$. By the Camembert lemma for $K$,  there exists
a path $u,u_1,\dots,u_k,v$ of $2$-cells of $K$ starting at $u$ and
ending at  $v$, such that every two consecutive $2$-cells are
$w$-neighbours. Axiom (4) follows immediately if this path is
entirely contained in $\hat{K}_T^l$. If the path contains
$2$-cells of $\hat{K}_T^r$, and, hence, crosses the two edges of
$T$ adjacent with $w$, then, by the Camembert lemma, we can
transform it into a path that does not cross an edge of $T$, by
the same reasoning as in the proof of Lemma \ref{fff}. Suppose now
that  $u=t$. By Lemma \ref{fff}, for the $1$-cell $e_0$ of $t$, we
can pick the $2$-cell $u'$ of $\hat{K}_T^l$ that has $e_0$ as a
$1$-cell; note that $u' \in L_w$. By the Camembert lemma for $K$,
there exists a sequence $u,u_1,\dots,u_k,u'$ of $2$-cells starting
at $u$ and ending at  $u'$, such that every two consecutive
$2$-cells are $w$-neighbours, and which is, moreover, entirely
contained in $\hat{K}_T^l$. The sequence of $2$-cells starting at
$u$ and ending at  $t$, such that every two consecutive $2$-cells
are $w$-neighbours, is then the sequence $u,u_1,\dots,u_k,u',t$.
\item[(5)] Let $c=\sum^n_{i=1}\varepsilon_i u_i$ be the
fundamental cycle of $K$, and let $c_l=\sum_{i\in
I}\varepsilon_iu_i$ and $c_r=\sum_{j\in J}\varepsilon_ju_j$ be
such that $c=c_l+c_r$, $K^l_T=\{u_i\,|\, i\in I\}$ and
$K^r_T=\{u_j\,|\, j\in J\}$. Consider the boundary $\partial c_l$.
By the axiom (3) for $\hat{K}_T^l$, each $1$-cell $e$  that
appears in $\partial c_l-\{e_0,e_1,e_2\}$  will have exactly one
more occurence in this boundary. By the orientability of $K$, we
can conclude that $\partial c_l=\sum^2_{m=0}\tau_me_m$, for
$\tau_{m}\in\{+1,-1\}$.
 Moreover, since $\partial^2=0$, we know that  $\sum^2_{m=0}\tau_me_m$ is a $1$-cycle. But since $e_0-e_1+e_2$ is also a $1$-cycle, it must be the case that
 either $\partial c_l=e_0-e_1+e_2$ or $\partial c_l=-(e_0-e_1+e_2)$. In the first case, the orientation of $\hat{K}^l_T$ is obtained by taking the coefficient $\varepsilon_l$ of $T$ to be $-1$ and, in the second case, by taking  it to be $+1$.
\end{itemize}
By Lemma \ref{disjointness-lemma}, $\hat{K}_T^l$ and $\hat{K}^r_T$
have no common faces other than the ones of $t$. Also, both
$\hat{K}_T^l$ and $\hat{K}^r_T$ have strictly less $2$-cells than
$K$. Indeed, if, say, $\hat{K}_T^l$ would contain the same number
of $2$-cells as $K$, then $\hat{K}_T^r$ would have to be the
${\mathcal M}$-complex given by two $2$-cells  having the same
boundaries, meaning that $\sim_T$ would not be an equivalence
relation.
 It follows by the definition of the connected sum construction that  $K$ is indeed a connected sum of  $\hat{K}_T^l$ and $\hat{K}^r_T$ along the $2$-cell $t$.

\smallskip

The proof of  the other direction goes by showing that, for $M = K
{_x\circ_y} L$,   the triangle $T=\{e_0,e_1,e_2\}$ resulting from
the identification of the 1-cells of $x$ and $y$, is a
cut-triangle. We leave the easy details to the reader.
 \end{proof}
 From now on, we shall confuse ``(ir)reducible'' with ``(in)decomposable''.

\subsection{A presentation of the Menelaus cyclic operad}\label{last}
Informed by the notion of cut-triangle, we can attribute the
relation induced by Example \ref{counterexample} to a ``bad''
configuration of two cut-triangles in one and the same ${\mathcal
M}$-complex. Let us now examine ``nice'' configurations, i.e.
those of two  cut-triangles that  will {\em not} induce relations
in the sought presentation.

\begin{definition}
Let $K$ be an ${\mathcal M}$-complex and let
$T_1=\{e^1_0,e^1_1,e^1_2\}\subseteq K_1$ and
$T_2=\{e^2_0,e^2_1,e^2_2\}\subseteq K_1$ be two cut-triangles of
$K$. We say that ${T}_1$ is {\em disjoint} from ${T}_2$   if all
the edges of $T_1$  are $1$-cells of one of the two ${\mathcal
M}$-complexes induced by ${T}_2$.
\end{definition}

\begin{lem}\label{subclass}
If $K$ is an ${\mathcal M}$-complex and if
$T_1=\{e^1_0,e^1_1,e^1_2\}\subseteq K_1$ and
$T_2=\{e^2_0,e^2_1,e^2_2\}\subseteq K_1$ are two cut-triangles of
$K$, then ${T}_1$ is  disjoint from ${T}_2$ if and only if one of
the equivalence classes of $2$-cells induced by $T_2$  is entirely
contained in one of the equivalence classes of $2$-cells induced
by $T_1$.
\end{lem}
\begin{proof} Suppose that ${T}_1$ is  disjoint from ${T}_2$, i.e., that, without  loss of generality, all the edges of $T_1$ belong to $\hat{K}^l_{T_2}$. We show that any two $2$-cells $u$ and $v$ of $\hat{K}^r_{T_2}$ belong to the same class with respect to $T_1$. Pick a path $p$ from $u$ to $v$ in $\hat{K}^r_{T_2}$ and suppose that it crosses $T_1$. Observe that this is only possible if $T_1$ and $T_2$ share one edge, say $e_1$. Moreover, the path contains the $2$-cells  $w_1$ (from  $\hat{K}_{T_2}^r$) and $w_2$ (from $\hat{K}_{T_2}^l$) adjacent with $e_1$. Since $e_1$ is an edge of $T_2$, and $u$ and $v$ are both in $\hat{K}_{T_2}^r$, the path must cross again $T_2$ before reaching $v$, say, at the edge $e_2$. Let $w'_1$  and $w'_2$ be the $2$-cells   from  $\hat{K}_{T_2}^l$ and $\hat{K}_{T_2}^r$, respectively, adjacent with $e_2$. Then, thanks to the Camembert lemma, the part of the path $p$ from $w_1$ to $w'_1$ can be replaced by another path that lies completely in $\hat{K}^r_{T_2}$. We conclude  by induction.

For the opposite direction,  suppose, without  loss of generality,
that  $\hat{K}^l_{T_2}\subsetneq \hat{K}^l_{T_1}$.  Suppose, again
without loss of generality, that $e^1_0$ is a $1$-cell of
$\hat{K}^l_{T_2}$ and that $e^1_1$ is a $1$-cell of
$\hat{K}^r_{T_1}$. Observe that this is only possible if $e^1_0$
shares its two  vertices, say $A$ and $B$, with one of the edges
of $T_2$, say $e_0^2$. Suppose that $e_1^2$ is the other edge of
$T_2$ adjacent with $A$. Let $u$ and $v$ be the 2-cells of
$\hat{K}^l_{T_2}$ adjacent with $e_0^2$ and $e_1^2$, respectively.
Since $\hat{K}^l_{T_2}\subsetneq \hat{K}^l_{T_1}$, the path on the
camembert of $A$ between $u$ and $v$ that lies in
$\hat{K}^l_{T_2}$ does not cross $T_1$. This means that $e^1_0$ is
on the opposite path between $u$ and $v$  on the camembert of $A$,
and, hence in $\hat{K}^r_{T_2}$. Contradiction.\end{proof}

  Lemma \ref{subclass} in particular  shows that being disjoint is a symmetric relation on cut-triangles of an ${\mathcal M}$-complex.

 \begin{lem}\label{nesting}
For two disjoint cut-triangles
$T_1=\{e^1_0,e^1_1,e^1_2\}\subseteq K_1$ and
$T_2=\{e^2_0,e^2_1,e^2_2\}\subseteq K_1$  of an ${\mathcal
M}$-complex $K$, the following two claims hold.
\begin{enumerate}
\item If $K^l_{T_1}\subsetneq K^r_{T_2}$,  then $T_1$ is a
cut-triangle of the ${\mathcal M}$-complex $\hat{K}^r_{T_2}$.
\item There exist ${\mathcal M}$-complexes $K_1$, $K_2$ and $K_3$
such that $$K=(K_1\, {_{t_1}\!\circ_{t_1}}\,
K_2)\,{_{t_2}\!\circ_{t_2}}\, K_3=K_1\, {_{t_1}\!\circ_{t_1}}\,(
K_2\,{_{t_2}\!\circ_{t_2}}\, K_3),$$ where $t_1$ and $t_2$ are the
$2$-cells induced by $T_1$ and $T_2$ as in the proof of
Proposition \ref{theprop}.
\end{enumerate}
\begin{proof}
(1)  We prove that the two equivalence classes of triangles with
respect to the relation of sharing an edge from
$(\hat{K}^t_{T_2})_1-T_1$ are given by $K^l_{T_1}$ and
$K^r_{T_1}\cap K^r_{T_2}$. Since $T_1$ is a cut-triangle of $K$,
we know that these two sets of $2$-cells make {\em at least} $2$
equivalence classes, and it remains to show that all the $2$-cells
from $K^r_{T_1}\cap K^r_{T_2}$ are equivalent, i.e. that for
$u,v\in K^r_{T_1}\cap K^r_{T_2}$ we can find a path from $u$ to
$v$ that does not leave $K^r_{T_1}\cap K^r_{T_2}$. Since
$K^r_{T_1}\cap K^r_{T_2}\subseteq K^r_{T_1}$ and $K^r_{T_1}$ is a
well-defined equivalence class with respect to $T_1$, there exists
a path    $u,u_1,\dots,u_k,v$  from $u$ to $v$ in  $K^r_{T_1}$,
which, however, might leave $K^r_{T_2}$. But if the $2$-cells
$u_i$ and $u_{i+1}$ are such that $u_i\in K^r_{T_2}$ and
$u_{i+1}\in K^r_{T_1}-K^r_{T_2}$, by the Camembert lemma, we can
replace $u_{i+1}$ by the $2$-cell $u'_{i+1}\in K^r_{T_2}$ adjacent
to $u_i$, and conclude by the induction hypothesis.

\smallskip

(2) This is a consequence of the claim (1). Indeed, supposing that
$K^l_{T_1}\subsetneq K^r_{T_2}$ and  $ K^l_{T_2}\subsetneq
K^{r}_{T_1}$, by (1) we know that $T_1$ (resp. $T_2$) is a
cut-triangle in $\hat{K}^r_{T_2}$ (resp. $\hat{K}^r_{T_1}$), and
we define $K_1=\hat{K}^l_{T_1}$, $K_3=\hat{K}^l_{T_2}$ and  $K_2$
as the ${\mathcal M}$-complex that, together with
$\hat{K}^l_{T_1}$, makes the decomposition of  the ${\mathcal
M}$-complex $\hat{K}^r_{T_2}$ induced by $T_1$.
\end{proof}
 \end{lem}
  \begin{definition}
Two cut-triangles $T_1=\{e^1_0,e^1_1,e^1_2\}\subseteq K_1$ and
$T_2=\{e^2_0,e^2_1,e^2_2\}\subseteq K_1$   of an ${\mathcal
M}$-complex $K$ are {\em imbricated} if they are not disjoint.
 \end{definition}

\begin{example} We give some examples of imbricated and disjoint cut-triangles.
\begin{enumerate}
\item Let $T_2$ be a cut-triangle. If $T_1$ is a cut-triangle
distinct from $T_2$  such that whenever two vertices $A,B$ of
$T_1$ are vertices of $T_2$, then also the edge of $T_1$
connecting $A,B$ is an edge of $T_2$, then $T_1$ is disjoint from
$T_2$. \item The following ${\mathcal M}$-complex contains
disjoint cut-triangles, $T_1=\{2,3,4\}$ and $T_2=\{1,5,6\}$, that
share two vertices (and no edge):
\begin{center}
\resizebox{3.5cm}{!}{\begin{tikzpicture} \draw[line width=0.18em]
(0,0) circle (3cm); \node (5) at (-2.8,-1.5) {$5$}; \node (6) at
(2.8,-1.5) {$6$}; \node (5) at (-2.8,1.5) {$3$}; \node (6) at
(2.8,1.5) {$4$}; \node(E) at (0,1.75)
[circle,draw=black,fill=black,minimum size=4,inner
sep=0,label={[xshift=0.3cm,yshift=-0.1cm]{\large $e$}}]  {};
\node(F) at (0,-1.75) [circle,draw=black,fill=black,minimum
size=4,inner sep=0,label={[xshift=0.3cm,yshift=-0.6cm]{\large
$f$}}]  {}; \node(A) at (0,3)
[circle,draw=black,fill=black,minimum size=4,inner
sep=0,label={[xshift=0cm,yshift=0.0cm]{\large $c$}}]  {}; \node(B)
at (0,-3) [circle,draw=black,fill=black,minimum size=4,inner
sep=0,label={[xshift=0cm,yshift=-0.7cm]{\large $d$}}]  {};
\node(C) at (3,0) [circle,draw=black,fill=black,minimum
size=4,inner sep=0,label={[xshift=0.25cm,yshift=-0.3cm]{\large
$b$}}]  {}; \node(D) at (-3,0)
[circle,draw=black,fill=black,minimum size=4,inner
sep=0,label={[xshift=-0.25cm,yshift=-0.3cm]{\large $a$}}]  {};
\draw[line width=0.18em,draw=teal] (D)to[out=-25,in=205] (C)
node[midway,below,xshift=0cm,yshift=-0.75cm] {$1$};
\draw[dashed,line width=0.18em,draw=qqwuqq] (D)to[out=20,in=160]
(C) node[midway,above,xshift=0cm,yshift=0.65cm] {$2$};
\draw[dashed,line width=0.18em,draw=red] (E)-- (A);
\draw[dashed,line width=0.18em,draw=qqccqq] (E)to[out=190,in=45]
(D); \draw[dashed,line width=0.18em,draw=ffxfqq]
(E)to[out=-10,in=135]  (C); \draw[line width=0.18em,draw=qqzzff]
(F)-- (B); \draw[line width=0.18em,draw=qqqqcc]
(F)to[out=170,in=-45]  (D); \draw[line width=0.18em,draw=ffwwzz]
(F)to[out=10,in=225]  (C);
\end{tikzpicture}}
\end{center}
\item This is to be contrasted with Example \ref{counterexample},
where we also have an ${\mathcal M}$-complex that contains two
cut-triangles  sharing two vertices and no edge, which are,
however, imbricated.
\end{enumerate}
\end{example}
Our main theorem pinpoints the presentation of the Menelaus cyclic
operad  investigated in this section.

\begin{thm} The Menelaus cyclic operad is the quotient   of the free cyclic operad generated by the irreducible ${\mathcal M}$-complexes, under the equivalence relation generated by the equalities of the form ${\mathcal T}_1 \,{_u\circ_v}\, {\mathcal T}_2={\mathcal T}'_1 \,{_{u'}\circ_{v'}}\, {\mathcal T}'_2$, for each quadruple $({\mathcal T}_1,{\mathcal T}_2,{\mathcal T}'_1,{\mathcal T}'_2)$ of unrooted trees and   quadruple $(u,v,u'v')$ of leaves, such that both hand sides are well formed and evaluate, up to $<\!\!2$-isomorphism, to the same ${\mathcal M}$-complex $K$, in which the cut-triangles $T$ and $T'$, associated with the pairs $(u,v)$ and $(u',v')$ by Proposition \ref{theprop}, are imbricated.
\end{thm}
\noindent (The reader may want to read this statement with Example
\ref{counterexample} in mind, taking the single-node trees
$K,L,U,V$ and $\gamma,\delta',\varphi,\psi$ as quadruples.)
\begin{proof} Given a reducible ${\mathcal M}$-complex  $K$, the goal is to prove that any two decompositions of $K$ are equal modulo the equalities of the presentation given in the statement.  We proceed by induction on the number of $2$-cells of $K$. Using the notation of the statement, the two decompositions may be written as ${\mathcal T}_1 \,{_u\circ_v}\, {\mathcal T}_2$ and ${\mathcal T}'_1 \,{_{u'}\circ_{v'}}\, {\mathcal T}'_2$, with underlying cut-triangles $T$ and $T'$.
If $T$ and $T'$ are imbricated, the claim is ensured by the
corresponding imposed relation. If $T$ and $T'$ are disjoint, the
claim follows by Lemma \ref{nesting} and the associativity of
composition in the free cyclic operad generated  by the
irreducible ${\mathcal M}$-complexes.
\end{proof}

\begin{center}\textmd{\textbf{Acknowledgements} }
\end{center}
\medskip
Jovana Obradovi\' c gratefully acknowledges the financial support
of the Praemium Academiae of M.\ Markl and RVO:67985840. This work was supported by the Serbian Ministry of Education, Science and Technological Development through Mathematical Institute of the Serbian
Academy of Sciences and Arts. The authors are greatly indebted to the anonymous referees and Editor for careful reading and valuable comments and remarks, and for their suggestions of possible new developments towards enhancing our proof system and tackling further completeness and complexity issues. The authors are thankful  to \v Sejla Dautovi\' c and Filip Jevti\' c for very useful suggestions on how to improve the presentation of our work.

\end{document}